\documentclass[leqno,a4paper,11pt]{article}
\usepackage{amsmath,amsthm,amssymb}
\usepackage[utf8]{inputenc}
\usepackage[T1]{fontenc}

  \usepackage{hyperref}
 \usepackage[svgnames,x11names]{xcolor}

 \hypersetup{
   colorlinks, linkcolor={Chocolate4},
   citecolor={Chocolate4}, urlcolor={DarkOliveGreen4}
  }

\usepackage{enumerate}
\usepackage{bbm}
\usepackage{dsfont}
\usepackage{todonotes}
\usepackage{mathrsfs}
\usepackage{mathabx}

\usepackage{nicefrac}
\usepackage{tikz}
\usetikzlibrary{matrix}
\usepackage[normalem]{ulem}

\usepackage{marginnote}

\usepackage{url}
\makeatletter
\g@addto@macro{\UrlBreaks}{\UrlOrds}
\makeatother

\usepackage[sort,numbers]{natbib}

\providecommand{\noopsort}[1]{} 

\def\qed{\unskip\quad \hbox{\vrule\vbox
to 6pt {\hrule width 4pt\vfill\hrule}\vrule} }

\newcommand{\bez}{\nopagebreak\hspace*{\fill}
 \nolinebreak$\qed$\vspace{5mm}\par}

\newtheorem{Th}{Theorem}[section]
\newtheorem{Prop}[Th]{Proposition}
\newtheorem{Lemma}[Th]{Lemma}
\newtheorem{Conj}{Conjecture}
\newtheorem{Cor}[Th]{Corollary}

\newtheorem{Thx}{Theorem}
\newtheorem{Corx}[Thx]{Corollary}

\newtheorem{Corxx}{Corollary}

\theoremstyle{definition}
\newtheorem{Remark}[Th]{Remark}
\newtheorem{Def}{Definition}[section]

\newtheorem{Question}{Question}

\newcommand{\beq}{\begin{equation}}
\newcommand{\eeq}{\end{equation}}

\def\scalar(#1,#2){(#1\mid#2)}

\renewcommand{\hat}{\widehat}

\newcommand{\eps}{\varepsilon}

\newcommand{\raz}{\mathds{1}}

\newcommand{\ca}{{\cal A}}
\newcommand{\cb}{{\cal B}}
\newcommand{\cc}{{\cal C}}
\newcommand{\cd}{{\cal D}}

\newcommand{\cf}{{\mathscr{F}}}

\newcommand{\xbm}{(X,{\cal B},\mu)}
\newcommand{\zdr}{(Z,{\cal D},\rho)}
\newcommand{\ycn}{(Y,{\cal C},\nu)}
\newcommand{\zdk}{(Z,{\cal D},\kappa)}
\newcommand{\ot}{\otimes}
\newcommand{\ov}{\overline}

\newcommand{\bs}{\mathbb{S}}

\newcommand{\Q}{\mathbb{Q}}
\newcommand{\R}{{\mathbb{R}}}
\newcommand{\T}{{\mathbb{T}}}
\newcommand{\C}{{\mathbb{C}}}
\newcommand{\Z}{{\mathbb{Z}}}
\newcommand{\N}{{\mathbb{N}}}
\newcommand{\E}{{\mathbb{E}}}
\newcommand{\EE}{{\mathbb{E}}}
\newcommand{\PP}{{\mathbb{P}}}
\newcommand{\D}{{\mathbb{D}}}
\newcommand{\vep}{\varepsilon}

\newcommand{\K}{{\cal K}}
\newcommand{\mob}{\boldsymbol{\mu}}
\newcommand{\lio}{\boldsymbol{\lambda}}

\newcommand{\tend}[3][]{\xrightarrow[#2\to#3]{#1}}

\newcommand{\bfu}{\boldsymbol{u}}

\newcommand{\bfv}{\boldsymbol{v}}
\newcommand{\bfw}{\boldsymbol{w}}

\newcommand{\bp}{{\bar P}}
\newcommand{\bq}{{\bar Q}}

\newcommand{\diam}{\mathop{\mathrm{diam}}}
\renewcommand{\P}{\mathscr{P}}
\renewcommand{\Q}{\mathscr{Q}}
\newcommand{\ind}[1]{\mathds{1}_{#1}}

\begin{document}

\title{On arithmetic functions orthogonal to deterministic sequences}
\author{Adam Kanigowski\and Joanna Ku\l aga-Przymus \and Mariusz Lema\'nczyk \and Thierry de la Rue}

\maketitle

\begin{abstract}
We prove Veech's conjecture on the equivalence of Sarnak's conjecture on M\"obius orthogonality with a Kolmogorov type property of Furstenberg systems of the M\"obius function. This yields a combinatorial condition on the M\"obius function itself which is equivalent to Sarnak's conjecture. As a matter of fact, our arguments remain valid in a larger context: we characterize all bounded arithmetic functions orthogonal to all topological systems whose all ergodic measures yield systems from a fixed characteristic class (zero entropy class is an example of such a characteristic class) with the characterization persisting in the logarithmic setup. As a corollary, we obtain that the logarithmic Sarnak's conjecture holds if and only if the logarithmic M\"obius orthogonality is satisfied for all dynamical systems whose ergodic measures yield nilsystems.

\end{abstract}

\tableofcontents

\section{Introduction}

All transformations in this paper are assumed to be invertible. A topological dynamical system is a pair $(X,T)$, where $T$ is a homeomorphism of a compact metric space $X$. A measure-theoretic dynamical system is a system of the form $(Z,\cb(Z),\kappa,R)$ where $R$ is an automorphism (invertible measure-preserving transformation) of a standard Borel probability space $(Z,\cb(Z),\kappa)$.

\paragraph{Sarnak's conjecture} Given a topological dynamical system $(X,T)$ and a bounded arithmetic function $\bfu\colon\N\to\C$, we consider the corresponding problem of orthogonality:
\beq\label{ort1}
\lim_{N\to\infty}\frac1N\sum_{n\leq N}f(T^nx)\bfu(n)=0~\mbox{ for all $f\in C(X)$ and $x\in X$.}
\eeq
 Once
\eqref{ort1} holds for $(X,T)$, we say that the system $(X,T)$ {\em satisfies the Sarnak property with respect to $\bfu$} and write $\bfu\perp(X,T)$.
When a class $\mathscr{C}$ of topological dynamical systems is given, and $\bfu\perp(X,T)$ for each $(X,T)\in\mathscr{C}$ then we write $\bfu\perp\mathscr{C}$.

The main motivation to study this orthogonality problem is Sarnak's conjecture \cite{Sa} on M\"obius orthogonality in which $\bfu$ is the M\"obius function $\mob$ (or, equivalently \cite{Fe-Ku-Le}, the Liouville function $\lio$) and $\mathscr{C}$  is the class $\mathscr{C}_{\rm ZE}$ of zero topological entropy dynamical systems.
Sequences of the form $(f(T^nx))$ with $(X,T)$ running over $\mathscr{C}_{\rm ZE}$, $f\in C(X)$ and $x\in X$
 are often called {\em deterministic} sequences. Focusing on a special $\bfu$ in \eqref{ort1}, say, being multiplicative, is important if we count on applications in number theory -- the main motivation of Sarnak himself for the M\"obius orthogonality conjecture was to ``attack'' the celebrated Chowla conjecture on auto-correlations of the M\"obius function dynamically (indeed the Chowla conjecture implies Sarnak's conjecture \cite{Sa}, see also \cite{Ta,Ab-Ku-Le-Ru1}).\footnote{On a potential equivalence of the Chowla and Sarnak's conjectures see \cite{Sa,Ta,Ta2,Ta3,Go-Kw-Le} and a resum\'e of that in the survey articles \cite{Fe-Ku-Le,Ku-Le}.} However, if we aim at providing an {\bf internal} characterization of $\bfu$ being orthogonal to $\mathscr{C}_{\rm ZE}$, then  dropping the assumption of multiplicativity of $\bfu$ seems to be  reasonable.
Let us give an argument for that.
First, note that the class of deterministic sequences  equipped with the coordinatewise multiplication and addition is a ring
(this is an easy consequence of the fact that the zero entropy class is closed under taking joinings and factors).
Moreover, the class of bounded sequences $\bfu$ orthogonal to $\mathscr{C}_{\rm ZE}$ is a module over this ring. In other words, even if our ``starting'' $\bfu$ displays some additional arithmetic property (like multiplicativity), the characterization which we aim at must still hold for $\bfu\cdot\bfv$, where $\bfv$ is an {\bf arbitrary} deterministic sequence.\footnote{In what follows, we replace $\mathscr{C}_{\rm ZE}$ by $\mathscr{C}_{\cf}$, where $\cf$ is a general {\bf characteristic} class. Note that the argument $\bfu\cdot\bfv\perp \mathscr{C}_{\cf}$ whenever $\bfu\perp \mathscr{C}_{\cf}$ and $\bfv$ is an arbitrary $\cf$-sequence (cf.\ Def.~\ref{def:charK}) persists.}
Of course, properties like the multiplicativity of $\bfu\cdot\bfv$ changes dramatically if $\bfv$ is arbitrary, while the characterization we are looking for has to be stable under such a change of~$\bfu$.

\paragraph{Visible measures and Furstenberg systems} Let us now briefly discuss the matter of the topological and measure-theoretic aspects of choosing a class $\mathscr{C}$. Recall that $M(X)$  stands for the set of probability (Borel) measures on $X$ and $M(X,T)$ is the (always non-empty) subset of $T$-invariant measures. Both spaces are compact in the weak-$\ast$-topology. There is a third natural  subspace $V(X,T)\subset M(X,T)$ which is the set of ($T$-invariant) {\em visible} measures, i.e.\ measures possessing a quasi-generic point. Formally, $\nu\in V(X,T)$ if, for some point $x\in X$ and some increasing sequence $(N_{\ell})$, we have
$$\lim_{\ell\to\infty}\frac1{N_{\ell}}\sum_{n\leq N_{\ell}}\delta_{T^nx}=\nu.\footnote{Not all invariant measures, even in a transitive subshift, have to be visible: let $X_y\subset\{0,1,2\}^{\N}$ be given by $y=0^{k_1}1^{k_1}2^{k_1}1^{k_1} 0^{k_2}1^{k_2}2^{k_2}1^{k_2}\dots$  with $k_1<k_2<\dots$ Let
 $\nu=\frac12\delta_{\underline{0}}+\frac12\delta_{\underline{2}}$ (two fixed points). If $A:=\{x\in X_y: x_0=0\text{ or } x_0=2\}$, then $\nu(A)=1$ but no window in $y$ has the property that the frequency of 0 jointly with 2 on it is close to the value~1.}
$$
We also write $\nu\in V_T(x)$ and say that $x$ {\it is generic for} $\nu$ {\it along} $(N_{\ell})$. Note that the set $V(X,T)$ of visible measures contains the set $M^e(X,T)$ of all ergodic measures,
but it may be strictly smaller than $M(X,T)$. To illustrate this, consider the automorphism $(x,y)\mapsto (x,x+y)$ of $\T^2$ in which each point is generic (i.e.\ generic along the whole sequence of natural numbers) and the corresponding measure-theoretic system is either a rational or an irrational rotation on the circle.
Hence, each visible measure gives rise to a system with discrete spectrum, while there are many other invariant measures which yield partly continuous spectrum.

Given a bounded $\bfu\colon\N\to\C$, $|\bfu|\leq L$, we first extend it to $\Z$ by setting $\bfu(-n)=\bfu(n)$ for $n\in\N$, $\bfu(0)\in\D_L:=\{z\in\C: |z|\leq L\}$, and then take the closure
$X_{\bfu}:=\overline{\{S^k\bfu:k\in\Z\}}$ of the orbit of $\bfu\in\D_L^{\Z}$ under the left shift $S\colon\D_L^{\Z}\to\D_L^{\Z}$, $S((v_k)_{k\in\Z})=(v_{k+1})_{k\in\Z}$ (recall that we only consider invertible systems).
Hence, $(X_{\bfu},S)$ is the {\em subshift} determined by $\bfu$. By the very definition, $\bfu\in X_{\bfu}$. Each measure $\kappa\in V_S(\bfu)$ yields a measure-theoretic system $(X_{\bfu},\mathcal{B}(X_u),\kappa,S)$ called a {\em Furstenberg system} of $\bfu$ (by some abuse of vocabulary, we may call $\kappa$ itself a Furstenberg system of $\bfu$). Fix $x\in X$ and consider the sequence $(\frac1N\sum_{n\leq N}\delta_{(T^nx,S^n\bfu)})$. By compactness of $M(X\times X_{\bfu})$, we can select an increasing sequence $(N_{\ell})$ so that
$$
\frac1{N_{\ell}}\sum_{n\leq N_{\ell}}\delta_{(T^nx,S^n\bfu)}\to\rho,
$$
where $\rho\in V_{T\times S}((x,\bfu)) \subset V(X\times X_{\bfu},T\times S)$. The projections of $\rho$ on $X$ and $X_{\bfu}$ are denoted by $\nu$ and $\kappa$, respectively,
so that $\rho$ is a joining of $\nu$ and $\kappa$. Since $\rho$ is a visible measure, so are $\nu$ and $\kappa$:
\beq\label{ort}
\nu\in V_T(x),
\eeq
and $\kappa$ yields a Furstenberg system of $\bfu$.
Now, setting
\begin{equation}
 \label{eq:defpi0}
 \pi_0: \begin{cases}
         &X_{\bfu}\to\C\\
         &z=(z_n)_{n\in\Z}\mapsto z_0,
        \end{cases}
\end{equation}
we have
\begin{align}
\begin{split}\label{three}
\lim_{\ell\to\infty}\frac1{N_{\ell}}\sum_{n\leq N_{\ell}}f(T^nx)\bfu(n)&=\lim_{\ell\to\infty}\int f\ot \pi_0\,d\left(\frac1{N_{\ell}}\sum_{n\leq N_{\ell}} \delta_{(T^nx,S^n\bfu)}\right)\\
&=\int f\ot \pi_0\,d\rho=\int\EE^\rho(f|X_{\bfu})\cdot\pi_0\,d\kappa.
\end{split}
\end{align}
So, the fact that the limit is~0 in \eqref{ort1}, i.e.\  $\bfu\perp(X,T)$, depends on the  joinings of measures from $V(X,T)$ with those from $V_S(\bfu)$ -- the invariant measures from $M(X,T)\setminus V(X,T)$ are irrelevant in this context. More precisely, what will matter is  the ``geometric position'' of the {\bf single} continuous function $\pi_0$ in the $L^2$-space of all Furstenberg systems $\kappa\in V_S(\bfu)$ for such joinings. Namely, inside the $L^2(\rho)$-space, we want $\pi_0$ to be orthogonal to $L^2(\nu)$ (in condition~\eqref{three} we use {\bf all} continuous functions $f$ on $X$ and we take into account all $\nu\in V(X,T)$).
Of course, whether this orthogonality can be established without referring to the (visible) joinings $\rho$ (and remaining at the level of $\kappa\in V_S(\bfu)$) is another question. A kind of surprise is that the answer to this question will turn out to be positive for $\mathscr{C}_{\rm ZE}$ and some other characteristic classes, see Theorem~\ref{tB}, below.

\paragraph{Characteristic classes and Veech's conjecture} With the above in mind, a knowledge  of some fundamental facts on joinings, especially on disjointness in the sense of Furstenberg in ergodic theory, strongly suggests to consider only the situation when the measures $\nu$ appearing above yield measure-theoretic systems $(X,\mathcal{B}(X),\nu,T)$ belonging  to one of so called \emph{characteristic classes of measure-theoretic dynamical systems}.
\begin{Def}\label{def:charK}
A class $\cf$ of measure-theoretic dynamical systems  is called {\em characteristic} if it is closed under taking isomorphisms, factors and (countable) joinings.
Given a characteristic class $\cf$, by an \emph{$\cf$-factor} of a measure-theoretic dynamical system $(Z,\mathcal{D},\kappa,R)$ we call any factor sub-$\sigma$-algebra of $\mathcal{D}$ on which the action of $R$ yields a system in the class $\cf$.
We denote by $\mathscr{C}_{\cf}$ the class of topological systems $(X,T)$ for which we have $(X,\mathcal{B}(X),\nu,T)\in\cf$ for each $\nu\in V(X,T)$. The sequences of the form $(f(T^nx))_{n\in\Z}$ with $f\in C(X)$, $x\in X$ and $(X,T)\in\mathscr{C}_{\cf}$ are called $\cf$-sequences.
\end{Def}

It follows from the above definition that:
\begin{enumerate}[(a)]
\item every measure-theoretic dynamical system $(Z,\mathcal{D},\kappa,R)$ has a \emph{largest $\cf$-factor} (in the sense of inclusion of sub-$\sigma$-algebras), which we denote by $\mathcal{D}_{\cf}$,
\item any joining of $(Z,\mathcal{D},\kappa,R)$  with a system from $\cf$ is uniquely determined by its restriction to the joining of the largest $\cf$-factor $\mathcal{D}_{\cf}$ of $R$ with the given system from $\cf$
\end{enumerate}
 (see Section~\ref{s:characteristic} for details). The class ZE of zero (Kolmogorov-Sinai) entropy is of course characteristic. The largest zero entropy factor $\mathcal{D}_{\rm ZE}$ of $(Z,\mathcal{D},\kappa,R)$ is called the {\em Pinsker factor} and
is denoted by $\Pi(R)$ or $\Pi(\kappa)$. Note that (via the variational principle) the family $\mathscr{C}_{\rm ZE}$ is precisely the family of all topological systems whose all invariant measures yield systems in ZE.
Returning to Sarnak's conjecture, in \cite{Ve}, Veech proves the following result:
\begin{Th}[\cite{Ve}] \label{t:veech1} If
\beq\label{w-ekveecha}
\pi_0\perp L^2(\Pi(\kappa))\text{ for each Furstenberg system }\kappa\in V_S(\mob)
\eeq
then Sarnak's conjecture holds, i.e.\ $\mob\perp \mathscr{C}_{\rm ZE}$.\footnote{An ergodic proof (which goes back to a suggestion of Sarnak in \cite{Sa}) of this result is implicit in \cite{Ab-Ku-Le-Ru1}, where the implication ``Chowla conjecture $\Rightarrow$ Sarnak's conjecture'' has been proved using joinings. Veech cites \cite{Ab-Ku-Le-Ru1} (which was on arXiv two years before Veech's lecture notes \cite{Ve} appeared) but instead he gives his own (slightly complicated) proof  using the concept of quasi-factors of Glasner and Weiss.}\end{Th}

Then he formulates the following conjecture (Conjecture 24.3   page 88 in~\cite{Ve}):
\begin{Conj}[Veech's conjecture]
Condition~\eqref{w-ekveecha} is equivalent to Sarnak's conjecture.
\end{Conj}

One of motivations for the present work was to prove the above conjecture.  Let us first formulate~\eqref{w-ekveecha} in full generality.
\begin{Def}
Given a characteristic class $\cf$   we say that a (bounded) arithmetic function $\bfu\colon\N\to\C$ satisfies the {\em Veech condition} with respect to $\cf$ if
\beq\label{veech3}
\pi_0\perp L^2((\mathcal{B}(X_{\bfu}),\kappa)_{\cf})\text{ for each Furstenberg system }\kappa\in V_S(\bfu).
\eeq
\end{Def}

Given a characteristic class $\cf$, we denote by $\cf_{\rm ec}$ the class of those automorphisms $R$ of $\zdk$ such that a.e.\ ergodic component of $\kappa$ yields a system in $\cf$.
Then (see Section~\ref{s:onec}), $\cf_{\rm ec}$ is also a characteristic class,

\beq\label{ec1}
\mathscr{C}_{\cf}\subset \mathscr{C}_{\cf_{\rm ec}}\eeq
and
\beq\label{ec2}
\mathscr{C}_{\cf_{\rm ec}}=\{(X,T):\:(X,\mathcal{B}(X),\nu,T)\in \cf\text{ for each {\bf ergodic} } \nu\in M(X,T)\}.
\eeq

Our main result is the following:

\begin{Thx} \label{tB} Assume that $\cf$ is a characteristic class.
Let $\bfu\colon\N\to\C$ be a bounded arithmetic function. Then $\bfu\perp \mathscr{C}_{\cf_{\rm ec}}$ if and only if $\bfu$ satisfies the Veech condition~\eqref{veech3} (with respect to $\cf_{\rm ec}$).

\end{Thx}

\begin{Remark}
\label{r:ThmAsubsequence}
 A subsequence version of this result also holds: If $(N_\ell)$ is an increasing sequence of integers, then $\bfu$ is $(N_\ell)$-orthogonal to
$\mathscr{C}_{\mathscr{F}_{\rm ec}}$  (i.e.\ $\frac1{N_\ell}\sum_{n\leq N_\ell}f(T^nx)\bfu(n)\to 0$ for each $(X,T)\in \mathscr{C}_{\mathscr{F}_{\rm ec}}$ and all $f\in C(X)$, $x\in X$) if and only if the Veech condition~\eqref{veech3} holds for each measure $\kappa$ for which $\bfu$ is generic along some subsequence of $(N_\ell)$. The reason for the validity of this ``local'' version is that all tools used in the proof of Theorem~\ref{tB} work well on subsequences, see also Remarks~\ref{r:strongMOMOsubsequence} and~\ref{r:mo1subsequence} to cope with the strong $\bfu$-MOMO property along subsequences.

See \cite{Chi} for the validity of the alternative: {\em either there are no Siegel zeros or there exists a (universal) subsequence along which Sarnak's conjecture holds} and \cite{Go-Le-Ru} for a density version of Sarnak's conjecture.
\end{Remark}

Theorem~\ref{tB} for $\bfu=\mob$ and $\cf={\rm ZE}$ (note that ${\rm ZE}_{\rm ec}={\rm ZE}$) has the following consequence:
\begin{Cor}
Veech's conjecture holds. Moreover, if Sarnak's conjecture holds then {\bf all} Furstenberg systems $\kappa\in V_S(\mob)$ have positive entropy.
\end{Cor}

We also generalize Veech's theorem (Theorem \ref{t:veech1}) to the setting of characteristic classes:

\begin{Thx}\label{tA}
Assume that $\cf$ is a characteristic class. Let $\bfu\colon\N\to\C$ be a bounded arithmetic function. If $\bfu$ satisfies the Veech condition~\eqref{veech3} with respect to $\cf$ then $\bfu\perp\mathscr{C}_{\cf}$.
\end{Thx}

\medskip

Let us now briefly discuss the methods involved in the proofs of the above theorems. The proof of Theorem~\ref{tA} is provided in Section~\ref{sec:prooftA}. It is a straightforward application of the above Property~(b) of the largest $\cf$-factor (see Proposition~\ref{p:sated} for more details). The proof of Theorem~\ref{tB} which we provide in Section~\ref{sec:prooftB} relies on:

\begin{itemize}
\item
so called Hansel models \cite{Ha} of possibly highly non-ergodic measure-preserving systems, 
\item
our version  of a lifting lemma  by Conze, Downarowicz, Serafin \cite{Co-Do-Se} on the existence of generic points for joinings, valid in a general context and using the concept of quasi-generic sequences (Proposition~\ref{p:momogen}),
\item general joining techniques.
\end{itemize}

Our first proof of Theorem~\ref{tB} was based on a different lifting lemma by Bergelson, Downarowicz and Vandehey (Theorem 5.16 in \cite{Be-Do-Va}) and a recent result by Downarowicz and Weiss showing the existence, for each zero-entropy measure-theoretic system, of a special Hansel model which is symbolic~\cite{Do-We}. The proof we finally chose to present here does not require the use of this symbolic model, as our lifting lemma works for all topological systems. Its additional advantage is that it is also valid in the context of logarithmic averages (see Section~\ref{a:logarithmic}). This makes all our results also true in the logarithmic set up: for example, the logarithmic Sarnak's conjecture (denoted as $\mob\perp_{\rm log}\mathscr{C}_{\rm ZE}$) is equivalent to the Veech condition for $\mob$ for all logarithmic Furstenberg systems $\kappa\in V_S^{\rm log}(\mob)$, see Corollary~\ref{c:SNC} for more.

\medskip

In Theorem~\ref{tB}, it is crucial that we deal with a characteristic class of the form $\cf_{\rm ec}$ since, by Proposition~\ref{p:mo1}, $\bfu\perp \mathscr{C}_{\cf_{\rm ec}}$ is equivalent to the so called \emph{strong $\bfu$-MOMO property} for systems in $\mathscr{C}_{\cf_{\rm ec}}$. This property, whose definition is recalled below, has been introduced in \cite{Ab-Ku-Le-Ru2,Go-Le-Ru}.  We leave as an open problem whether Theorem~\ref{tB} holds for an arbitrary characteristic class.

\begin{Def}
 \label{def:strongMOMO}
 A topological system $(X,T)$ satisfies the \emph{strong $\bfu$-MOMO property}, if
\beq\label{defmomo}\lim_{K\to\infty}\frac1{b_K}\sum_{k<K}\left\|\sum_{b_k\leq n<b_{k+1}}\bfu(n)f\circ T^n\right\|_{C(X)}=0\eeq
for each $f\in C(X)$ and each increasing sequence $(b_k)\subset\N$ such that $b_{k+1}-b_k\to\infty$.
\end{Def}

Clearly, the strong $\bfu$-MOMO property implies \eqref{ort1} uniformly in $x\in X$. The concept of strong $\bfu$-MOMO is formally stronger than the usual orthogonality. To see the difference between the usual orthogonality and strong MOMO, consider the system $(x,y)\mapsto(x,x+y)$ on $\T^2$ whose M\"obius orhogonality follows easily from the DDKBSZ criterion,\footnote{DDKBSZ stands for Daboussi-Delange-K\'atai-Bourgain-Sarnak-Ziegler \cite{Bo-Sa-Zi}, \cite{Ka}.} see e.g.~\cite{Fe-Ku-Le}) (in fact, the orthogonality holds even uniformly due to  the Davenport estimate \cite{Da}) , while the strong $\mob$-MOMO property (apply the definition to $f(x,y)=e^{2\pi iy}$) yields
\[
\frac1{b_K}\sum_{k<K}\sup_{x\in\T}\left|\sum_{b_k\leq n<b_{k+1}}\mob(n)e^{2\pi inx}\right|\to 0
\]
when $K\to\infty$ which is an open problem,\footnote{In fact, it is open whether a non-periodic, zero entropy, continuous, algebraic automorphism of $\T^2$ satisfies the strong $\mob$-MOMO property.} see Sections~\ref{s:FSMOMO} and~\ref{s:orthogonality} for more details.

\begin{Remark}
 \label{r:strongMOMOsubsequence}
 Note that in the definition of the strong $\bfu$-MOMO property, convergence~\eqref{defmomo} can be replaced by
 \begin{multline}
 \label{defmomobis}
 \lim_{N\to\infty}\frac1{N}\left(\sum_{k<K_{\!N}}\left\|\sum_{b_k\leq n<b_{k+1}}\bfu(n)f\circ T^n\right\|_{C(X)} \right. \\
\left. + \left\| \sum_{b_{K_{\!N}}\le n < N} \bfu(n)f\circ T^n \right\|_{C(X)}\right) =0,
 \end{multline}
 where $K_N:=\max\{k:b_k<N\}$. As a matter of fact, the definition is unchanged if we only restrict to sequences $(b_k)$ which further satisfy the condition
 \[
  \lim_{k\to\infty} \frac{b_{k+1}-b_k}{b_k} = 0.
 \]
 (If we have to consider a sequence $(b_k)$ which does not satisfy the above condition, we can always add more integers to the set $\{b_k:k\ge1\}$ so that this convergence holds for the new sequence. And the validity of~\eqref{defmomo} or~\eqref{defmomobis} for the new sequence is stronger than the same for the former sequence $(b_k)$.)

 Then it is easy to define also the strong $\bfu$-MOMO property along an increasing sequence $(N_\ell)$, by restricting convergence~\eqref{defmomobis} to the subsequence $(N_\ell)$.
\end{Remark}

\paragraph{Veech condition and combinatorics}  Given a characteristic class $\cf$, Theorems~\ref{tA} and~\ref{tB} determine a natural strategy to describe the arithmetic functions $\bfu$ orthogonal to all $\cf$-sequences. Namely, we need to describe the $\cf$-factors and understand the orthogonality to their $L^2$-space (i.e.\ the Veech condition), hoping that  this description can be expressed (for $\pi_0$ in $X_{\bfu}$ and a Furstenberg system $\kappa$ of $\bfu$) by some asymptotics of the integrals of continuous functions on $X_{\bfu}$. The final step would be to use the definition of a Furstenberg system to obtain a combinatorial condition on $\bfu$ itself. The first part of the strategy should be seen as an extension of the theory of characteristic factors $\mathcal{Z}_s(T)$ (given an automorphism $T$) and the Gowers-Host-Kra  (GHK in what follows) seminorms $\|\cdot\|_{u^s}$ for  $s\geq1$ \cite{Ho-Krbook}. In this perspective the Veech condition on $\pi_0$ is the counterpart of $\|\pi_0\|_{u^s}=0$ for each Furstenberg system $\kappa\in V_S(\bfu)$ (and has its combinatorial translation in terms of the GHK seminorm of $\bfu$). We will give more details on this shortly.

Let us discuss this strategy for the class ZE, see Section~\ref{s:wnioski} for details. Here, the characteristic factor of a measure-preserving system is the Pinsker factor.
The reader has certainly noticed that even though we study dynamical properties of Furstenberg systems, as a matter of fact, at the end we deal with a process $(\pi_0\circ S^n)_{n\in\Z}$, stationary with respect to $\kappa\in V_S(\bfu)$ (each such measure is invariant in the subshift $(X_{\bfu},S)$). Now, the Veech condition leads to the following concept.

\begin{Def} A centered stationary process $\underline{X}=(X_n)$ taking finitely many values is called a {\em Sarnak process} if $\EE(X_0\mid \sigma(X_{N},X_{N+1},\ldots))\to0$ in $L^2$ (or a.e.); equivalently $\EE(X_0\mid \Pi(\underline{X}))=0$, where $\Pi(\underline{X})$ stands for the tail $\sigma$-algebra.
\end{Def}
Understanding the structure of Sarnak processes seems to be a problem of an independent interest and it will be studied elsewhere.

Now, when $\bfu$ takes finitely many values, our  main result (Theorem~\ref{tB} applied to ZE)  can  be reformulated in the following manner:

\begin{Cor}\label{c:Sarnakpro}
Let $\bfu\colon\N\to\C$ be an arithmetic function taking finitely many values. Then  $\bfu\perp \mathscr{C}_{\rm ZE}$ if and only if all stationary processes $(\pi_0\circ S^n)_{n\in\Z}$ determined by $\kappa\in V_S(\bfu)$ are Sarnak.\end{Cor}

From the ergodic theory point of view we are close to the concept of the {\bf relative} Kolmogorov property (K-property) which is however perturbed by the fact that we need this property for a single function. But even though only $\pi_0$ is involved, the dynamical idea of the equivalence between the K-property and K-mixing (uniform mixing) works, and we can apply K-mixing of $\pi_0$ against the family of
functions depending on finitely many non-negative coordinates.
This  leads to studying the asymptotics of integrals of some continuous functions and finally gives the following combinatorial characterization of the orthogonality of $\bfu$ to all deterministic sequences.
In the following corollary, we use the fact that when $\bfu$ takes its values in a finite set, a subset $A\subset X_{\bfu}$ depends on finitely many non-negative coordinates if and only if there exists $\ell\ge1$ and a set $C$ of blocks of length $\ell$ appearing in $\bfu$ such that $\ind{A}(y)=\ind{(y(0),y(1),\ldots,y(\ell-1)) \in C}$.

\begin{Corx}\label{cA}
Let $\bfu\colon\N\to\C$ be an arithmetic function taking finitely many values.  Then $\bfu\perp \mathscr{C}_{\rm ZE}$ if and only if, for each subsequence $(N_k)$ defining a Furstenberg system of $\bfu$ and each subset $A\subset X_{\bfu}$ depending on finitely many non-negative coordinates, we have the cancellation phenomenon of the values of $\bfu$ uniformly along sufficiently large shifts of the set of visits of $\bfu$ in $A$: for each $\vep>0$, there exists $M\geq1$ such that for each $\ell\ge1$ and each set $C$ of blocks of length $\ell$, we have for each $m\geq M$
\begin{equation}\label{corCcond}
\lim_{k\to\infty}\left|\frac1{N_k}\sum_{n\leq N_k}\bfu(n)\underbrace{\raz_{(\bfu(m+n),\bfu(m+n+1),\ldots,\bfu(m+n+\ell-1))\in C}}_{\ind{A}(S^{m+n}\bfu)}\right|\leq \vep.
\end{equation}
\end{Corx}

Note that $M$ above depends on $(N_k)$ and $\vep$. This combinatorial condition looks more attractive if we assume that $\bfu$ is generic:

\setcounter{Corxx}{2}
\begin{Corxx}\label{cB}
Let $\bfu\colon\N\to\C$ be an arithmetic function taking finitely many values. If $\bfu$ is generic then $\bfu\perp\mathscr{C}_{\rm ZE}$ if and only if
$$
\lim_{m\to\infty}\lim_{N\to\infty}\left|\frac1{N}\sum_{n\leq N}\bfu(n)\raz_{(\bfu(m+n),\bfu(m+n+1),\ldots,\bfu(m+n+\ell-1))\in C}\right|=0$$
uniformly in $\ell\geq1$ and in $C$, a set of blocks of length $\ell$.
\end{Corxx}

\begin{Remark}\label{r:sarnakcomb}
Note however that the above two corollaries do not say much if the (clopen) sets $C$ are already of small measures (e.g.,\ in the most interesting case of blocks of large length). In fact, condition~\eqref{corCcond} in Corollary~\ref{cA} is equivalent to the following: for each subsequence $(N_k)$ defining a Furstenberg system of $\bfu$ we have the {\bf conditional} cancellation phenomenon of the values of $\bfu$ uniformly along sufficiently large shifts of the set of visits of $\bfu$ to ``typical'' blocks. More precisely, for each $\vep>0$ there exists $M\geq1$ such that
$$
\lim_{k\to\infty}\frac{\left|\sum_{n\leq N_k}\bfu(n)\raz_{(\bfu(m+n),\bfu(m+n+1),\ldots,\bfu(m+n+\ell-1))=C}\right|}{\sum_{n\leq N_k}\raz_{(\bfu(n),\ldots,\bfu(n+\ell-1))=C}}\leq \vep$$
for all $m\geq M$, all $\ell$ sufficiently large and blocks $C$ of length $\ell$ forming a family of measure $>1-\vep$.

If, additionally, $\bfu$ is generic, then the above condition reduces to
$$
\lim_{m\to\infty}\lim_{N\to\infty}\frac{\left|\sum_{n\leq N}\bfu(n)\raz_{(\bfu(m+n),\bfu(m+n+1),\ldots,\bfu(m+n+\ell-1))= C}\right|}{\sum_{n\leq N}\raz_{\bfu(m+n),\ldots,\bfu(m+n+\ell-1)=C}}=0$$
uniformly in $m$, for ``good'' blocks of length $\ell$ sufficiently large.\end{Remark}

\begin{Remark} Note the basic difference between the sums
$$
\sum_{n\leq N_k}\bfu(n)\raz_{(\bfu(n),\bfu(n+1),\ldots,\bfu(n+\ell-1))=C}
=C[0]\cdot\sum_{n\leq N_k}\raz_{(\bfu(n),\bfu(n+1),\ldots,\bfu(n+\ell-1))=C}
$$
and
$$
\sum_{n\leq N_k}\bfu(n)\raz_{(\bfu(m+n),\bfu(m+n+1),\ldots,\bfu(m+n+\ell-1))= C},$$
namely, the first one does not display {\bf any} cancellation (that is, along the return times to a fixed block, we have no cancellation) and in the second one cancellations are possible and the fact that along further and further shifts of the set of return times we observe more and  more cancellations, characterizes the Sarnak property.
\end{Remark}

\paragraph{Veech's and Sarnak's conjectures for other characteristic classes}
The family of all characteristic classes is enormous, see Section~\ref{s:characteristic} for natural examples. Here, let us just notice that the discrete spectrum automorphisms form a characteristic class and it contains uncountably many characteristic subclasses whose pairwise intersections are equal to the class of all identities (indeed, discrete spectrum automorphisms whose group of eigenvalues is contained in $\Z\alpha$, with $\alpha$ irrational, is a characteristic class). Moreover, there are the largest proper and the smallest non-trivial characteristic classes. Indeed,
although our study of the zero entropy class was originally motivated by Sarnak's conjecture, yet, ZE plays a special role since it is  the largest (proper) characteristic class. In fact, we have
\begin{equation}\label{zawierania}
\{\ast\}\subset {\rm ID}\subset \cf\subset {\rm ZE}\subset{\rm ALL}
\end{equation}
for each characteristic class $\cf$, where ${\rm ID}$ stands for the (characteristic) class of identities (of all standard Borel probability spaces) and {\rm ALL} stands for the (characteristic) class of all systems. Note that
$$
{\rm ID}=\{\ast\}_{\rm ec}\text{ and } {\rm ZE}={\rm ZE}_{\rm ec}.
$$

Clearly, the topological class $\mathscr{C}_{\rm ALL}$ consists of all topological systems and the only $\bfu$ orthogonal to all of them is $\bfu(n)=0$ on a subset of $n$ of full density  which is compatible with the Veech condition (which in this setting means that $\pi_0$ equals (a.e.) zero for each Furstenberg system).

The topological class $\mathscr{C}_{\{\ast\}}$ consists of topological systems whose all visible measures are given by fixed points. The reader can check that the Veech condition here is just $\int_{X_{\bfu}}\pi_0\,d\kappa=0$ for each $\kappa\in V_S(\bfu)$. The combinatorial condition~\eqref{corCcond} from Corollary~\ref{cA} (equivalent to the Veech condition) in this setting reduces to $\lim_{N\to\infty}\frac1N\sum_{n\leq N}\bfu(n)=0$. Clearly, in this setting, the latter is nothing but the Sarnak condition.

It is not hard to see that the topological class $\mathscr{C}_{\rm ID}$ consists of topological systems whose only ergodic measures are Dirac measures at fixed points. The Veech condition here is $\pi_0\perp L^2(\mathcal{I}_\kappa)$ for each $\kappa\in V_S(\bfu)$, where $\mathcal{I}_\kappa$ stands for the $\sigma$-algebra of invariant sets. Finally, the counterparts of Corollaries~\ref{cA} and~\ref{cB} are the following.

\begin{Cor}\label{c:ID}
Let $\bfu\colon\N\to\C$ be a bounded arithmetic function.
Then $\bfu\perp \mathscr{C}_{\rm ID}$ if and only if for each $\vep>0$ and each subsequence $(N_k)$ defining a Furstenberg system of $\bfu$, there exists $H_0\geq1$ such that for each $H\geq H_0$,
\beq\label{ortmsvf}
\lim_{k\to\infty}\frac1{N_k}\sum_{n\leq N_k}\left|\frac1H\sum_{h\leq H}\bfu(n+h)\right|^2\leq\vep.
\eeq
\end{Cor}

Note that if, additionally, $\bfu$  is generic and it satisfies a non-quantitative version of the Matom\"aki-Radziwi\l\l \ \cite{Ma-Ra} convergence on a typical short interval:
$$\lim_{M, H\to\infty, H={\rm o}(M)}\frac1{M}\sum_{n\leq M}\left|\frac1H\sum_{h\leq H}\bfu(n+h)\right|^2=0$$
then   $\bfu\perp \mathscr{C}_{\rm ID}$.

\begin{Remark} It is also worth mentioning that the ID-sequences  are precisely  the {\em mean slowly varying functions} (see Proposition 5.1 in \cite{Go-Le-Ru1}), i.e.\ (bounded) arithmetic functions $\bfv\colon\N\to\C$ for which
$$
\lim_{N\to\infty}\frac1N\sum_{n<N}|\bfv(n+1)-\bfv(n)|=0.
$$
Therefore, sequences satisfying \eqref{ortmsvf} are precisely those orthogonal to all mean slowly varying functions.\end{Remark}

Notice that it follows from~\eqref{zawierania} that, for a non-trivial class $\cf$, the ``zero mean condition on a typical short interval'' \eqref{ortmsvf} is a necessary condition for $\bfu \perp \mathscr{C}_{\cf}$, whereas the condition given by Corollary~\ref{cA} is sufficient for $\bfu \perp \mathscr{C}_{\cf}$.

In Section~\ref{s:orthogonality}, we discuss the case $\cf={\rm DISP}_{\rm ec}$, where DISP stands for the (characteristic) class of discrete spectrum automorphisms. In view of~\eqref{ec2}, $\mathscr{C}_{{\rm DISP}_{\rm ec}}$ consists of homeomorphisms whose all {\bf ergodic} measures yield systems with discrete spectrum.  Let $\bfu\colon\N\to \C$ be bounded. For  an increasing sequence of integers $(N_k)$, we set
\begin{equation}\label{eq:1us}
\|\bfu\|^2_{u^1((N_k))}:=\lim_{H\to \infty} \frac{1}{H}\sum_{h\leq H} \Big(\lim_{k\to \infty}\frac1{N_k}\sum_{n\leq N_k}\bfu(n)\overline{\bfu(n+h)}\Big)
\end{equation}
and, for $s>1$,
\begin{equation}\label{eq:2us}
\|\bfu\|_{u^s((N_k))}^{2^{s+1}}:=\lim_{H\to \infty} \frac{1}{H}\sum_{h\leq H}\|\bfu(\cdot+h)\cdot \bfu(\cdot)\|_{u^s((N_k))}^{2^s},
\end{equation}
whenever all the above limits exist. If $N_k=k$, we set $\|\bfu\|_{u^s}:=\|\bfu\|_{u^s((N_k))}$.

\begin{Cor}\label{c:averagedCh'} Let $\bfu\colon\N\to\C$ be bounded.
Then
$\bfu\perp \mathscr{C}_{{\rm DISP}_{\rm ec}}$ if and only if $\|\bfu\|_{u^2((N_k))}=0$ for each subsequence $(N_k)$ along which $\bfu$ is generic. In particular, if $\bfu$ is generic then $\|\bfu\|_{u^2}=0$.
\end{Cor}

The main assertion of Corollary~\ref{c:averagedCh'} is equivalent to $\|\pi_0\|_{u^2(\kappa)}=0$ for each Furstenberg system $\kappa\in V_S(\bfu)$. The reason for the validity of this result is that given an automorphism $(Z,\cd,\rho,R)$, we have the equality
$$
\cd_{{\rm DISP}_{\rm ec}}=\mathcal{Z}_1(R,\rho)$$
which is a consequence of ${\rm DISP}_{\rm ec}={\rm NIL}_1$, where ${\rm NIL}_s$ stands for the class of automorphisms whose a.a.\ ergodic components are inverse limits of $s$-step nil-automorphisms (see Section~\ref{s:characteristic} for more details).

When we turn to the classes ${\rm NIL}_s(=({\rm NIL}_s)_{\rm ec})$ and return to the original Sarnak's conjecture (for $\bfu$) then clearly  $\pi_0\perp L^2(\Pi(\kappa))$ implies $\pi_0\perp L^2(\mathcal{Z}_s(\kappa))$ for each $s\geq1$. We hence obtain  one more necessary condition for $\bfu$ to be orthogonal to all deterministic sequences:

\begin{Cor}\label{deter} Let $\bfu:\N\to \C$ be bounded. If $\bfu\perp \mathscr{C}_{\rm ZE}$ then  $\|\bfu\|_{u^s((N_k))}=0$ for each $s\in\N$ (for each subsequence $(N_k)$  along which $\bfu$ is generic). 
\end{Cor}

In Section~\ref{s:distalandother}, we prove that the Sarnak property of $\bfu$ for the fundamental (in ergodic theory) class of distal automorphisms
is equivalent to the Veech property of $\bfu$. We leave as an open problem whether the Veech property can be expressed combinatorially in the distal case.

\paragraph{The logarithmic Sarnak's conjecture} As we have already noticed, our results also hold in the logarithmic case. In the corollary below we put together conditions which are equivalent to the logarithmic Sarnak's conjecture.\footnote{Our thanks go to N.\ Frantizkinakis who pointed out to us one of crucial equivalences: (iv) $\Leftrightarrow$ (v).}
\begin{Cor}\label{c:SNC} Let $\bfu=\mob$ or $\lio$. The following conditions are equivalent:\\
(i)  $\bfu\perp_{\rm log} \mathscr{C}_{\rm ZE}$ (i.e.\ zero entropy systems satisfy the logarithmic Sarnak property with respect to $\bfu$),\\
(ii) $\bfu \perp_{\rm log} \mathscr{C}_{\rm NIL_s}$ for each $s\geq1$,\\
(iii)  $\pi_0\perp L^2(\Pi(\kappa))$ for each $\kappa\in V_S^{\rm log}(\bfu)$ (i.e.\ $\bfu$ satisfies the Veech condition for each logarithmic Furstenberg system of $\bfu$),\\
(iv) $\pi_0\perp L^2(\mathcal{Z}_s(\kappa))$ for each $\kappa\in V_S^{\rm log}(\bfu)$ and $s\geq1$,\\
(v) $\bfu$ satisfies the logarithmic Chowla conjecture.
\end{Cor}
The equivalence of (ii) and (iv) is due to Theorem~\ref{tB} (in its logarithmic form), the equivalence of (iv) and (v) (based on the facts proved by Tao \cite{Ta2} for the equivalence of the logarithmic Sarnak's and Chowla's conjectures for the Liouville function) is formally proved  in \cite{Fr} (implicit in Section 2.7 therein). Other (needed) implications are standard.  Recall also that, by \cite{Go-Le-Ru}, (i) is equivalent to the logarithmic strong $\bfu$-MOMO property of all zero entropy systems.

Note that the equivalence of (i) and (ii) in Corollary~\ref{c:SNC}  yields immediately the following.
\begin{Cor}\label{c:logS+N} The logarithmic Sarnak's conjecture
holds if and only if for each $s\geq1$, $\mob$ is orthogonal to all systems whose all ergodic measures yield ${\rm NIL}_s$-systems.\end{Cor}

\begin{Remark} Yet, in \cite{Hu-Xu-Ye} another characterization of logarithmic Sarnak's conjecture is given, namely, the conjecture holds if and only if $\mob$ is orthogonal to all topological systems with polynomial mean complexity. It follows from \cite{Ka-Vi-We} that the topological (and hence measure theoretic) mean complexity of any zero entropy homogeneous system is polynomial. However, the polynomial mean complexity in \cite{Hu-Xu-Ye} is considered in the topological setting (for systems that need not be homogeneous), so it seems there is no direct relation between Corollary \ref{c:logS+N} and the main result in \cite{Hu-Xu-Ye}.
\end{Remark}

\begin{Remark} As Theorem~\ref{tB} is true in a larger context, also Corollary~\ref{c:SNC} can be formulated for more general {\bf multiplicative} functions bounded by~1, cf.\ Theorem~1.8 in \cite{Fr}.\end{Remark}

\paragraph{Averaged Chowla property}  The ``iff'' assertion of Theorem~\ref{tB}  cannot be applied to the class $\mathscr{C}_{\rm DISP}$. However, in Section~\ref{s:avCh}, we will show that the Sarnak and Veech conditions are equivalent in this setting for each bounded $\bfu\colon\N\to\C$ such that
\beq\label{avch}
\mbox{all circle rotations satisfy the strong $\bfu$-MOMO property}
\eeq
(strong $\bfu$-MOMO property has been defined in Definition~\ref{def:strongMOMO}).

\begin{Cor}\label{c:averagedCh}
Assume that $\bfu\colon\N\to\C$ is bounded by~1 and satisfies~\eqref{avch}. Then $\bfu$ satisfies the Veech condition for the $\cf=$DISP. In particular, the Sarnak and Veech conditions are equivalent for $\cf=$DISP.  Moreover,   for every sequence $(N_k)$ along which $\bfu$ is generic, $\bfu$ satisfies the averaged 2-Chowla property:
\beq\label{avch1}
\lim_{H\to\infty}\frac1H\sum_{h\leq H}\lim_{k\to\infty}\frac{1}{N_k}\left|\sum_{n\leq N_k}\bfu(n)\overline{\bfu(n+h)}\right|=0\eeq
and, for all sequences $\bfv_1,\ldots,\bfv_{k}$ (bounded by~1), we have
\beq\label{avch2}
\lim_{H\to\infty}\frac1{H^k}\sum_{h_1,\ldots,h_k\leq H}\lim_{k\to\infty}\frac{1}{N_k}\left|\sum_{n\leq N_k}\bfu(n)\prod_{i=1}^k\bfv_i(n+h_i)\right|=0.\eeq
\end{Cor}
Property~\eqref{avch2}, called the {\em averaged Chowla property} (cf.\ \cite{Ma-Ra-Ta}), follows from \eqref{avch1} -- this will be shown in Appendix~\ref{a:drugi}.  For an alternative approach to obtain the assertions of the above corollary, see the method in~\cite{FrPol}: for \eqref{avch1} cf.\ Thm.~4.1 and Prop.\ 4.3 \cite{FrPol}, and for \eqref{avch2} cf.\ Thm.~2.1 and Prop.\ 5.1 therein.

\begin{Cor}\label{c:averagedCh2}
Let $\bfu\colon\N\to\C$ be a multiplicative function bounded by~1. If, for each Dirichlet character $\chi$, $\bfu\cdot\chi$ satisfies the short interval behaviour~\eqref{ortmsvf}~\footnote{This is equivalent to saying that along arithmetic sequences the averages on a typical short interval vanish.} then $\bfu$ satisfies the averaged Chowla property~\eqref{avch2} (along each sequence $(N_k)$ for which $\bfu$ is generic).
\end{Cor}

The above result with  $(N_k)$ being all positive integers has been proved by Matom\"aki, Radziwi\l\l \ and Tao in~\cite{Ma-Ra-Ta}. Note that to obtain Corollary \ref{c:averagedCh2} it is enough to show that irrational rotations satisfy the strong $\bfu$-MOMO property. Via the DDKBSZ criterion, this follows from the ergodic property AOP introduced and proved to hold for totally ergodic rotations in \cite{Ab-Le-Ru}.  Thus, we obtain an ergodic proof of the averaged Chowla property for each sequence $(N_k)$ as above for all multiplicative $\bfu$ enjoying the special short interval behavior. In particular, assuming that $\bfu$ is generic, we get a non-quantitative version of \cite{Ma-Ra-Ta}.

In Section~\ref{s:momo}, we prove (see Theorem~\ref{t:thmC}) that for each $\bfu:\N\to\C$ taking finitely many values, if it satisfies the Sarnak property for the class $\mathscr{C}_{\rm ZE}$, then no positive entropy system satisfies the strong $\bfu$-MOMO property (this was previously known for the Liouville function assuming the Chowla conjecture \cite{Ab-Ku-Le-Ru2}).

\section{Characteristic classes}\label{s:characteristic}

\subsection{Definition, examples, basic properties}
Recall that a class $\cf$ of measure-theoretic dynamical systems is characteristic if it is closed under taking isomorphisms, factors and (countable) joinings. Recall also the following classical result on such classes (see e.g. \cite{Ru}).

\begin{Prop}\label{p:najwf}  Given a characteristic class $\cf$,
each automorphism $R$ on $(Z,\mathcal{D},\kappa)$ has a largest $\cf$-factor, denoted by $\mathcal{D}_{\cf}$.
\end{Prop}

The following result whose proof is based on a fundamental non-disjointness lemma from \cite{Le-Pa-Th} will be crucial for us:

\begin{Prop}[\cite{Ru}]
\label{p:sated}
Let $(X,\cb,\nu,T)$ be a measure-theoretic dynamical system in the characteristic class $\cf$, and let $(Z,\cd,\kappa,R)$ be any measure-theoretic dynamical system. Then any joining of $R$ and $T$ is relatively independent over the largest $\cf$-factor $\cd_{\cf}$ of $R$.
That is: if $g\in L^2(Z,\kappa)$ is such that $\EE^{\kappa}[g|\cd_{\cf}]=0$, and if $\rho$ is a joining of $T$ and $R$, then for any $f\in L^2(X,\nu)$ we have
\beq\label{relind1}
\EE^\rho(f\otimes g)=0.\eeq
\end{Prop}

{\bf \noindent Examples of characteristic classes} (some acronyms are used for those which will be used in the sequel):
\begin{itemize}
\item ALL: all automorphisms of  standard Borel probability spaces;
\item $\{\ast\}$: the identity on the one-point space;
\item ID: identity automorphisms (of all standard Borel probability spaces);
\item  DISP: discrete spectrum automorphisms;
\item RDISP: rational discrete spectrum automorphisms;
\item DISP($G$): discrete spectrum automorphisms whose group of eigenvalues is contained in fixed countable subgroup $G$ of the circle;

\item ${\rm NIL}_s$: automorphisms whose a.a.\ ergodic components are inverse limits of $s$-step nilautomorphisms. The fact that ergodic joinings of nilsystems remain  nil, see Proposition 15, page 186 in the book \cite{Ho-Krbook}, and the same holds for inverse limits (this is actually Lemma A.4 in \cite{Fr-Ho1}). Regarding factors of ergodic nilsystems, see Theorem 11 in page 230 \cite{Ho-Krbook}. Here $s\in\N$.

\item DIST: distal automorphisms are those which are given as a transfinite (indexed by ordinals smaller than a fixed countable ordinal) sequence of consecutive extensions each of which  either has relative discrete spectrum or (in case of a limit ordinal) is the corresponding inverse limit. The structural theorem Theorem~6.17 together with the concluding remark (for $\Z$-actions) on page 139 \cite{Fu0} tell us that each system has a largest distal factor, hence DIST is closed under countable joinings. In Lemma~\ref{l:re3}, we note that an automorphism is distal if and only if a.a.\ its ergodic components are distal. To see that this class is closed under taking factors, let us first recall that this fact holds for ergodic automorphisms (see Theorem 10.18 \cite{Gl}). If now $\ca\subset\cd$ is a factor of a distal automorphism $(Z,\cd,\kappa,R)$ then $\ca$ (as an $R$-invariant $\sigma$-algebra) is also a factor of a.e.\ of its ergodic components. So a double use of Lemma~\ref{l:re3} together with Theorem 10.18 from \cite{Gl} gives that $\ca$ is also distal.
\item ZE: zero entropy automorphisms;
\item RIG$_{(q_n)}$: automorphisms with a fixed sequence $(q_n)$  of rigidity;
\item multipliers $\mathscr{M}(\mathscr{D}^{\perp})$ of a class $\mathscr{D}^\perp$  ($\mathscr{D}$ is {\bf any} class of automorphisms and by $\mathscr{D}^{\perp}$ we mean the set of automorphisms disjoint from all systems from  $\mathscr{D}$, and by $\mathscr{M}(\mathscr{D}^{\perp})$ we mean the set of systems whose all joinings with any element of $\mathscr{D}^{\perp}$ remain in $\mathscr{D}^{\perp}$); interesting classes of multipliers arise e.g.\ for $\mathscr{D}$=all weakly mixing (cf.\ Proposition 5.1 in \cite{Le-Pa}) or all mixing automorphisms; see \cite{Gl0}, \cite{Le-Pa}.
 \item the class of factors of all infinite self-joinings of a fixed automorphism $R$ (the smallest characteristic class containing $R$); especially in case of MSJ and simple automorphisms (cf.\ \cite{Gl}, Chapter~12). Characteristic classes of such type were used in~\cite{Le-Ri-Ru}.
\end{itemize}

Note also that the intersection of any family of characteristic classes yields again a characteristic class.  In Section~\ref{s:onec}, we will show that each characteristic class $\cf$ determines another characteristic class $\cf_{\rm ec}$ consisting of those automorphisms whose ergodic components are in $\cf$.

\subsection{The smallest nontrivial and the largest proper characteristic class}
An obvious observation has been made already in the introduction that the family ALL of all automorphisms is the largest characteristic class, while the one-element $\{\ast\}$ family (which is  the one-point space automorphism) is the smallest characteristic class. It is more interesting however that the smallest non-trivial and the largest proper characteristic classes exist.

\begin{Prop}\label{p:smallest}
ID is the smallest non-trivial characteristic class.
\end{Prop}
\begin{proof} Let us first notice that the system $([0,1],{\rm Leb}, {\rm Id})$ has {\bf all} other identities as factors. Indeed, any standard Borel probability space is determined by a sequence $(t_i)_{i\geq 0}$ of non-negative numbers such that $\sum_{i\geq0}t_i=1$ and $t_0$ corresponds to the mass of the continuous part and $t_1,t_2,\ldots$ correspond to the masses of atoms. Then, take the corresponding partition of $[0,1]$ into intervals $I_i$ of length $t_i$ and, for each $i\geq1$, the factor map will glue points in $I_i$.

Now, notice that any non-trivial characteristic class $\cf$ contains a non-ergodic automorphism. Indeed, suppose that $T$ is ergodic, acting on a non-trivial space $(Y,\nu)$. Since $Y$ is non-trivial and so is $T$, the graph joinings $\Delta_{{\rm Id}}$ and $\Delta_T$ are ergodic and different, so any non-trivial convex combination of them yields a non-ergodic member of $\cf$.  It follows that by taking the factor $\mathcal{I}_{\nu}$ of (a.e.) $T$-invariant sets (which belongs to $\cf$), we obtain the identity on a non-trivial standard Borel probability space $(\ov{Y},\ov{\nu})$. But then the infinite Cartesian product $(\ov{Y}^{\times\infty},\ov{\nu}^{\otimes\infty})$ is also in $\cf$ and this infinite product is isomorphic to $([0,1],{\rm Leb})$, which completes the proof.
\end{proof}

In order to prove the existence of the largest characteristic (proper) class, we need to recall some results.

\begin{Th} [non-ergodic Sinai's factor theorem \cite{Ki-Ra, Ta-Ve}]\label{t:Sinai}
Assume that $R$ is an automorphism of $\zdr$ and let $\rho=\int_{X/\mathcal{I}_{\rho}}\rho_{\ov{x}}\,dm(\ov{x})$ stand for the ergodic decomposition of $\rho$. Assume that
$$
m-{\rm essinf}_{\ov{x}\in X/\mathcal{I}_\rho}h(\rho_{\ov{x}},R)\geq\alpha>0.$$
Then a Bernoulli automorphism of entropy $\alpha$ is a factor of $R$.\end{Th}
In \cite{Ta-Ve} (see Theorem 4.3 therein) the above result is attributed to Kieffer and Rahe \cite{Ki-Ra}, see also \cite{Se} p.2 (The non-ergodic factor theorem).

We also need the following well-known result (we include its proof for completeness).

\begin{Prop}\label{p:ST?}
Each automorphism $R$ is a factor of a self-joining of the (infinite entropy) Bernoulli automorphism $([0,1]^{\Z},{\rm Leb}^{\ot\Z},S)$.\end{Prop}

\begin{Remark}\label{shift realization}
Before we prove the above result, let us notice that any automorphism $T$ of $\xbm$ has an isomorphic copy in the space $([0,1]^{\Z},\kappa,S)$.\footnote{The same arguments apply to $\D=\{z\in\C : |z|\leq 1\}$ instead of $[0,1]$.}
Consider first the aperiodic part of $T$ which is realized on a standard Borel space. This space is isomorphic to $[0,1]$, via a Borel isomorphism $I$. It follows that the distribution $\mu'$ of the process $(I\circ S^k)_{k\in\Z}$ yields a realization of the aperiodic part. Now, $\mu'$ takes measure zero on the set of periodic points for the shift. Moreover, the set of periodic points of $S$ can be identified with a subset of $[0,1]$, points of period~2 with a subset of $[0,1]^2$ etc., and we can easily settle an isomorphism of the fixed point subspace for $T$ with a subset of $[0,1]$, period~2 points with a subset of $[0,1]^2$, etc. Thus, it suffices to take $\kappa$ equal to the sum of $\mu'$ and the relevant atomic measures corresponding to the periodic points.
\end{Remark}
\begin{proof}[Proof of Proposition~\ref{p:ST?}] Fix any automorphism $R$ of $\zdr$ and take its isomorphic copy in the space $([0,1]^{\Z},\kappa,S)$.  Take the product space $([0,1]^{\Z}\times [0,1]^{\Z},{\rm Leb}^{\otimes\Z}\ot \kappa)$ and consider the map $\psi\colon[0,1]^{\Z}\times [0,1]^{\Z}\to[0,1]^{\Z}$ given by
$$
(x_n,y_n)\mapsto (x_n+y_n).$$
Then $\psi_\ast({\rm Leb}^{\otimes\Z}\ot \kappa)={\rm Leb}^{\otimes\Z}$ and clearly the join of the $\sigma$-algebra of the first coordinate and of $\psi^{-1}(\mathcal{B}([0,1]^{\otimes\Z}))$ is the product $\sigma$-algebra in $[0,1]^{\otimes\Z}\times [0,1]^{\otimes\Z}$. The result follows.
\end{proof}

We now prove the following.

\begin{Lemma}\label{l:dode}
Assume that $\cf$ is a characteristic class such that $\cf\setminus {\rm ZE}\neq\emptyset$. Then $\cf={\rm ALL}$.
\end{Lemma}
\begin{proof}
Fix $T\in \cf\setminus {\rm ZE}$. Because of Proposition~\ref{p:ST?}, we only need to prove that the infinite entropy Bernoulli automorphism is in $\cf$. The first step is to consider the factor of $T$ that arises by gluing together the periodic part and the ergodic components from the aperiodic part whose entropy is smaller than $\alpha=h(T)$. Clearly, this factor remains in $\cf\setminus {\rm ZE}$. Moreover, in its ergodic decomposition we have a single point and the remaining part (which may still be non-ergodic) consists of ergodic components of entropy at least $\alpha$. By Theorem~\ref{t:Sinai}, it follows that as a further factor $R\in \cf\setminus {\rm ZE}$ we can obtain a non-ergodic automorphism with two ergodic components: one of them is a Bernoulli of entropy $\alpha$ and the other one is a fixed point. Finally, we take the infinite Cartesian product $R^{\times\infty}$. It is not hard to see that a.e.\ ergodic component of this automorphism is a Bernoulli with infinite entropy. Using once more Sinai's theorem (Theorem~\ref{t:Sinai}), we obtain that a Bernoulli with infinite entropy belongs to $\cf$ which completes the proof.
\end{proof}

Now, using the lemma we obtain the following.
\begin{Prop}\label{p:largest}
ZE is the largest proper characteristic class.
\bez
\end{Prop}

\subsection{Characteristic classes given by ergodic components}\label{s:onec}
Assume that $\cf$ is a characteristic class. By $\cf_{\rm ec}$ we denote the class of those automorphisms $R$ such that (a.e.) ergodic components of $R$ are in $\cf$ (or more precisely, in $\cf\cap$\,Erg, where Erg stands for the family of all ergodic automorphisms). Note that we have
\beq\label{mo1}
\cf \cap {\rm Erg}=\cf_{\rm ec}\cap {\rm Erg}.\eeq

\begin{Lemma}
$\cf_{\rm ec}$ is a characteristic class.
\end{Lemma}
\begin{proof}
The proof has two parts: we need to show that $\cf_{\rm ec}$ is closed under taking factors and joinings.
~\paragraph{Factors}
Let $R$ acting on $\zdk$ belong to $\cf_{\rm ec}$ and fix a factor $\ca\subset \cd$ of $R$. Let $\kappa=\int \kappa_{\ov{x}}\,dP(\ov{x})$ denote the ergodic decomposition of $\kappa$. Since the ergodic components $\kappa_{\ov{x}}$ are $R$-invariant measures, $\ca$ (being an $R$-invariant sub-$\sigma$-algebra) is also a factor of the automorphism $(Z,\kappa_{\ov{x}},R)$ and $\kappa|_{\ca}=\int \kappa_{\ov{x}}|_{\ca}\,dP(\ov{x})$ is the ergodic decomposition of $\kappa|_{\ca}$. It follows that the ergodic components of the factor are factors of ergodic components of $R$, and since $R\in \cf_{\rm ec}$, $(\kappa_{\ov{x}},R)\in \cf$, so also $(\kappa_{\ov{x}}|_{\ca},R|_{\ca})\in \cf$ for $P$-a.e.\ $\ov{x}$.

\paragraph{Joinings}  Take $(X,\mu,T)$ and $(Y,\nu,S)$ from $\cf_{\rm ec}$ and let $\rho\in J(T,S)$ be their joining.  Let
$$\rho=\int_0^1 \rho_t\,dP(t),\;\mu=\int_0^1 \mu_t\,dQ(t),\;\nu=\int_0^1\nu_t\,dR(t)$$
be the relevant ergodic decompositions. Then
   $$\int_0^1\mu_t\,dQ(t)=\mu=\rho|_X=\int_0^1\rho_t|_X\,dP(t),$$
   so since $\rho_t|X$ are also ergodic,  these two decompositions are the same. So for a $P$-``typical'' $t\in[0,1]$, the projection of $\rho_t$ on $X$ is an ergodic component of $T$. The same argument applies on the coordinate $Y$ and we see that the ergodic components of $\rho$ are joinings of ergodic components of $\mu$ and $\nu$. It follows that $(X\times Y,\rho,T\times S)\in \cf_{\rm ec}$. The argument extends to countable joinings.
\end{proof}

\subsubsection{Largest characteristic factor} \label{largDCSec}
\paragraph{ID, ZE, DISP and RIG$_{(q_n)}$}
Given a characteristic class $\cf$, according to Proposition~\ref{p:najwf}, each automorphism $R$ acting on $\zdk$ has a largest $\cf$-factor $\cd_{\cf}\subset\cd$. Often, its description is classical:
\begin{itemize}
\item
the $\sigma$-algebra of invariant sets for $\cf={\rm ID}$,
\item
the Pinsker factor for $\cf={\rm ZE}$,
\item
the Kronecker factor for $\cf={\rm DISP}$,
\item
the largest factor for which $(q_n)$ is a rigidity time for $\cf={\rm RIG}_{(q_n)}$.
\end{itemize}
\paragraph{DISP$_{\rm ec}$}
We will comment now on $\cd_{\cf_{\rm ec}}$ when $\cf={\rm DISP}$, cf.\ Proposition~\ref{p:examples}~(ii) and its connections with the theory of nil-factors. Most of the material presented below is known to aficionados but not necessarily the material is explicitly present in the literature. Our discussion is based on \cite{Fr-Ho3}, \cite{Fu}, \cite{Ho-Kr} and \cite{Ho-Krbook}.
We provide some details to explain clearly why the problem of whether $\mob\perp {\rm DISP}_{\rm ec}$ is open, cf.\ Corollary~\ref{c:averagedCh'}, Corollary~\ref{pods} and Remark~\ref{r:remark56}.


Recall that
according to the Furstenberg-Zimmer theory \cite{Fu}, given $R$ on $\zdk$ and a factor $\mathcal{C}\subset\mathcal{D}$, there exists a certain intermediate factor
\[
\mathcal{C}\subset \K=\K(\mathcal{C})\subset \mathcal{D},
\]
called the {\em relative Kronecker factor} (with respect to $\mathcal{C}$).
It is the largest intermediate factor with the following property (see condition C$_2$ in \cite{Fu}, p.~131):
\begin{equation}\label{C41}
\begin{split}
\parbox[t]{0.8\linewidth}{there exists a dense set of functions $F\in L^2(\K,\kappa|_{\K})$ such that
for each $\delta>0$ there is a finite set $g_1,\ldots,g_k\in L^2(\K,\kappa|_{\K})$ such that for each $h\in\Z$,
\[
\min_{1\leq j\leq k}\|F\circ R^h-g_j\|_{L^2(\kappa_y)}<\delta
\]}
\end{split}
\end{equation}
for a.e.\ $y\in Z/\mathcal{C}$, where
\beq\label{disint}
\kappa|_{\K}=\int_{Z/\mathcal{C}}\kappa_y\,d\kappa(y).
\eeq
Whenever condition~\eqref{C41} holds, we speak of {\em relative compactness} or of {\em relatively discrete spectrum} of the intermediate factor over $\mathcal{C}$.

A particular situation arises when $\cc=\mathcal{I}_\kappa$, i.e.\ it is the $\sigma$-algebra of invariant sets. Then~\eqref{disint} is nothing but the ergodic decomposition of $\kappa$ and the conditional measures $\kappa_y$ are also $R$-invariant. In this case condition~\eqref{C41} yields in a.e.\ fiber $\pi^{-1}(y)$ (where $\pi\colon Z/\K\to Z/\mathcal{I}_\kappa$ stands for the factor map) a dense set of functions $F|_{\pi^{-1}(y)}$ {in $L^2(\K,\kappa_y)$ whose orbits under the unitary action of $R$ are relatively compact. It follows that the (ergodic) automorphism $(R,\kappa_y)$ has discrete spectrum for a.e.\ $y\in Z/\mathcal{I}_\kappa$. In other words,
\[
\K(\mathcal{I}_\kappa)\subset  \mathcal{D}_{{\rm DISP_{ec}}}
\]
In fact, the opposite inclusion is also true, i.e.
\begin{equation}\label{opo}
\mathcal{D}_{{\rm DISP_{ec}}} = \K(\mathcal{I}_\kappa),
\end{equation}
that is,  $\mathcal{A}:=\mathcal{D}_{{\rm DISP_{ec}}}$ has relatively discrete spectrum over $\mathcal{I}_\kappa$.
Indeed, by the definition of $\mathcal{A}$, a.e.\ ergodic component of $R|_{\mathcal{A}}$ has discrete spectrum.
Fix $F\in L^\infty(\ca,\kappa|_{\ca})$. Fix also $\vep,\delta>0$ and $k\geq1$. Consider the set $W_k\subset Z/\mathcal{I}_\kappa$ of those $y$ for which
$$
\min_{-k\leq j\leq k}\left\|F\circ R^n-F\circ R^j\right\|_{L^2(\kappa_y)}<\vep$$
for each $n\in\Z$. Since on each fiber  $R$ is an ergodic automorphism with discrete spectrum, the measure of $W_k$ goes to~1, when $k\to\infty$, so it will be greater than $1-\delta$ for $k$ large enough.  It follows that the function $F$ is compact as it has been defined in the proof of Theorem 6.15 \cite{Fu}. Therefore, $F\in L^2(\K(\mathcal{I}_\kappa))$, which (by \cite{Fu}) concludes the proof of~\eqref{opo}.

\begin{Remark} \label{r:Tim+Nil}
As a matter of fact, in \cite{Au}, the Furstenberg-Zimmer theory is developed without assuming ergodicity (cf.\ e.g.\ Proposition 5.7 therein to obtain the equality $\mathcal{D}_{{\rm DISP}_{\rm ec}}=\mathcal{K}(\mathcal{I}_\kappa)$).
\end{Remark}
We will now see that $\mathcal{D}_{\rm DISP_{ec}}$ appears naturally in the classical theory of characteristic nil-factors \cite{Ho-Kr,Ho-Krbook}.\footnote{We would like to thank Bryna Kra and Nikos Frantzikinakis for fruitful discussions and useful references on this subject.}
Recall that if $R$ acting on $\zdk$ is ergodic then for a function $f\in L^\infty\zdk$ its $u^s$ norms (in fact, seminorms) are defined in the following way:
\beq\label{hkg1}
\|f\|_{u^1}:=\Big|\int f\,d\kappa\Big|,
\eeq
\beq\label{hkg2}
\|f\|_{u^{s+1}}^{2^{s+1}}:=\lim_{H\to\infty}\frac1H\sum_{h\leq H}\|f\circ R^h\cdot\ov{f}\|_{u^s}^{2^s}.
\eeq
If $R$ is non-ergodic then  instead of \eqref{hkg1}, we put
\[
\|f\|^2_{u^1}:=\lim_{H\to\infty}\frac1H\sum_{h\leq H}\int f\circ R^h\cdot \ov{f}\,d\kappa
\]
and \eqref{hkg2} remains unchanged.
Then, by \cite{Ho-Kr, Ho-Krbook}, for each $s\geq1$ there is a special factor $\mathcal{Z}_s=\mathcal{Z}_s(R)\subset\cb$, namely, the largest factor whose
\beq\label{hkchar}
\mbox{a.e.\ ergodic component is an inverse limit of $s$-step nil-systems.}
\eeq
In other words,  $\mathcal{Z}_s(R)$ is the largest (characteristic) ${\rm NIL}_s$-factor of $R$.
Moreover (see Proposition 7 (page 138) and Proposition 13 (page 141) in~\cite{Ho-Krbook}),
\begin{equation}\label{prop13fubook}
\|f\|_{u^{s+1}}=0 \iff f\perp L^2(\mathcal{Z}_s) \iff f\perp L^2(\mathcal{Z}_s(R,\kappa_y)) \text{ for $\kappa$-a.e. }y.
\end{equation}

A special case arises when our measure-preserving systems are Furstenberg systems of a bounded $\bfu\colon\N\to\C$.
As in (for example) \cite{Fr}, see Sections~2.4 and 2.5 therein, one can introduce the uniformity norms (along subsequences of intervals) for $\bfu$. The definitions are given in \eqref{eq:1us} and \eqref{eq:2us}. They are  very similar to those (in the non-ergodic case) to the definitions for functions.

We will now show that
\begin{equation}\label{rownosc}
\mathcal{Z}_1(R)=\mathcal{K}(\mathcal{I}_\kappa).
\end{equation}
 If $R$ is ergodic then the above means just that
\begin{equation}\label{rownerg}
\mathcal{Z}_1 \text{ is the Kronecker factor of $R$}.
\end{equation}

To see that~\eqref{rownerg} indeed holds, notice that~\eqref{prop13fubook} for $s=1$ yields
\[
\|f\|_{\bfu^2}=0 \iff f\perp L^2(\mathcal{Z}_1)
\]
and it remains to notice that (using the Wiener lemma)
\[
\|f\|^4_{\bfu^2}=\lim_{H\to \infty}\frac{1}{H}\sum_{h\leq H}\|f\circ R^h \cdot \ov{f} \|^2_{\bfu^1}=\]\[\lim_{H\to\infty}\frac{1}{H}\sum_{h\leq H}\left| \int f\circ R^h \cdot \ov{f}\right|^2\,d\kappa\to\sum_{z\in\mathbb{T}}\sigma_f(\{z\})^2,
\]
where $\sigma_f$ stands for the spectral measure of $f$.

Let us return to a possibly non-ergodic $R$. The inclusion $\mathcal{Z}_1(R)\subset \mathcal{K}:=\mathcal{K}(\mathcal{I}_\kappa)$ follows directly by Theorem 5.2 in~\cite{Fr-Ho3}. To obtain the opposite inclusion, one can argue in the following way. Suppose that $f\perp L^2(\mathcal{Z}_1(R))$ and $|f|\leq 1$. Take $g\in L^2(\mathcal{K})$, cf.\ \eqref{C41} with $F=g$. We want to show that $\int fg\, d\kappa=0$. Notice that
\[
\int fg\, d\kappa=\int \left(\frac{1}{N}\sum_{n\leq N}\int f\circ T^n \cdot g\circ T^n\, d\kappa_y\right)\, d\kappa(y).
\]
Let $g_j$, $1\leq j\leq k$, be as in~\eqref{C41}. Then
\begin{multline*}
\left|\frac{1}{N}\sum_{n\leq N}\int f\circ T^n \cdot g\circ T^n\, d\kappa_y\right|\\
\leq \left|\sum_{1\leq j\leq k}\frac{1}{N}\sum_{n\leq N}\int f\circ T^n \cdot g_j\, d\kappa_y \right| + \frac{1}{N}\sum_{n\leq N}\min_{1\leq i\leq k}\int \left|f\circ T^n (g\circ T^n-g_i) \right|\, d\kappa_y.
\end{multline*}
Each term  in the average in the second summand is bounded by $\delta$. Moreover,
\begin{equation}\label{gra}
\frac{1}{N}\sum_{n\leq N}\left|\int f\circ T^n \cdot g_j\, d\kappa_y \right|^2 \to \sum_{z\in\mathbb{T}}\sigma_{f,g_j,\kappa_y}(\{z\})^2,
\end{equation}
where $\sigma_{f,g_j,\kappa_y}$ stands for the spectral measure of the pair $f,g_j$ (on the ergodic component $(\pi^{-1}(y),\kappa_y)$ given by $y$). But by~\eqref{prop13fubook} and~\eqref{rownerg}, we have
\begin{align*}
f\perp L^2(\mathcal{Z}_1) &\iff f\perp L^2(\mathcal{Z}_1(R,\kappa_y))\text{ for a.e. }y\\
& \iff \sigma_{f,\kappa_y} \text{ is continuous for a.e. }y
\end{align*}
($\sigma_{f,\kappa_y}$ stands for the spectral measure of $f$ on the ergodic component $(\pi^{-1}(y),\kappa_y)$ given by $y$).
Since $f\perp L^2(\mathcal{Z}_1)$ and $\sigma_{f,g_i,\kappa_y}\ll \sigma_{f,\kappa_y}$, it remains to use the classical equivalence
\[
\frac{1}{N}\sum_{n\leq N}a_n \to 0 \iff \frac{1}{N}\sum_{n\leq N}a_n^2\to 0
\]
for any bounded sequence $(a_n)\subset [0,\infty)$, to conclude that the limit in~\eqref{gra} is equal to zero. Thus $f\perp L^2(\mathcal{Z}_1) \implies f\perp L^2(\mathcal{K})$.

Finally, let us compare the above with the notion of relative weak mixing. Recall that relative weak mixing over $\mathcal{I}_\kappa$ for $f$ means that
\[
\frac1H\sum_{h\leq H}\int \left|\E(f\circ R^h\cdot \overline{f}|\mathcal{I}_\kappa)\right|^2\, d \kappa \to0.
\]
Moreover,
$$
\frac1H\sum_{h\leq H}\int \left|\E(f\circ R^h\cdot \overline{f}|\mathcal{I}_\kappa)\right|^2\, d \kappa=\int\left(
\frac1H\sum_{h\leq H}\left|\int f\circ R^h\cdot \overline{f}\,d\kappa_y\right|^2\right)\, d \kappa,
$$
and, once more by the Wiener lemma,
$$\frac1H\sum_{h\leq H}\left| \int f\circ R^h\cdot \overline{f}\,d\kappa_y\right|^2\to\sum_{z\in\T}
\sigma_{f,\kappa_y}(\{z\})^2.
$$
It follows immediately that $\sigma_{f,\kappa_y}$ is continuous for a.e.\ $y$ if and only if $f$ is relatively weakly mixing over $\mathcal{I}_\kappa$.

The above discussion can be summarized in the following statement.
\begin{Cor}\label{pods}
Let $(Z,\cd,\kappa,R)$ be a measure-theoretic dynamical system and let $f\in L^2\zdk$. The following conditions are equivalent:
\begin{enumerate}
\item[(i)] $f\perp L^2(\mathcal{Z}_1)$,
\item[(ii)] $f\perp L^2(\mathcal{D}_{\rm DISP_{ec}})$,
\item[(iii)] $f\perp L^2(\mathcal{K}(\mathcal{I}_\kappa))$,
\item[(iv)] $\sigma_{f,\kappa_y}$ is continuous for $\kappa$-a.e. $y$,
\item[(v)] $f$ is relatively weakly mixing over $\mathcal{I}_\kappa$.
\end{enumerate}
\end{Cor}

\subsubsection{A class vs. its ec-class}\label{s:distalandother}
Let us continue our observations on the relations between characteristic classes and the corresponding ec-classes. Note that in general there are no relations between $\cf$ and $\cf_{\rm ec}$:

\begin{Prop}\label{p:examples}
We have:\\
(i) ${\rm ZE}={\rm ZE}_{\rm ec}$,  ${\rm ALL}={\rm ALL}_{\rm ec}$, ${\rm ID}={\rm ID}_{\rm ec}$, ${\rm NIL}_s=({\rm NIL}_s)_{\rm ec}$, $\{\ast\}\subsetneq\{\ast\}_{\rm ec}$;\\
(ii) ${\rm DISP}\subsetneq {\rm DISP}_{\rm ec}$;\\
(iii) ${\rm RDISP}={\rm RDISP}_{\rm ec}$;\\
(iv) ${\rm DIST}={\rm DIST}_{\rm ec}$;\\
(v) $\Big({\rm RIG}_{(q_n)}\Big)_{\rm ec}\subsetneq {\rm RIG}_{(q_n)}$.
\end{Prop}
\begin{proof}[Proof of (i)-(iii)]~\\ ~
\noindent
{\bf (i)} The first claim follows from the fact that the entropy function is convex, the other claims are obvious.

\noindent
{\bf (ii)} If an automorphism has discrete spectrum then its $L^2$-space is generated by eigenfunctions. The restrictions (if non-zero) of these (global) eigenfunctions yield orthonormal bases in $L^2$-spaces of ergodic components. The inclusion is strict since
$(x,y)\mapsto (x,x+y)$ (on $\T^2$, considered with Lebesgue measure) does not have discrete spectrum while the ergodic components do.

\noindent
{\bf (iii)} We want to show that if each ergodic component has  rational discrete spectrum then the whole automorphism has. Given $p/q\in\mathbb{Q}
$ and an ergodic component $c$, we choose $f_c$ a modulus~1 eigenfunction corresponding to the eigenvalue $e^{2\pi ip/q}$. Since $f_c$ is unique up to a constant of modulus~1, this choice can be done measurably. In this way, we will create global eigefunctions.\footnote{The same argument works if we consider the characteristic class of automorphisms having discrete spectrum contained in a {\bf fixed} countable subgroup of the circle.}
\end{proof}
Before we give the proof of (iv), we need to recall some more notions and facts from the relative ergodic theory, e.g.\ \cite{Fu, Gl}.
Given an automorphism $T$ of $\xbm$ and its factor $S$ on $\ycn$ with the factor map $\pi\colon X\to Y$, we say that this extension is {\em relatively ergodic} (rel.\ erg.) if each $f\in L^1\xbm$ satisfying $f\circ T=f$ ($\mu$-a.e.) is $\pi^{-1}(\mathcal{C})$-measurable. It follows immediately from the definition that:
\begin{itemize}
\item
any composition of relatively ergodic extensions remains relatively ergodic;
\item
an inverse limit of relatively ergodic extensions remains relatively ergodic (as the conditional expectation, with respect to a factor, of an invariant function remains invariant);
\item
$\ov{\pi}\colon Y\to\ov{Y}:=Y/\mathcal{I}_\nu$, where $(\ov{Y},\ov{\nu})$ stands for the space of ergodic components (on which acts the identity map), is relatively ergodic.
\end{itemize}

Let $\mu=\int_Y \mu_y\,d\nu(y)$ stand for the disintegration of $\mu$ over $\nu$ and let
$$
\nu=\int_{\ov{Y}}\ov{\nu}_{\ov{y}}\,d\ov{\nu}
$$
denote the ergodic decomposition of $\nu$ (which is precisely the disintegration of $\nu$ over $\ov{\nu}$). Then the ergodic components of $S$ acting on $Y$ are of the form
$$
(\ov{\pi}^{-1}(\ov{y}),\ov{\nu}_{\ov{y}}, S)
$$
(the measures $\ov{\nu}_{\ov{y}}$ are $S$-invariant). Therefore, we have the following lemma.

\begin{Lemma}\label{l:re1} If $T$ is relatively ergodic over $S$ then the ergodic components of $T$ are of the form
$$
\left(\pi^{-1}(\ov{\pi}^{-1}(\ov{y})),\int_{\ov{\pi}^{-1}(\ov{y})}\mu_y\, d \ov{\nu}_{\ov{y}}(y)\right).$$
\end{Lemma}
Note that it follows that the ergodic components of $T$ have as their factors (via the relevant restriction of $\pi$) ergodic components of $S$, and that the spaces of ergodic components of $T$ and $S$ are the same (i.e.\ $\ov{X}=\ov{Y}$).

\begin{Lemma} \label{l:re2}
Let $T$ be relatively ergodic over $S$. Then the following are equivalent:
\begin{enumerate}[(a)]
\item
$T$ over $S$ has relatively discrete spectrum.
\item
The ergodic components of $T$ have relatively discrete spectrum over the ergodic components of $S$ being their relevant factors.
\end{enumerate}
\end{Lemma}
\begin{proof} By Lemma~\ref{l:re1}, we see that the disintegration of an ergodic component $\pi^{-1}(\ov{\pi}^{-1}(\ov{y}))$   over $\ov{\pi}^{-1}(\ov{y})$ (which is its factor) consists of the same conditional measures $\mu_y$ as the total disintegration of $\mu$ over $\nu$. We proceed now as in the proof of the equality $\mathcal{K}(I_\kappa)=\mathcal{D}_{\rm DISP_{ec}}$ (page~\pageref{opo}), showing compactness.
\end{proof}

Recall that an automorphism $T$ is {\em distal} if it is a limit of a transfinite (indexed by countable ordinals) sequence of consecutive maximal Kronecker extensions (if an ordinal is not isolated, we pass to the corresponding inverse limit). Note that, by the very definition, the $\sigma$-algebra Inv is contained in {\bf the} Kronecker factor of $T$, so in  this transfinite chain of consecutive extensions, all but (perhaps) the first one are relatively ergodic. By applying Lemma~\ref{l:re2} and transfinite induction, we obtain the following.

\begin{Lemma}\label{l:re3} $T$ is distal if and only if all its ergodic components are distal.\end{Lemma}

\begin{proof}[Proof of (iv)-(v)]~\\ ~
{\bf (iv)}
This follows directly from Lemma~\ref{l:re3}.\\
\noindent
{\bf (v)}
It is clear that if $(q_n)$ is a rigidity time for an a.e.\ ergodic component, it is also a rigidity time for the whole automorphisms. Not vice versa however (for $(q_n)$ sufficiently sparsed). We will provide a relevant construction below.
\end{proof}

\paragraph{${\rm RIG}_{\rm ec}$ is a proper subclass of ${\rm RIG}$}
Let us first notice that we only need to construct a continuous measure $\sigma$ on the circle such that
\beq\label{rigid10}
e^{2\pi i q_n\cdot}\to 1\text{ in measure }\sigma\text{ but not } \sigma-\text{a.e.}\eeq
Indeed, suppose \eqref{rigid10} holds, and consider on $\T^2$ the automorphism
$$
T(x,y)=(x,y+x)\text{ with measure }\sigma\ot\,{\rm Leb}.$$
If $F(x,y)=f(x)e^{2\pi i\ell y}$ then by \eqref{rigid10},
$$
\int|F(T^{q_n}(x,y))-F(x,y)|\,d\sigma(x)dy=\int|f(x)||e^{2\pi iq_n\ell x}-1|\,d\sigma(x)\to 0$$
when $n\to\infty$. On the other hand, the rotation by $x$ on an ergodic component $\{x\}\times\T$ has $(q_n)$ as its rigidity time if and only if $q_nx\to0$ mod~1. This is not true for $\sigma$-a.e.\ $x\in\T$ in view of \eqref{rigid10}.

We now sketch how to construct such a measure assuming that $(q_n)$ is sufficiently sparsed. Fix $0<p_n<1$ so that $p_n$ is decreasing to zero and $\sum_{n\geq1}p_n=\infty$. Set $f_n(x)=\{q_nx\}$. We intend to construct a Cantor set (together with a Cantor measure $\sigma$ on it). Let
$$
A_n:=f_n^{-1}([1/4,3/4]), \;B_n=f_n^{-1}([0,p_n]).$$
Our postulates are:
$$\sigma(B_n)=1-p_n,\; \sigma(A_n)=p_n.$$
In fact, we need to be more precise in description of the measure at stage $n$ to be able to continue its definition. So at stage $n$ the circle is divided into intervals of the form $[\frac j{q_n},\frac{j+1}{q_n})$ (many of such intervals are of measure $\sigma$ equal to zero). We now require that the conditional measures satisfy:
\beq\label{kf}
\sigma\left(B_n|[\frac j{q_n},\frac{j+1}{q_n})\right)=1-p_n,\,
\sigma\left(A_n|[\frac j{q_n},\frac{j+1}{q_n})\right)=p_n\eeq
for each $j=0,\ldots,q_n-1)$.
Passing to step $n+1$, we require that all the intervals $[\frac j{q_n},\frac{j+1}{q_n}))$   contain at least two intervals of the form $[\frac k{q_{n+1}},\frac{k+1}{q_{n+1}})$, we choose two of such (of course only in those $[\frac j{q_n},\frac{j+1}{q_n})$ which are of positive measure $\sigma$) and apply the rule~\eqref{kf} to $A_{n+1}$, $B_{n+1}$ with $p_n$ replaced with $p_{n+1}$.

Note that $\int e^{2\pi i q_nx}\,d\sigma(x)=1+O(p_n(1+p_n)+1\cdot p_n)$, so $e^{2\pi i q_n\cdot}\to 1$ in measure~$\sigma$. On the other hand $\sigma(A_n)=p_n$  and the sets $A_n$ are almost independent. Since $\sum_{n\geq1}p_n=\infty$, for $\sigma$-a.e.\ $x$, we have $x\in A_n$ for infinitely many $n$ (by the Borel-Cantelli lemma), so \eqref{rigid10} holds.

\subsubsection{Strong $\bfu$-MOMO property of systems whose visible measures yield systems in an ec-class}
While we have seen rather unclear relations between $\cf$ and $\cf_{\rm ec}$ (cf.\ Proposition~\ref{p:examples}), on the topological level we always have the following.

\begin{Prop}\label{p:tilde}
Let $\cf$ be a characteristic class. Then $\mathscr{C}_{\cf}\subset\mathscr{C}_{\cf_{\rm ec}}$.
\end{Prop}
\begin{proof} This follows immediately from the fact that a homeomorphism $T$ (acting on a compact metric space $X$) belongs to $\mathscr{C}_{\cf_{\rm ec}}$ if and only if for each $\kappa\in M^e(X,T)$, $(X,\mathcal{B}(X),\kappa,T)\in \cf$.\end{proof}

Note that in view of Proposition~\ref{p:tilde} and Proposition~\ref{p:examples},
\beq\label{tildeRIG}
\mathscr{C}_{{\rm RIG}_{(q_n)}}=\mathscr{C}_{({\rm RIG}_{(q_n)})_{\rm ec}}.\eeq

The special role of ec-classes stands in the next proposition.
\begin{Prop}\label{p:mo1}
Let $\cf$ be a characteristic class. Then $\bfu \perp \mathscr{C}_{\cf_{\rm ec}}$ if and only if each element in $\mathscr{C}_{\cf_{\rm ec}}$ satisfies the strong $\bfu$-MOMO property.
\end{Prop}
The below proof of Proposition~\ref{p:mo1} is an adaptation of the proof of Corollary~9 in \cite{Ab-Ku-Le-Ru2}. It uses the following elementary result (see Lemma~18 in \cite{Ab-Ku-Le-Ru2}).
\begin{Lemma}\label{l:cone}
	Assume that $(c_n)\subset \C$ and $(m_n)\subset \N$. Then if the sequence $(c_n)$ is contained in a closed convex cone which is not a half-plane then
	\[
	\frac{1}{m_N}\sum_{n\leq N}c_n \to 0 \iff \frac{1}{m_N}\sum_{n\leq N}|c_n|\to 0 \text{ as }N\to\infty.
	\]
\end{Lemma}

\begin{proof}[Proof of Proposition~\ref{p:mo1}]
Only one implication needs to be proved.
Suppose that $\bfu \perp \mathscr{C}_{\cf_{\rm ec}}$, and let $(Y,S)\in\mathscr{C}_{\cf_{\rm ec}}$. We fix $f\in C(Y)$, an increasing sequence $(b_k)$ in $\N$, with $b_1=1$ and $b_{k+1}-b_k\to\infty$, and a sequence $(y_k)$ of points in $Y$.
We introduce the finite set $\mathbb{A}:=\{1,e^{2\pi i/3},e^{4\pi i /3}\}$, and for each $k\ge1$, we define $e_k\in\mathbb{A}$ such that the complex number
\[
e_k\Big(\sum_{b_k\leq n<b_{k+1}}f(S^{n-b_k}y_k)\bfu(n) \Big)
\]
is in the closed convex cone $\{0\}\cup\{z\in\C^* : \text{arg}(z)\in[-\pi/3,\pi/3]\}$.
Then, by Lemma~\ref{l:cone}, the convergence that we need to prove, i.e.
\[
\frac{1}{b_K}\sum_{k< K}\Big|\sum_{b_k \leq n <b_{k+1}}f(S^{n-b_k}y_k)\bfu(n) \Big|\tend{K}{\infty} 0
\]
is equivalent to the convergence
\begin{equation}
 \label{eq:toShow}
 \frac{1}{b_K}\sum_{k< K}\sum_{b_k\leq n<b_{k+1}}e_kf(S^{n-b_k}y_k)\bfu(n)\tend{K}{\infty} 0.
\end{equation}
Consider the dynamical system $(X,T)$, where $X:=(Y\times \mathbb{A})^\Z$, and $T$ is the left shift.
Let $x\in X$ be such that
\[
x_n:=(S^{n-b_k}y_k,e_k),\text{ whenever }b_k\leq n<b_{k+1},
\]
and $x_n=x_0$ is a fixed arbitrary point  of $Y\times \mathbb{A}$ for $n\le 0$.

By setting $F:=f\otimes{\rm Id}$ on $Y\times \mathbb{A}$, it easily follows that~\eqref{eq:toShow} amounts to
\[
\frac{1}{b_K}\sum_{k< K}\sum_{b_k\leq n<b_{k+1}}F\circ \pi_0(T^nx)\bfu(n)\tend{K}{\infty} 0.
\]
To prove the convergence above, we define the subspace $X_x$ as the closure of $\{T^nx : n\in\Z\}$. By assumption on $\bfu$, we only
have to check that the system $(X_x,T)$ is in $\mathscr{C}_{\cf_{\rm ec}}$. So let $\mu$ be a visible measure in $(X_x,T)$, and we first consider the case where $x$ itself is generic for $\mu$, along a sequence $(N_\ell)$. Set
\[
 B:=\bigl\{(v_j,a_j)_{j\in\Z}\in X : (v_1,a_1) = (Sv_0,a_0) \bigr\}
\]
Since $b_{k+1}-b_k\to\infty$, we have
\[
\frac{1}{N_\ell}\sum_{n<N_\ell}\delta_{T^nx}(B)\tend{\ell}{\infty} 1,
\]
and since the set $B$ is closed, by the Portmanteau theorem, it must be of full measure $\mu$ in $(X,T)$. Moreover, such a measure $\mu$ must be $T$-invariant, hence,
\begin{equation}\label{where}
	1=\mu\left(\bigcap_{n\in\Z}T^nB\right)=\mu\Big( \bigl\{(v_j,a_j)_{j\in\Z}\in X : \forall j,\ (v_j,a_j) = (S^jv_0,a_0) \bigr\}   \Big).
\end{equation}
Denote by $\mu^{(0)}$ the restriction of $\mu$ to the zero-coordinate (that is, with the above notation, the distribution of $(v_0,a_0)$ under $\mu$). Since $\mu$ is $T$-invariant, it follows that $\mu^{(0)}$ is $(S\times {\rm Id}_{\mathbb{A}})$-invariant. Moreover, $Y\times \mathbb{A}$ consists of three copies of $Y$, each of them is invariant under $S\times {\rm Id}_{\mathbb{A}}$. Thus,
\[
\mu^{(0)}=\alpha_0 \mu^{(0)}_1 + \alpha_1 \mu^{(0)}_2 + \alpha_2 \mu^{(0)}_3,
\]
where $\alpha_0+\alpha_1+\alpha_2=1$, $\alpha_j\geq 0$ and $\mu^{(0)}_j(Y\times \{e^{2\pi  ij /3}\})=1$ for $j=0,1,2$. It follows that the ergodic components of $(Y\times \mathbb{A},\mu^{(0)},S\times {\rm Id}_{\mathbb{A}})$ yield measure-theoretic systems isomorphic to ergodic measures on $(Y,S)$, hence in $\cf$ since $(Y,S)\in \mathscr{C}_{\cf_{\rm ec}}$ (this is the moment in our proof where we use the fact that we deal with a characteristic ec-class and not a general characteristic class $\cf$).
Thus $(Y\times \mathbb{A},\mu^{(0)},S\times {\rm Id}_{\mathbb{A}})\in\cf_{\rm ec}$.
Now, using~\eqref{where}, we see that $(X_x,\mu,T)$ is isomorphic to  $(Y\times \mathbb{A},\mu^{(0)},S\times {\rm Id}_{\mathbb{A}})$, thus is also in $\cf_{\rm ec}$.

Now, suppose that $\mu\in V_T(x')$ for some point $x'$ in the orbit closure of $x$, say $x'=\lim_{r\to\infty}T^{n_r}x$.

If $x'=T^nx$ for some $n\in\Z$, then $\mu\in V_T(x)$ and we already know that $(X_x,\mu,T)\in\{\ast\}\subset\cf_{\rm ec}$ in this case.

If $n_r\to-\infty$, then
$x'=(\ldots,x_0,x_0,x_0,\ldots)$
is a fixed point, and $\mu=\delta_{x'}$. In this case, $(X_x,\mu,T)\in\{\ast\}\subset\cf_{\rm ec}$.

If $n_r\to +\infty$, and if we write $x'=(v_j,a_j)_{j\in\Z}=\lim_{r\to\infty}T^{n_r}x$, then as $b_{k+1}-b_k\to\infty$, there exists at most one $j\in\Z$ such that $(v_{j+1},a_j)\neq (Sv_j,a_j)$.
We can then use the same arguments as for $x$ to show that a measure $\mu$ for which $x$ is quasi-generic satisfies $(X_x,\mu,T) \in \cf_{\rm ec}$.

We conclude that $(X_x,T)$ is in $\mathscr{C}_{\cf_{\rm ec}}$.
\end{proof}

\begin{Remark}
In general, when instead of $\cf_{\rm ec}$ we consider $\cf$, $\bfu\perp\mathscr{C}_{\cf}$ implies the strong $\bfu$-MOMO property for each for $(Y,S)$ in which all invariant measures yield systems in $\cf$ (in particular, if $(Y,S)\in\mathscr{C}_{\cf}$ and each invariant measure is visible).
\end{Remark}

\begin{Question}
Is Proposition~\ref{p:mo1} true for {\bf each} characteristic class?
\end{Question}

\begin{Remark}
 \label{r:mo1subsequence}
 A straightforward adaptation of the proof shows that the subsequence version of Proposition~\ref{p:mo1} also holds:
 for each characteristic class $\cf$ and each increasing sequence of integers $(N_\ell)$, $\bfu$ is $(N_\ell)$-orthogonal to $\mathscr{C}_{\cf_{\rm ec}}$  if and only if each element in $\mathscr{C}_{\cf_{\rm ec}}$ satisfies the strong $\bfu$-MOMO property along $(N_\ell)$. See Remarks~\ref{r:ThmAsubsequence} and~\ref{r:strongMOMOsubsequence}.
\end{Remark}

\section{Lifting lemma}

The purpose of this section is to prove Proposition~\ref{p:momogen}, which is an alternative version of Conze-Downarowicz-Serafin lifting lemma from \cite{Co-Do-Se} and seems to be of independent interest. It may seem weaker than the original where the genericity was lifted to a single orbit, but the main advantage here is that we do not need assumptions on the nature of the second topological space:  it does not have to be a full shift. The second advantage is that the result has its extension to the logarithmic case, see Appendix~\ref{a:logarithmic}, while the lifting lemma of Conze-Downarowicz-Serafin and other results of that type so far have been proved for Ces\`aro averages.

\begin{Prop}
\label{p:momogen}
Let $(Y,S)$ and $(X,T)$ be two topological systems and $u\in Y$ generic along an increasing sequence $(N_m)$ for some $S$-invariant measure $\kappa$ on $Y$.  Let $\rho$ be a joining of $\kappa$ with a $T$-invariant measure $\nu$ on $X$. Then there exist a sequence $(x_n)\subset X$ and a subsequence $(N_{m_\ell})$ such that:
\begin{itemize}
\item the sequence $(S^nu,x_n)$ is generic for $\rho$ along $(N_{m_\ell})$:
\[ \frac1{N_{m_\ell}}\sum_{0\le n< N_{m_\ell}}\delta_{(S^nu,x_n)}\tend{\ell}{\infty}\rho, \]
\item The sequence $(x_n)$ is constituted of longer and longer pieces of orbits. More precisely, $\{n\ge0:\ x_{n+1}\neq Tx_n\}$ is of the form $\{b_1<b_2<\cdots<b_k<b_{k+1}<\cdots\}$, where $b_{k+1}-b_k\to\infty$.
\end{itemize}
\end{Prop}

\subsection{Good sequences of partitions}
We need a convenient tool to estimate the weak*-convergence of a sequence of probability measures to a given measure.

\begin{Def}
 Let $(E,d)$ be a compact metric space, and let $\nu$ be a Borel probability measure on $E$, i.e. $\nu\in M(E)$. We consider a sequence $(\P_\ell)$ of finite partitions of $E$ into Borel subsets. The sequence $(\P_\ell)$ is said to be \emph{good for $(E,\nu)$} if the following conditions hold:
\begin{itemize}
 \item for each $\ell$, $\P_{\ell+1}$ refines $\P_\ell$,
 \item $\diam(\P_\ell):=\max_{P\text{ atom of }\P_\ell} \diam(P) \tend{\ell}{\infty}0$,
 \item for each $\ell$ and each atom $P$ of $\P_\ell$, $\nu(\partial P)=0$.
\end{itemize}
\end{Def}

The motivation for introducing this definition comes from the following result.
\begin{Lemma}
 \label{lemma:good_convergence}
 If $(\P_\ell)$ is a good sequence of partitions for $(E,\nu)$, then a sequence $(\nu_n)\subset M(E)$ converges to $\nu$ in the weak*-topology if and only if, for each $\ell$ and each atom $P$ of $\P_\ell$, we have
 \[
  \nu_n(P) \tend{n}{\infty} \nu(P).
 \]
\end{Lemma}
\begin{proof}
 If $\nu_n\tend[\text{w}*]{n}{\infty} \nu$, then by the Portmanteau theorem, for each $P\subset E$ such that $\nu(\partial P)=0$, we have $\nu_n(P)\to\nu(P)$.

 Conversely, assume that for each $\ell$ and each atom $P$ of $\P_\ell$ we have $\nu_n(P)\to\nu(P)$. Then any weak*-limit $\mu$ of a subsequence of $(\nu_n)$ satisfies (again by the Portmanteau theorem) $\mu(P)=\nu(P)$ for each atom $P$ of $\P_\ell$. But since $\diam(\P_\ell)\to 0$, the sequence $(\P_\ell)$ separates points in $E$, hence it generates the Borel $\sigma$-algebra of $E$. Thus we have $\mu=\nu$, and using the compactness  of $M(E)$ for the weak* topology, we get that $\nu_n\tend[\text{w}*]{n}{\infty} \nu$.
\end{proof}

\begin{Lemma}
 \label{lemma:good_exist}
 For each $\nu\in M(E)$ of a compact metric space $(E,d)$, there exists a good sequence of partitions for $(E,\nu)$.
\end{Lemma}

\begin{proof}
 We first show that, for each $\ell\ge1$, there exists a finite partition $\Q_\ell$ in which each atom $Q$ satisfies
 \begin{itemize}
  \item $\diam(Q)<1/\ell$,
  \item $\nu(\partial Q)=0$.
 \end{itemize}
Indeed, by compactness, there exists a finite set $\{x_1,\ldots,x_k\}\subset E$ such that
\[ E\subset \bigcup_{1\le i\le k} B\bigl(x_i,\frac{1}{3\ell}\bigr). \]
Then, for each $1\le i\le k$, there exist at most countably many $r>0$ such that
\[
 \nu \left(\partial B(x_i,r)\right) > 0.
\]
Therefore, we can find $r\in \left(\frac{1}{3\ell},\frac{1}{2\ell}\right)$ such that
\[
 \forall 1\le i\le k,\quad \nu \left(\partial B(x_i,r)\right) = 0.
\]
Then the partition $\Q_\ell$ generated by the open balls $B(x_i,r)$, $1\le i\le k$, satisfies the required conditions.

Once we have $\Q_\ell$ for each $\ell\ge1$, we set
\[
 \P_\ell:=\Q_1\vee\cdots\vee \Q_\ell,
\]
and we get a good sequence $(\P_\ell)$ for $(E,\nu)$.
\end{proof}

\begin{Lemma}
\label{lemma:good_product}
 Let $(\P_\ell)$ be a good sequence of partitions for $(E_1,\nu_1)$, and let $(\Q_\ell)$ be a good sequence of partitions for $(E_2,\nu_2)$.
 Then for each coupling $\rho$ of $\nu_1$ and $\nu_2$, $(\P_\ell\times\Q_\ell)$ is a good sequence of partitions for $(E_1\times E_2,\rho)$.
\end{Lemma}

\begin{proof}
 This is obvious, since for each atom $P$ of $\P_\ell$ and each atom $Q$ of $\Q_\ell$,
 \[ \partial(P\times Q) \subset (\partial P\times E_2) \cup (E_1\times \partial Q), \]
 and the marginals of $\rho$ are $\nu_1$ and $\nu_2$.
\end{proof}

\subsection{Proof of Proposition~\ref{p:momogen}}

Without loss of generality, we can (and we do) assume that the measure-theoretic dynamical system $(Y,\kappa,S)$ is aperiodic. Indeed, if this is not the case, we consider any uniquely ergodic topological system $(Y',S')$ whose unique invariant measure $\kappa'$ is such that $(Y',\kappa',S')$ is aperiodic. Then we take any point $u'\in Y'$, and we replace $Y$ by $Y\times Y'$, $S$ by $S\times S'$, and $u$ by $(u,u')$. We also replace $(N_m)$ by a subsequence of $(N_m)$ along which $(u,u')$ is generic, for some measure $\tilde \kappa$ whose marginals have to be $\kappa$ and $\kappa'$. But then the system $(Y\times Y',\tilde\kappa,S\times S')$ is aperiodic, because it is an extension of the aperiodic system $(Y',\kappa',S')$.

\medskip

We fix a good sequence of partitions $(\Q_\ell)$ for $(Y,\kappa)$ and a good sequence of partitions $(\P_\ell)$ for $(X,\nu)$. Then by Lemma~\ref{lemma:good_product}, $(\Q_\ell\times \P_\ell)$ is a good sequence of partitions for $(Y\times X,\rho)$.

\medskip

\newcommand{\lz}{{\ell_0}}

\begin{Def}
 \label{def:separated}
 Let $M>0$. A subset $E$ of $\N$ is said to be \emph{$M$-separated} if for each integers $n\neq m$, $n,m\in E\Longrightarrow|n-m|\ge M$.
\end{Def}

The main argument to prove Proposition~\ref{p:momogen} stands in the following proposition.

\begin{Prop}
 \label{prop:partial}
 Under the assumptions of Proposition~\ref{p:momogen}, and assuming also that $(Y,\kappa,S)$ is aperiodic (see above), given $\lz\ge1$ and $\eps\in(0,\frac{1}{2})$, there exists a sequence $(x_n)$ of points in $X$ such that:
 \begin{itemize}
  \item $\{n\ge0: x_{n+1}\neq Tx_n\}$ is $\frac{1}{\eps}$-separated,
  \item  for each atom $A$ of $\Q_\lz\times\P_\lz$, we have
 \begin{equation}
  \label{eq:liminf}
  \rho(A)-\eps < \liminf_{m\to\infty} \frac{1}{N_m} \sum_{0\le n<N_m} \ind{A }(S^nu,x_n),
 \end{equation}
 and
 \begin{equation}
  \label{eq:limsup}
  \limsup_{m\to\infty} \frac{1}{N_m} \sum_{0\le n<N_m} \ind{A }(S^nu,x_n) < \rho(A)+\eps.
 \end{equation}
 \end{itemize}

\end{Prop}

\begin{proof}Let $h$ be a natural number such that $\frac{1}{h}<\eps$. We claim that for $\ell$ large enough, we can find a set $B\subset Y$ which is measurable with respect to $\bigvee_{0\le j\le h-1}S^j\Q_\ell$, and such that
\begin{itemize}
 \item $B,SB,\ldots,S^{h-1}B$ are pairwise disjoint, 
 \item $\kappa\left(\bigcup_{0\le j\le h-1}S^jB\right)>1-\eps$.
\end{itemize}
Indeed, since $(Y,\kappa,S)$ is assumed to be aperiodic, we can use the Rokhlin lemma to find a Borel subset $\tilde B\subset Y$ such that $\tilde B, S\tilde B,\ldots, S^{h-1}\tilde B$ are pairwise disjoint, and such that
\[ \kappa\left(\bigcup_{0\le j\le h-1}S^j\tilde B\right)>1-\frac{\eps}{2}. \]
Then we use the fact that the good sequence of partitions $(\Q_\ell)$ generates the Borel $\sigma$-algebra: it follows that for $\ell$ large enough, we can find a $\Q_\ell$-measurable set $B'$ such that
\[ \kappa(B'\bigtriangleup\tilde B)< \frac{\eps}{8h^2}. \]
For each $1\le j\le h-1$, we have
\[ B'\cap S^jB' \subset (B'\setminus\tilde B)\cup(S^jB'\setminus S^j\tilde B), \]
hence
\begin{equation}
 \label{eq:piece}\kappa(B'\cap S^jB')\le \frac{\eps}{4h^2}.
\end{equation}
It remains to define $B$ by
\[ B:= B'\setminus\left(\bigcup_{1\le j\le h-1} S^jB'\right). \]
Then, by construction, $B$ is disjoint from $S^jB$ for each $1\le j\le h-1$, thus $B,SB,\ldots,S^{h-1}B$ are pairwise disjoint. Moreover, from~\eqref{eq:piece}, we have
\[
 \kappa(B) \ge \kappa(B')-\frac{\eps}{4h}\ge\kappa(\tilde B)-\frac{\eps}{2h},
\]
which implies
\[ \kappa\left(\bigcup_{0\le j\le h-1}S^j B\right) = h\kappa(B) \ge h\kappa(\tilde B)-\frac{\eps}{2}>1-\eps, \]
and our first claim is proved.

Since $u$ is generic for $\kappa$ along $(N_m)$, and since the set $\bigcup_{0\le j\le h-1}S^jB$ is measurable with respect to $\bigvee_{0\le j\le 2h}S^j\Q_\ell$ (in particular, the $\kappa$-measure of its boundary vanishes), we have
\begin{equation}
 \label{eq:density_in_tower}
 \frac{1}{N_m} \sum_{0\le n< N_m} \ind{\bigcup_{0\le j\le h-1}S^jB} (S^n u) \tend{m}{\infty} \kappa\left(\bigcup_{0\le j\le h-1}S^jB\right) > 1-\eps.
\end{equation}
This implies in particular that the set $P_B(u):=\{n\ge0:\ S^nu\in B\}$ is infinite.
We number in order the elements of this set:
\[ P_B(u) = \{b_1<b_2<\cdots<b_k<\cdots\}\; \]
The integers $(b_k)$ will correspond to the times when we will be allowed to change the orbit of the desired sequence.
As $B$ is disjoint from $S^jB$ for each $1\le j\le h-1$, the set $P_B(u)$ is $h$-separated, hence $\frac{1}{\eps}$-separated.

We consider the partition $\bigvee_{0\le j\le h-1}(S\times T)^{-j}(\Q_\lz\times\P_\lz)$ of $Y\times X$.
Any atom of this partition is of the form $\bq\times\bp$, where $\bq$ (respectively $\bp$) is an atom of
$\bigvee_{0\le j\le h-1}S^{-j}\Q_\lz$ (respectively of $\bigvee_{0\le j\le h-1}T^{-j}\P_\lz$). For such atoms
$\bq$ and $\bp$, we can write
\begin{equation}
 \label{eq:atomq}
 \bq=Q_0\cap S^{-1}Q_1\cap\cdots\cap S^{-(h-1)}Q_{h-1},
\end{equation}
each $Q_j$ being an atom of $\Q_\lz$, and
\begin{equation}
 \label{eq:atomp}
 \bp=P_0\cap S^{-1}P_1\cap\cdots\cap S^{-(h-1)}P_{h-1},
\end{equation}
each $P_j$ being an atom of $\P_\lz$.
Since the $\kappa$-measure of the boundary of each involved set is always 0, we again have for each atom $\bq$ of
$\bigvee_{0\le j\le h-1}S^{-j}\Q_\lz$
\begin{equation}
 \label{eq:density_bbq}
 \frac{1}{N_m} \sum_{b_k< N_m} \ind{\bq} (S^{b_k} u) =
 \frac{1}{N_m} \sum_{0\le n< N_m} \ind{B\cap\bq} (S^n u)
 \tend{m}{\infty} \kappa(B\cap \bq).
\end{equation}

If $C$ is a measurable subset of $Y$ with $\kappa(C)>0$, we denote by $\rho^Y_C$ the marginal on $X$ of the conditional probability measure
$\rho(\,\cdot\,|C\times X)$. Then, for each measurable $A\subset X$, we have
\begin{align}
\label{eq:conditionnement}
 \begin{split}
 \rho(C\times A) &=\rho\bigl( (C\times X) \cap (Y\times A) \bigr) \\
    &=\rho(C\times X) \, \rho\bigl(Y\times A|C\times X\bigr)\\
    &=\kappa(C) \, \rho^Y_C(A).
\end{split}
\end{align}
 On an appropriate probability space, we construct a sequence $(\xi_k)$ of independent random variables,
 taking values in $X$, such that for each $k$, $\xi_k$ is distributed according to $\rho^Y_{B\cap\bq}$, where $\bq$ is the atom of $\bigvee_{0\le j\le h-1}S^{-j}\Q_\lz$ containing $S^{b_k}u$.

 For each atom $\bq$ of $\bigvee_{0\le j\le h-1}S^{-j}\Q_\lz$ and each atom $\bp$ of $\bigvee_{0\le j\le h-1}T^{-j}\P_\lz$, by~\eqref{eq:density_bbq}, the law of large numbers and~\eqref{eq:conditionnement}, with probability~1, we have
 \begin{equation}
  \label{eq:densitybqbp}
  \frac{1}{N_m} \sum_{b_k< N_m} \ind{\bq} (S^{b_k} u) \ind{\bp} (\xi_k) \tend{m}{\infty} \kappa(B\cap \bq) \rho^Y_{B\cap\bq}(\bp) = \rho\bigl((B\cap\bq)\times\bp\bigr).
 \end{equation}
Let us fix a realization of $(\xi_k)$ which satisfies~\eqref{eq:densitybqbp} for each atom $\bq\times\bp$ of $\bigvee_{0\le j\le h-1}(S\times T)^{-j}(\Q_\lz\times\P_\lz)$. Then, for each $n\ge0$, we define the point $x_n\in X$ as follows:
\[
 x_n := \begin{cases}
         T^{n-b_1}\xi_1 &\text{ if }n<b_1,\\
         T^{n-b_k}\xi_k &\text{ if }b_k\le n<b_{k+1}\text{ for some }k\ge1.
        \end{cases}
\]
The set of integers $n$ such that $x_{n+1}\neq Tx_n$ is contained in $P_B(u)$, therefore, it is $\frac{1}{\eps}$-separated.

Now, let $A=Q\times P$ be a fixed atom of $\Q_\lz\times\P_\lz$. We set
\[
 R:=\bigcup_{0\le j\le h-1}S^jB\times X,
\]
and we observe that
\begin{equation}
\label{eq:rhoR}
 \rho(R) = \kappa\left(\bigcup_{0\le j\le h-1}S^jB\right) >1-\eps.
\end{equation}
We also note that for each $n\ge b_1$, $(S^nu,x_n)\in R$ if and only if there exists $k$ and $0\le j\le h-1$ such that $n=b_k+j$.
In this case, $(S^nu,x_n)\in A\cap R$ if and only if the atom $\bq\times\bp$ of $\bigvee_{0\le j\le h-1}(S\times T)^{-j}(\Q_\lz\times\P_\lz)$ containing $(S^{b_k}u,\xi_k)$ satisfies $Q_j=Q$ and $P_j=P$ (using the notations given in~\eqref{eq:atomq} and~\eqref{eq:atomp}, and remembering that $A=Q\times P$).

We can then use~\eqref{eq:densitybqbp} to get
\begin{align}
\label{eq:densityAR}
 \begin{split}
  &\frac{1}{N_m} \sum_{b_1\le n<N_m} \ind{A\cap R}(S^nu,x_n) \\
  &= \frac{1}{N_m} \sum_{b_k<N_m} \sum_{0\le j\le h-1} \sum_{(\bq,\bp):\ Q_j=Q\text{ and }P_j=P} \ind{\bq\times\bp}(S^{b_k}u,\xi_k)\\
  &\tend{m}{\infty} \sum_{0\le j\le h-1} \sum_{(\bq,\bp):\ Q_j=Q\text{ and }P_j=P} \rho\bigl((B\cap\bq)\times\bp\bigr).
 \end{split}
\end{align}

But, on the other hand, we can write
\begin{align}
 \label{eq:rhoAR}
 \begin{split}
 \rho(A\cap R) &= \sum_{0\le j\le h-1} \rho\bigl(A\cap(S^jB\times X)\bigr)\\
 &= \sum_{0\le j\le h-1} \rho\bigl((S^{-j}Q\times T^{-j}P)\cap(B\times X)\bigr)\\
 &= \sum_{0\le j\le h-1} \sum_{(\bq,\bp):\ Q_j=Q\text{ and }P_j=P} \rho\bigl((B\cap\bq)\times\bp\bigr).
 \end{split}
\end{align}
From~\eqref{eq:densityAR} and~\eqref{eq:rhoAR}, it follows that
\begin{equation}
 \label{eq:densityAR2}
 \frac{1}{N_m} \sum_{0\le n<N_m} \ind{A\cap R}(S^nu,x_n) \tend{m}{\infty} \rho(A\cap R).
\end{equation}

From~\eqref{eq:density_in_tower}, we get that
\[
 \lim_{m\to\infty}\frac{1}{N_m} \sum_{0\le n<N_m} \ind{(Y\times X)\setminus R}(S^nu,x_n)<\eps,
\]
and since $\ind{A}\le\ind{A\cap R}+\ind{Y\times X\setminus R}$, this yields by~\eqref{eq:densityAR2},
\begin{align*}
 \limsup_{m\to\infty} \frac{1}{N_m} \sum_{0\le n<N_m} \ind{A }(S^nu,x_n)
 &< \lim_{m\to\infty} \frac{1}{N_m} \sum_{0\le n<N_m} \ind{A\cap R}(S^nu,x_n) + \eps \\
 &=\rho(A\cap R)+ \eps\\
 &\le\rho(A)+\eps,
\end{align*}
and we have~\eqref{eq:limsup}.
On the other hand, using $\ind{A}\ge\ind{A\cap R}$, we get by~\eqref{eq:rhoR}
\begin{align*}
 \liminf_{m\to\infty} \frac{1}{N_m} \sum_{0\le n<N_m} \ind{A }(S^nu,x_n)
 &\ge \lim_{m\to\infty} \frac{1}{N_m} \sum_{0\le n<N_m} \ind{A\cap R}(S^nu,x_n) \\
 &=\rho(A\cap R) \\
 &>\rho(A)-\eps.
\end{align*}
and we have~\eqref{eq:liminf}.
\end{proof}

We can now give the proof of Proposition~\ref{p:momogen}, in which we use the following obvious fact: if we modify the sequence $(x_n)$ given by Proposition~\ref{prop:partial} on a finite number of terms, we still get~\eqref{eq:liminf} and \eqref{eq:limsup}.

\begin{proof}[Proof of Proposition~\eqref{p:momogen}]
 We fix a sequence $(\eps_\ell)_{\ell\ge1}$ of numbers in $(0,\frac{1}{2})$, decreasing to 0, and we construct inductively the desired sequence $(x_n)$ and the subsequence $(N_{m_\ell})$ by a repeated use of Proposition~\ref{prop:partial}.

 We start by applying Proposition~\ref{prop:partial} with $\eps:=\eps_1$ and $\lz:=1$. It provides us with an integer $m_1$, and a finite sequence $(x_n)_{0\le n<N_{m_1}}$ of points in $X$ such that
\begin{itemize}
 \item the set of integers $n\in\{0,\ldots,N_{m_1}-2\}$ such that $x_{n+1}\neq Tx_n$ is $\frac{1}{\eps_1}$-separated,
 \item for each atom $A$ of $\Q_1\times\P_1$, we have
 \[
  \rho(A)-\eps_1<\frac{1}{N_{m_1}} \sum_{0\le n<N_{m_1}} \ind{A}(S^nu,x_n) < \rho(A)+\eps_1.
 \]
\end{itemize}

Now, assume that for some $\ell\ge1$ we have already constructed $m_1<\cdots<m_\ell$ and the sequence $(x_n)_{0\le n<N_{m_\ell}}$ of points in $X$ such that
\begin{itemize}
 \item for each $1\le j<\ell$, the set of integers $n\in\{N_{m_{j-1}},\ldots,N_{m_j}-2\}$ such that $x_{n+1}\neq Tx_n$ is $\frac{1}{\eps_j}$-separated (with the convention that $N_{m_0}=0$),
 \item for each atom $A$ of $\Q_\ell\times\P_\ell$, we have
 \[
  \rho(A)-\eps_\ell<\frac{1}{N_{m_\ell}} \sum_{0\le n<N_{m_\ell}} \ind{A}(S^nu,x_n) < \rho(A)+\eps_\ell.
 \]
\end{itemize}
Then we apply again Proposition~\ref{prop:partial}, with $\eps:=\eps_{\ell+1}$ and $\lz:=\ell+1$. It provides us with an integer $m_{\ell+1}$ and a finite sequence of points $(x_n)_{N_{m_\ell}\le n<N_{m_{\ell+1}}}$ in $X$ which satisfy:
\begin{itemize}
 \item the set of integers $n\in\{N_{m_\ell},\ldots,N_{m_{\ell+1}}-2\}$ such that $x_{n+1}\neq Tx_n$ is $\frac{1}{\eps_{\ell+1}}$-separated,
 \item for each atom $A$ of $\Q_{\ell+1}\times\P_{\ell+1}$, we have
 \begin{equation}
  \label{eq:induction}
  \rho(A)-\eps_{\ell+1}<\frac{1}{N_{m_{\ell+1}}} \sum_{0\le n<N_{m_{\ell+1}}} \ind{A}(S^nu,x_n) < \rho(A)+\eps_{\ell+1}.
 \end{equation}
\end{itemize}
(We keep the points $(x_n)_{0\le n<N_{m_\ell}}$ already provided by the induction hypothesis, refering to the obvious fact stated before the proof.)

Moreover, we can assume that the sequence $(\eps_\ell)$ decreases sufficiently fast so that the validity of~\eqref{eq:induction} for each atom $A$ of $\Q_{\ell+1}\times\P_{\ell+1}$ ensures the validity of the analog inequalities for each $A$ which is a finite union of atoms of $\Q_{\ell+1}\times\P_{\ell+1}$ (in particular, for each $A$ which is an atom of the previous partitions), but with $\eps_\ell$ instead of $\eps_{\ell+1}$.

The sequence $(x_n)_{n\ge0}$ of points in $X$ and the subsequence $(N_{m_\ell})$ we construct with the above inductive procedure then satisfy the conditions announced in Proposition~\ref{p:momogen}.
\end{proof}

\subsection{Logarithmic case}\label{a:logarithmic}
We would like to study the logarithmic version of Proposition~\ref{p:momogen}, in which we replace each arithmetic average of the form
\[
 \frac{1}{N_m}\sum_{0\le n< N_m} f(n)
\]
by the logarithmic average
\[
 \frac{1}{L(N_m)}\sum_{1\le n\le N_m} \frac{1}{n}f(n).
\]
(Here we use the notation $L(N):=1+\frac{1}{2}+\cdots+\frac{1}{N}$.) In fact, this logarithmic version, whose statement is written below, is also valid, and the arguments to prove it are exactly the same as in the arithmetic average case. We just point out below the few technical changes that need to be made in the proof for the logarithmic case.

\begin{Prop}
\label{prop:log_case}
Let $(Y,S)$ and $(X,T)$ be two topological systems and $u\in Y$, logarithmically generic along an increasing sequence $(N_m)$ for some $S$-invariant measure $\kappa$ on $Y$.  Let $\rho$ be a joining of $\kappa$ with a $T$-invariant measure $\nu$ on $X$. Then there exist a sequence $(x_n)\subset X$ and a subsequence $(N_{m_\ell})$ such that:
\begin{itemize}
\item the sequence $(S^nu,x_n)$ is logarithmically generic for $\rho$ along $(N_{m_\ell})$:
\[ \frac1{L(N_{m_\ell})}\sum_{1\le n\le N_{m_\ell}}\frac{1}{n}\delta_{(S^nu,x_n)}\tend{\ell}{\infty}\rho, \]
\item the set $\{n\ge0:\ x_{n+1}\neq Tx_n\}$ is of the form $\{b_1<b_2<\cdots<b_k<b_{k+1}<\cdots\}$, where $b_{k+1}-b_k\to\infty$.
\end{itemize}
\end{Prop}

The changes that need to be made to the proof are almost all quite obvious, they consist in formally replacing the arithmetic average by the logarithmic average. One point maybe needs some explanations, namely when we arrive at the proof of the logarithmic analog of~\eqref{eq:densitybqbp}.
We put these explanations in the form of a lemma, which we will apply in the following context: $(d_k)$ is the ordered sequence of positive integers $n$ such that
$S^nu\in B\cap {\bar Q}$, and the sequence $(\rho_k)$ is defined by $\rho_k:=\ind{{\bar P}}(\xi_k)$.

\begin{Lemma}
 \label{lemma:log_case}
 Let $(d_k)$ be an increasing sequence of positive integers such that
 \[
  \frac1{L(N_{m})}\sum_{d_k\le N_{m}}\frac{1}{d_k} \tend{m}{\infty} \kappa\in[0,1],
 \]
and let $(\rho_k)$ be a sequence of real numbers in $[0,1]$ such that
\[
 \frac{1}{K} \sum_{1\le k\le K} \rho_k \tend{K}{\infty} \rho\in[0,1].
\]
Then we have
 \[
  \frac1{L(N_{m})}\sum_{d_k\le N_{m}}\frac{\rho_k}{d_k} \tend{m}{\infty} \kappa\rho.
 \]
\end{Lemma}

\begin{proof}
For each $m$, let us denote by $k_m$ the largest $k$ such that $d_k\le N_m$.
 We use the classical identity
 \[
  \sum_{d_k\le N_{m}} \frac{\rho_k}{d_k} = \sum_{1\le k < k_m}  \left( \frac{1}{d_k}-\frac{1}{d_{k+1}} \right) (\rho_1+\cdots+\rho_k) + \frac{1}{d_{k_m}}(\rho_1+\cdots+\rho_{k_m}).
 \]
 Given $\eps>0$, let $K_\eps$ be such that
 \[ K\ge K_\eps \Longrightarrow \left| \frac{1}{K} \sum_{1\le k\le K} \rho_k - \rho \right| < \eps. \]
 We can then write
 \begin{align*}
  &\left| \frac1{L(N_{m})} \sum_{d_k\le N_{m}} \frac{\rho_k}{d_k} -  \frac1{L(N_{m})} \sum_{d_k\le N_{m}} \frac{\rho}{d_k}\right|\\
  &= \left| \frac1{L(N_{m})} \sum_{K_\eps\le k < k_m} k \left( \frac{1}{d_k}-\frac{1}{d_{k+1}} \right)  \left( \frac{1}{k} (\rho_1+\cdots+\rho_k)-\rho \right) \right| + O\left(\frac1{L(N_{m})}\right)\\
  &< \eps + O\left(\frac1{L(N_{m})}\right).
 \end{align*}
But by assumption, we have
\[
  \frac1{L(N_{m})} \sum_{d_k\le N_{m}} \frac{\rho}{d_k}\tend{m}{\infty} \kappa\rho,
\]
hence we get
\[
  \frac1{L(N_{m})} \sum_{d_k\le N_{m}} \frac{\rho_k}{d_k}\tend{m}{\infty} \kappa\rho.
\]
\end{proof}

The last place in the proof where a (very easy) correction should be made in the logarithmic case is to get the analog of~\eqref{eq:densityAR}: at some point we have to replace some coefficients
$\frac{1}{b_k+j}$ by $\frac{1}{b_k}$, which is of no consequence since $j$ remains bounded between 0 and $h-1$ here. To be more precise, \eqref{eq:densityAR} becomes
\begin{align}
\label{eq:densityARlog}
 \begin{split}
  &\frac{1}{L(N_m)} \sum_{b_1\le n<N_m} \frac{1}{n}\ind{A\cap R}(S^nu,x_n) \\
  &= \frac{1}{L(N_m)} \sum_{b_k<N_m} \sum_{0\le j\le h-1} \frac{1}{b_k+j}\sum_{(\bq,\bp):\ Q_j=Q\text{ and }P_j=P} \ind{\bq\times\bp}(S^{b_k}u,\xi_k)\\
  &= \frac{1}{L(N_m)} \sum_{b_k<N_m} \sum_{0\le j\le h-1} \frac{1}{b_k}\sum_{(\bq,\bp):\ Q_j=Q\text{ and }P_j=P} \ind{\bq\times\bp}(S^{b_k}u,\xi_k) + o(1)\\
  &\tend{m}{\infty} \sum_{0\le j\le h-1} \sum_{(\bq,\bp):\ Q_j=Q\text{ and }P_j=P} \rho\bigl((B\cap\bq)\times\bp\bigr).
 \end{split}
\end{align}

\section{Proofs of main theorems} \label{s:cztery}

\subsection{Proof of Theorem~\ref{tA}}
\label{sec:prooftA}

\begin{proof}[Proof of Theorem~\ref{tA}]
Take any topological system $(X,T)\in\mathscr{C}_{\cf }$ and fix $f\in C(X)$, $x\in X$.
Take any increasing sequence $(N_k)$ for which, with no loss of generality, we can assume that $\frac1{N_k}\sum_{n\leq N_k} \delta_{(T^nx,S^n\bfu)}\to\rho$. It follows that
$$
\lim_{k\to\infty}\frac1{N_k}\sum_{n\leq N_k}f(T^nx)\bfu(n)=
\int f\ot \pi_0\,d\rho.
$$
But $\rho$ is a joining of some $T$-invariant measure $\nu\in V(X,T)$ for which $x$ is generic along $(N_k)$, and some Furstenberg system $\kappa$ of $\bfu$. Since $(X,T)\in\mathscr{C}_{\cf }$, the system $(X,\nu,T)$ is in $\cf$, and the integral on the right-hand side above vanishes by the Veech condition and Proposition~\ref{p:sated}.

\end{proof}

\subsection{Proof of Theorem~\ref{tB}}
\label{sec:prooftB}
Before we begin the proof, let us make the following remark concerning topological models. Given an automorphism $(Z,\mathcal{D},\kappa, R)$, and a fixed subset  of full measure of ergodic components of $\kappa$, recall that by a {\em Hansel model} of $R$, we mean a  topological system $(X,T)$ which has a $T$-invariant measure $\nu$ such that, as dynamical systems, $\kappa$ and $\nu$ are isomorphic and such that each point $x\in X$ is generic for one of these chosen ergodic components.
In \cite{Ha}, it is proved that each automorphism has a Hansel model.

We assume that $\bfu\perp\mathscr{C}_{\cf_{\rm ec}}$ for some characteristic class $\cf$.
Take $\kappa\in V(\bfu)$ and fix $(N_m)$ so that $$\frac1{N_m}\sum_{n\leq N_m} \delta_{S^n\bfu}\to\kappa.$$
Denote by $\ca(\kappa)\subset\cb(X_{\bfu})$ the largest $\cf_{\rm ec}$-factor of $(X_{\bfu},\kappa,S)$, i.e.\ $\ca(\kappa)=\cb(X_{\bfu})_{\cf_{\rm ec}}$.

Consider the factor $(X_{\bfu}/\ca(\kappa), \ca(\kappa), \kappa|_{\ca(\kappa)},S)$ and take a Hansel model $(X,\nu,T)$ of it (by choosing in the ergodic decomposition of $\kappa|_{\ca(\kappa)}$ only ergodic measures in $\cf$).
By definition,
\beq\label{goodmod00}
(X,T)\in \mathscr{C}_{\cf_{\rm ec}}.
\eeq

Fix a measure-theoretic factor map $J\colon(X_{\bfu},\kappa,S)\to(X,\nu,T)$ such that $J^{-1}(\cb(X))=\ca(\kappa)$, and let $\nu_J$ denote the corresponding graph joining (of $\nu$ and $\kappa|_{\ca(\kappa)}$). Let $\widehat{\nu}_J$ be the relatively independent extension of $\nu_J$ to a joining of $\nu$ and $\kappa$: for $f\in L^2(\nu)$ and $g\in L^2(\kappa)$, we have
\begin{equation}
 \label{eq:rij}
 \int_{X_{\bfu}\times X} g\otimes f \, d\widehat{\nu}_J = \int_{X_{\bfu}} \E^{\kappa}\bigl(g | \ca(\kappa)\bigr)\,f\circ J \, d\kappa.
\end{equation}
Now, by applying Proposition~\ref{p:momogen}, we can find
$(x_n)\subset X$   such that
\begin{equation}\label{qugen00}
((x_n),\bfu) \text{ is generic for }\widehat{\nu}_J\text{ along some subsequence }(N_{m_\ell}),
\end{equation}
and the set $\{n\ge0:\ x_{n+1}\neq Tx_n\}$ is of the form $\{b_1<b_2<\cdots\}$ with $b_{k+1}-b_k\to\infty$.

Since $\bfu\perp\mathscr{C}_{\cf_{\rm ec}}$, \eqref{goodmod00} and Proposition~\ref{p:mo1} ensure that the system $(X,T)$ satisfies the  strong $\bfu$-MOMO property. Therefore, for each $f\in C(X)$ we have
$$
\lim_{m\to\infty}\frac1{N_m}\sum_{n\leq N_m}f(x_n)\bfu(n)
=$$$$\lim_{m\to\infty}\frac1{b_{K_m}}\sum_{k<K_m}\left(
\sum_{b_k\leq n<b_{k+1}}f(T^{n-b_k}x_{b_k})\bfu(n)\right)
=0,
$$
and it follows from~\eqref{qugen00} that $\int f\otimes \pi_0\, d \hat{\nu}_J=0$. Using~\eqref{eq:rij}, we get
\[
\int_{X_{\bfu}} \E^{\kappa}\bigl(\pi_0 | \ca(\kappa)\bigr)\,f\circ J \, d\kappa=0.
\]
But $\{f\circ J:\ f\in C(X)\}$ is dense in $L^2(\ca(\kappa))$ and therefore $\pi_0\perp L^2(\ca(\kappa))$,
which is the Veech condition for  $\bfu$ with respect to the characteristic class~$\cf_{\rm ec}$.
}

\section{Combinatorics}
\subsection{Orthogonality to zero entropy systems}
\label{s:wnioski}
\subsubsection{Cancellations. Proof of Corollaries~\ref{cA} and~\ref{cB}}

We need the following interpretation of the Veech condition in terms of relative uniform mixing (K-mixing) of the function $\pi_0$. For $n\in\N$, let $\pi_n:=\pi_0\circ S^n$.

\begin{Prop}\label{prop:mich}
Let $\kappa\in V_S(\bfu)$ and $\int \pi_0\,d\kappa=0$. Then the following conditions are equivalent:
\begin{enumerate}
\item[(a)]
$\pi_0\perp L^2(\Pi(\kappa))$,
\item[(b)]
$\pi_0$ is relatively K-mixing, i.e.
for each $\vep>0$, there exists $N$ such that
\begin{equation*}\label{relK}
\left|\int \pi_0 \raz_C\,d\kappa -\int \pi_0\, d\kappa\int \raz_C d\kappa\right|=\left| \int \pi_0 \raz_C\,d\kappa\right|<\vep
\end{equation*}
for each set $C\in \sigma(\pi_n,\pi_{n+1},...)$  and $n\geq N$.
\end{enumerate}
If we additionally assume that $\bfu$ takes values in a finite set $E\subset \C$ and $(M_k)$ is a sequence along which we have a Furstenberg system $\kappa$ then the above conditions are equivalent to
\begin{enumerate}
\item[(c)]
for each $\vep>0$ there exists $N\geq1$ such that for any $s\geq1$ and any function $f$
depending on coordinates $N\leq n,n+1,\ldots,n+s$, $\|f\|_{C(X_{\bfu})}\leq1$, we have
$$
\limsup_{k\to\infty}\left|\frac1{M_k}\sum_{m\leq M_k}\bfu(m)f(S^m\bfu)\right|<\vep.$$
\end{enumerate}
\end{Prop}

\begin{proof}
(a) $\implies$ (b). Assume that $\EE(\pi_0|\Pi(\kappa))=0$. Let $C\in \sigma(\pi_n,\pi_{n+1},\ldots)$. We have
\begin{multline*}
\left|\int \pi_0 \raz_C d\kappa\right|=\left|\int \raz_C \EE(\pi_0|\sigma(\pi_n,\pi_{n+1},...)\, d\kappa\right|\\
\leq \int\left|\EE(\pi_0|\sigma(\pi_n,\pi_{n+1},\ldots))\right|\, d\kappa.
\end{multline*}
Hence, we have an upper bound which does not depend on $C$. Since
\[
\EE(\pi_0|\sigma(\pi_n,\pi_{n+1},...))\to \EE(\pi_0|\Pi(\kappa))=0
\]
$\kappa$-a.e. and thus also in $L^1$, which is precisely the relative K-mixing for $\pi_0$.

(b) $\implies$ (a). Suppose that $\pi_0$ is relatively K-mixing. Then,
in particular, we have~\eqref{relK} for each $C\in \Pi(\kappa)$. In fact, since $\vep>0$ is  arbitrary,
$\int \pi_0 \raz_C\,d\kappa=0$ for each $C\in\Pi(\kappa)$. Whence   $\EE(\pi_0|\Pi(\kappa))=0$.

(a) $\implies$ (c). Since
$$
\frac1{M_k}\sum_{m\leq M_k}\bfu(m)f(S^m\bfu)=\frac1{M_k}\sum_{m\leq M_k}(\pi_0f)(S^m\bfu)\to \int_{X_{\bfu}} \pi_0 f\, d\kappa,
$$
we can repeat the same argument as was used to prove (a) $\implies$ (b) (replacing $\raz_C$ by $f$).

(c) $\implies$ (b). Suppose that $\left|\int_{X_{\bfu}}\pi_0 f\,d\kappa\right|<\vep$, whenever $f$ depending on coordinates $n,n+1,\ldots,n+s$ with $n\geq N$ is bounded by~1. Consider all blocks on coordinates $n,n+1,\ldots,n+s$ that is all $$B=\{x\in X_{\bfu}:\: x_n=b_0,\ldots,x_{n+s}=b_s\}
$$
with $b_j\in E$. Let $C$ be any union of such blocks. Then $\raz_C$ is a (continuous) function depending on coordinates $n,\ldots,n+s$ and is bounded by~1 and, by assumption,
$$\left|\int_{X_{\bfu}}\pi_0 \raz_C\,d\kappa\right|<\vep.$$
Note that with $N$ fixed and $s$ arbitrary, the family of $C$ defined above is dense in the $\sigma$-algebra $\sigma(\pi_N,\pi_{N+1},\ldots)$. Hence, given $D\in \sigma(\pi_N,\pi_{N+1},\ldots)$ and $\vep>0$, we first find $s\geq0$ and then $C$ as above (a union of blocks ``sitting'' on coordinates $N,\ldots,N+s$) such that $\kappa(C\triangle D)<\vep$ and find that
$$
\left|\int \pi_0 \raz_D\,d\kappa\right|\leq
\left|\int \pi_0 \raz_C\,d\kappa\right|+\kappa(C\triangle D)<2\vep.
$$
\end{proof}
Now, since each clopen set is a finite union of cylinders of a fixed length, Corollary~\ref{cB} follows directly by the above proposition. Corollary~\ref{cB} is a special case of Corollary~\ref{cA}.

\subsubsection{Conditional cancellations. Remark~\ref{r:sarnakcomb}}
The ``cancellation law'' of the values of $\bfu$ along large shifts of the return times to a block (for most of the blocks) claimed in Remark~\ref{r:sarnakcomb} is a consequence of a refinement of Proposition~\ref{prop:mich}.

\begin{Prop} Let $\kappa\in V_S(\bfu)$ and $\int \pi_0\,d\kappa=0$. Then the following conditions are equivalent:
\begin{enumerate}
\item[(a)]
$\pi_0\perp L^2(\Pi(\kappa))$,
\item[(d)]
for each $\vep>0$ there exist $N\geq1$ and $L\geq1$ such that for each $\ell\geq L$ there is a family  of ``good'' $\ell$-blocks $C\in\sigma(\pi_N,\pi_{N+1},\ldots)$, i.e. of blocks satisfying
$$
\left|\int \raz_C\cdot\pi_0\,d\kappa\right|\leq \kappa(C)\vep,$$
whose measure is $>1-\vep$. In other words, for a ``good'' $\ell$-block $C$, $\left|\int \pi_0\,d\kappa_C\right|<\vep$, where $\kappa_C$ stands for the conditional measure on $C$.
\end{enumerate}
\end{Prop}
\begin{proof}
(a) $\implies$ (d). Fix $\vep>0$ and note that $\EE^\kappa(\pi_0|\sigma(\pi_N,\pi_{N+1},\ldots))\to0$ $\kappa$-a.e. This implies convergence in measure, i.e., we can find a set $A_\vep$ of measure at least $1-\vep$ such that for $N$ large enough,
$$
|\EE^\kappa(\pi_0|\sigma(\pi_N,\pi_{N+1},\ldots))(x)|<\vep\text{ for all }x\in A_{\vep}.$$
Fix such an $N$. There is $M\geq 1$ large enough such that
$$
\kappa(A_{\vep}\triangle A^{(M)}_{\vep})<\vep,$$
where $A^{(M)}_\vep\in\sigma(\pi_{-M},\pi_{-M+1},\ldots)$
and note that $S^{N+M}A^{(M)}_\vep\in \sigma(\pi_N,\pi_{N+1},\ldots)$. Now, for $\ell$ large enough, we can approximate $S^{N+M}A^{(M)}_{\vep}$ by a (disjoint) union of $\ell$-blocks belonging to $\sigma(\pi_N,\pi_{N+1},\ldots)$,
$$
\kappa(S^{N+M}A^{(M)}_{\vep}\triangle \bigcup_{j\in J}  C^{(\ell)}_j)<\vep.$$
But $\kappa(S^{N+M}A^{(M)}_{\vep}\triangle A_{\vep})<2\vep$, so
$$
\kappa(\bigcup_{j\in J}  C^{(\ell)}_j \setminus A_{\vep})\leq \kappa(A_{\vep}\triangle \bigcup_{j\in J}  C^{(\ell)}_j)<3\vep.
$$
Consider those $j\in J$ for which $\kappa(C^{(\ell)}_j\setminus A_{\vep})\geq\sqrt{\vep}\kappa(C_j^{(\ell)})$. Then the measure $m$ of the union of such blocks has to satisfy $\sqrt{\vep} m<3\vep$, so $m<3\sqrt{\vep}$. In other words ``most'' (in measure) of the $C_j^{(\ell)}$'s are ``good'', i.e. they satisfy $\kappa(C^{(\ell)}_j\cap A_{\vep})>(1-3\sqrt{\vep})\kappa(C^{(\ell)}_j)$. Take such a ``good'' $C$. We have
\begin{multline*}
\left|\int \raz_C\cdot \pi_0\,d\kappa\right|\leq \int\raz_{C}|\EE^\kappa(\pi_0|\sigma(\pi_N,\pi_{N+1},\ldots))|
\,d\kappa\\
=\int_{A_{\vep}}\raz_{C}|\EE^\kappa(\pi_0|
\sigma(\pi_N,\pi_{N+1},\ldots))|\,d\kappa
+
\int_{A_{\vep}^c}\raz_{C}|\EE^\kappa(\pi_0|\sigma(\pi_N,\pi_{N+1},
\ldots))|\,d\kappa\\
\leq \vep \kappa(C)+3\sqrt{\vep}\kappa(C).
\end{multline*}

(d) $\implies$ (a).  Fix $A\in\Pi(\kappa)$ of positive measure $\kappa$. Then for $\vep>0$, we can find $N$ such that for all $\ell$ large enough ``most'' of the $\ell$-blocks in $\sigma(\pi_N,\pi_{N+1}<\ldots)$ is ``good'' in the sense that
$$\left|\int \raz_C\cdot\pi_0\,d\kappa\right|\leq \kappa(C)\vep.
$$
Since $A\in \sigma(\pi_N,\pi_{N+1},\ldots)$, we can approximate it by unions of $\ell$-blocks (for $\ell$ sufficiently large) and most of the used blocks is ``good''. Whence
$$
\left|\int \raz_A\cdot\pi_0\,d\kappa\right|\leq 2\vep,$$
and since $\vep>0$ was arbitrary, $\pi_0\perp L^2(\Pi(\kappa))$.
\end{proof}

\subsection{Orthogonality to $\mathscr{C}_{\rm {ID}}$. Proof of Corollary~\ref{c:ID}}  We recall that (Proposition~\ref{p:examples})
$$
\mathscr{C}_{\rm ID}=\mathscr{C}_{{\rm ID}_{\rm ec}}.$$
Since the characteristic factor is represented by the $\sigma$-algebra of invariant sets, by Theorems~\ref{tA} and~\ref{tB}, we obtain immediately that:
\begin{Cor}\label{t:bhu} $\bfu\perp \mathscr{C}_{\rm ID}$ if and only if for  each Furstenberg system $\kappa$ of $\bfu$, $\pi_0\perp L^2(\mathcal{I}_\kappa)$.\end{Cor}

Let us now pass to a combinatorial characterization of the Veech condition.
Assume that $\kappa$ is given as the limit of $\frac1{N_k}\sum_{n\leq N_k}\delta_{S^n\bfu}$. In view of Corollary~\ref{t:bhu}, we need to decipher $\EE(\pi_0|\mathcal{I}_\kappa)=0$.
By the von Neumann theorem, we have
$$
\frac1H\sum_{h\leq H}\pi_0\circ S^h\to\EE(\pi_0|\mathcal{I}_\kappa) \text{ in }L^2,
$$
i.e.
$$
\lim_{H\to\infty}\int_{X_{\bfu}}\left|\frac1H\sum_{h\leq H}\pi_0\circ S^h\right|^2\,d\kappa=0
$$
as $\EE(\pi_0|\mathcal{I}_\kappa)=0$. So, given $\vep>0$,
$$
\int_{X_{\bfu}}\left| \frac1H\sum_{h\leq H}\pi_0\circ S^h\right|^2\,d\kappa<\vep\text{ for all }H\geq H_{\vep}.$$
The latter is equivalent to
$$
\lim_{k\to\infty}\frac1{N_k}\sum_{n\leq N_k}\left|\frac1H\sum_{h\leq H}\pi_0\circ S^h\right|^2(S^n\bfu)<\vep,$$
that is,
$$\lim_{k\to\infty}\frac1{N_k}\sum_{n\leq N_k}\left|\frac1H\sum_{h\leq H}\bfu(n+h)\right|^2<\vep.$$
The proof of Corollary~\ref{c:ID} follows immediately.\bez

\begin{Remark}\label{r:mrtID}  Hence, the Matom\"aki-Radziwi\l\l \ theorem \cite{Ma-Ra} on the behaviour of a strongly aperiodic multiplicative function $\bfu$ on a typical short interval implies $\bfu\perp \mathscr{C}_{\rm ID}$. However, as shown in \cite{Go-Le-Ru1}, the  aperiodic multiplicative functions defined in \cite{Ma-Ra-Ta} do not satisfy the assertion of Corollary~\ref{c:ID}.\end{Remark}

In Corollary~\ref{t:bhu}, the Veech condition (for $\bfu$)  equivalent to $\bfu\perp \mathscr{C}_{\rm ID}$
is written as $\pi_0\perp L^2(\mathcal{I}_\kappa)$. If we look at it more spectrally, we obtain immediately that
\beq\label{vcID}
\bfu\perp \mathscr{C}_{\rm ID}\text{ if and only if }\sigma_{\pi_0,\kappa}(\{1\})=0\eeq
for all $\kappa\in V_S(\bfu)$,
i.e.\ the spectral measure of $\pi_0$ (with respect to each Furstenberg system) has no atom at $1$.  Classically (by a simple computation), we have:
\begin{Lemma}\label{l:cl1} If  $\sigma$ is a measure on the circle $\bs^1$ then
$$
\frac1H\sum_{h=0}^{H-1} \hat{\sigma}(h)\to\sigma(\{1\}).$$
\end{Lemma}

Hence, the Veech condition  is equivalent to
$$\frac1H\sum_{h=0}^{H-1}\int \pi_0\cdot\ov{\pi_0}\circ S^h d\kappa \to 0.$$
Combinatorially, we obtain
\beq\label{cl2}
\frac1H\sum_{h=0}^{H-1}\lim_{k\to\infty} \frac1{N_k}\sum_{n\leq N_k} \bfu(n)\ov{\bfu}(n+h) \to 0
\eeq
for each sequence $(N_k)$ defining a Furstenberg system $\kappa$.

It follows that \eqref{cl2} is equivalent to the short interval behaviour~\eqref{ortmsvf}. In other words, condition
$$\lim_{H\to\infty}\left(\lim_{k\to\infty}\frac1{N_k}\sum_{n\leq N_k}\left|\frac1H\sum_{h\leq H} \bfu(n+h)\right|^2\right)=0$$
is equivalent to
$$\lim_{H\to\infty}\frac1H\sum_{h\leq H}\left(\lim_{k\to\infty}\frac1{N_k}\sum_{n\leq N_k}\bfu(n)\ov{\bfu}(n+h)=0\right).$$

\subsection{Orthogonality to $\mathscr{C}_{\rm DISP(G)}$ with $G$ countable}
Let $G\subset\bs^1$ be a countable subgroup and recall that
${{\rm DISP}(G)}$ stands for the (characteristic) class of discrete spectrum automorphisms whose groups of eigenvalues are contained in $G$. Since $z\in\bs^1$ is an eigenvalue of $(Z,\cd,\kappa,R)$ if and only if it is an eigenvalues of a subset of positive measure of ergodic components, it is not hard to see that
$$
\cf_{{\rm DISP}(G)}=(\cf_{{\rm DISP}(G)})_{\rm ec}.$$
It follows that
$$\bfu \perp\mathscr{C}_{\cf_{{\rm DISP}(G)}}
\text{ if and only if }
\sigma_{\pi_0,\kappa}(G)=0,$$
i.e.\ the spectral measure of $\pi_0$ has no atoms belonging to $G$ (for each Furstenberg system $\kappa\in V_S(\bfu)$).

Suppose that $e^{2\pi i\alpha}\in G$. Consider
$\bfv(n):=e^{2\pi in\alpha}\bfu(n)$ for $n\in\N$. Note that
$$
\frac1{N_k}\sum_{n\leq N_k}\bfv(n)\ov{\bfv}(n+h)=e^{-2\pi ih\alpha}\frac1{N_k}\sum_{n\leq N_k}\bfu(n)\ov{\bfu}(n+h).$$
So, if we have a subsequence $(N_k)$ along which both $\frac1{N_k}\sum_{n\leq N_k}\delta_{S^n\bfw}$ with $\bfw=\bfu,\bfv$ converge to $\kappa,\kappa'$ respectively,\footnote{Note that these common sequences yield {\bf all} Furstenberg systems for both $\bfu$ and $\bfv$.} then
$$
\sigma_{\pi_0,\kappa}=\delta_{e^{2\pi i\alpha}}\ast\sigma_{\pi_0,\kappa'},$$
whence
$$
\sigma_{\pi_0,\kappa}(\{e^{2\pi i\alpha}\})=0 \text{ if and only if }
\sigma_{\pi_0,\kappa'}(\{1\}).$$
By our previous subsection it follows that the latter condition is equivalent to:
$$\lim_{H\to\infty}\left(\lim_{k\to\infty}\frac1{N_k}\sum_{n\leq N_k}\left|\frac1H\sum_{h\leq H} \bfv(n+h)\right|^2\right)=0,$$
that is,
$$\lim_{H\to\infty}\left(\lim_{k\to\infty}\frac1{N_k}\sum_{n\leq N_k}\left|\frac1H\sum_{h\leq H}e^{2\pi i(n+h)\alpha} \bfu(n+h)\right|^2\right)=0$$
which is the strong $\bfu$-MOMO condition for the irrational rotation by $\alpha$.\footnote{Note that if $f(t)=e^{2\pi i t}$ then
$$
\frac1{b_K}\sum_{k<K}\left|\sum_{b_k\leq n<b_{k+1}}f(R_\alpha^nx_k)\bfu(n)\right| =\frac1{b_K}\sum_{k<K}\left|\sum_{b_k\leq n<b_{k+1}}e^{2\pi i n\alpha}\bfu(n)\right|.$$}

\subsection{Furstenberg systems and the strong $\bfu$-MOMO property} \label{s:FSMOMO}
The following proposition helps us to exclude some measure-theoretic systems from the list of Furstenberg systems of an arithmetic function.
\begin{Prop}\label{p:franz}
Let $\bfu:\N\to\C$ be a bounded arithmetic function. Then no Furstenberg system $\kappa\in V_S(\bfu)$ has a topological model which is strongly $\bfu$-MOMO.
\end{Prop}
\begin{proof}Suppose  $(X_{\bfu},\kappa, S)$ has a topological model $(Z,\nu,R)$ which satisfies the strong $\bfu$-MOMO property. Let $J:Z\to X_{\bfu}$ settles a  measure-theoretic isomorphism and let $\nu_J$ be the corresponding graph joining. We assume that $\frac1{N_j}\sum_{n\leq N_j}\delta_{S^n\bfu}\to\kappa$. From Proposition~\ref{p:momogen} we can find a sequence $(z_n)\subset Z$ consisting of pieces of orbits of different points:
$\{n:\: Rz_n\neq z_{n+1}\}=\{b_k:\:k\geq1\}$ with $b_{k+1}-b_k\to\infty$,  and a subsequence $(N_{j_\ell})$ such that
$$
\frac1{N_{j_\ell}}\sum_{n\leq N_{j_\ell}}\delta_{(z_n,S^n\bfu)}\to\nu_J.$$
Then, by the strong $\bfu$-MOMO property of $(Z,R)$,
$$
\int f\ot \pi_0\,d\nu_J=\lim_{\ell\to\infty}\frac1{b_{K_\ell}}\sum_{k<K_\ell}\big(
\sum_{b_k\leq n<b_{k+1}}f(R^{n-b_k}z_{b_k})\bfu(n)\big)=0.$$
Hence, $\int \EE^{\nu_J}(f|X_{\bfu})\pi_0\,d\kappa=0$ for each continuous $f$ on $Z$, and we obtain a contradiction since $\EE^{\nu_{J}}(L^2(\nu)|X_{\bfu})=L^2(\kappa)$.
\end{proof}

\begin{Cor}\label{c:franz}
Assume that for each $(b_k)$ with $b_{k+1}-b_k\to\infty$,
\beq\label{momo28}
\lim_{K\to\infty}\frac1{b_K}\sum_{k<K}\sup_{\alpha\in\R}
\left|\sum_{b_k\leq n< b_{k+1}}\bfu(n)e^{2\pi i\alpha n}\right|=0.\eeq
Then the unipotent system $(x,y)\mapsto (x,y+x)$ (on $\T^2$) is not a Furstenberg system of $\bfu$.
\end{Cor}
\begin{proof}
Since condition \eqref{momo28} is the strong $\bfu$-MOMO property of the unipotent system, the result follows from Proposition~\ref{p:franz}.
\end{proof}

\begin{Remark} Corollary \ref{c:franz} brings a better understanding  of Problem 3.1 (due to Frantzikinakis) of the workshop \cite{AIM}:\\
{\em The system $(x,y)\mapsto(x,y+x)$ is not a Furstenberg system of the Liouville function}\\
(see also slide no 6 in \cite{nikos_talk}).

We recall that in \cite{Ma-Ra-Ta} (see Theorem 1.3 therein), it is proved that
$$
\lim_{M\to\infty}\limsup_{N\to\infty}\sup_{\alpha\in\R}\frac1N\sum_{n\leq N}\left|\frac1M\sum_{n\leq m<n+M}\lio(m)e^{2\pi i m\alpha}\right|=0,$$
so the sup has changed the place! The strong $\lio$-MOMO property for the unipotent system remains hence open.

For the equivalence of
$$
\lim_{M\to\infty}\limsup_{N\to\infty}\frac1N\sum_{n\leq N}\sup_{\alpha\in\R}\left|\frac1M\sum_{n\leq m<n+M}\bfu(m)e^{2\pi i m\alpha}\right|=0,$$
with~\eqref{momo28} see the appendix in \cite{Ka-Le-Ul} - only in the arXiv version of the paper.
\end{Remark}

\subsection{Orthogonality to $\mathscr{C}_{{\rm DISP}_{\rm ec}}$. Proof of Corollary~\ref{c:averagedCh'}}\label{s:orthogonality}
In view of Corollary~\ref{pods} (see also \eqref{opo}) and Theorem~\ref{tB},
in order to obtain $\bfu\perp\mathcal{C}_{{\rm DISP}_{ec}}$ it is sufficient and necessary to have $\bfu \perp L^2(\K(\mathcal{I}_\kappa))$ for each Furstenberg system $\kappa$ of $\bfu$.

By our  previous results, for the class of all topological systems whose all {\bf ergodic} measures yield discrete spectra, Sarnak and Veech conditions  are equivalent. We now write the Veech condition combinatorially, i.e.,\ we provide the proof of Corollary~\ref{c:averagedCh'}.

\begin{proof}[Proof of Corollary~\ref{c:averagedCh'}]
By Corollary~\ref{pods}, we need to show that for each $\kappa$ being a Furstenberg system of $\bfu$, we have
\begin{equation}\label{goal}
\int\frac1H\sum_{h\leq H}|\E(\pi_0\circ S^h\cdot\overline{\pi_0}|\mathcal{I}_\kappa)|^2\, d\kappa \to 0.
\end{equation}
By the von Neumann theorem,
$$
\int\Big|\EE(\pi_0\circ S^h\cdot \ov{\pi}_0|\mathcal{I}_\kappa)\Big|^2\,d\kappa=
\lim_{N\to\infty}\frac1N\sum_{n\leq N}\int(\pi_0\circ S^h\cdot \ov{\pi}_0)\circ S^n\cdot \ov{\pi_0\circ S^h\cdot \ov{\pi}_0}\,d\kappa.$$
Therefore,~\eqref{goal} is equivalent to
$$
\lim_{H\to\infty}\frac1H\sum_{h\leq H}\lim_{N\to\infty}\frac1N\sum_{n\leq N}\int(\pi_0\circ S^h\cdot \ov{\pi}_0)\circ S^n\cdot \ov{\pi_0\circ S^h\cdot \ov{\pi}_0}\,d\kappa=0.
$$
Let $(M_k)$ be such that $\kappa=\lim_{k\to\infty}\frac1{M_k}\sum_{m\leq M_k}\delta_{S^m\bfu}$. It follows immediately that~\eqref{goal} is equivalent to
$$
\lim_{H\to\infty}\frac1H\sum_{h\leq H}\lim_{N\to\infty}\frac1N\sum_{n\leq N}\lim_{k\to\infty}\frac1{M_k}\sum_{m\leq M_k}\bfu(m+n+h)\ov{\bfu(m+n)}\ov{\bfu(m+h)}\bfu(m)=0$$
which is precisely $\|\bfu \|_{u^2((M_k))}=0$. Now, it suffices to use~\eqref{prop13fubook}.
\end{proof}

\begin{Remark} \label{r:remark56}
In fact, already Frantzikinakis~\cite{nikos_talk} (see slide no 10) showed that if $\bfu$ is generic then $\|\bfu\|_{u^2}=0$ if and only if
$$
\lim_{M\to\infty}\limsup_{N\to\infty}\frac1N\sum_{n\leq N}\sup_{\alpha\in\R}\left|\frac1M\sum_{n\leq m<n+M}\bfu(m)e^{2\pi i m\alpha}\right|=0.
$$
We recall that this condition is equivalent to the strong $\bfu$-MOMO property of the unipotent system $(x,y)\mapsto (x, y+x)$.

\end{Remark}

\begin{Remark}
Note that for each (bounded) $\bfu\colon\N\to\C$ satisfying $\|\bfu\|_{u^2}=0$  the system
    \beq\label{unip}
    (x,y)\mapsto (x,x+y)\text{ on }(\T^2,{\rm Leb}\,\ot\, {\rm Leb})\eeq
    cannot appear (up to isomorphism) as a Furstenberg system of $\bfu$ (because $\pi_0$ is orthogonal to the $L^2(\K(\mathcal{I}_\kappa))$ but for the unipotent system \eqref{unip} the
whole system is relative Kronecker over the $\sigma$-algebra of invariant sets).

In particular, if $\|\lio\|_{u^2}=0$ holds for the Liouville function then \eqref{unip} is not its Furstenberg system -- this would answer a question by N.\ Frantzikinakis asked in 2016 (it is an official Problem 3.1 in \cite{AIM}). However, the problem of whether $\|\lio\|_{u^2}=0$ (or more generally $\|\lio\|_{u^s}=0$) seems to be difficult. The best known results  \cite{Ma-Ra-Ta2,Ma-Ra-Ta-Te-Zi} require a quantitative dependence between the parameters $M$ and $N$, i.e.\ $M=N^{\theta}$, for arbitrary small, but fixed $\theta>0$.

If $\|\lio\|_{u^2}=0$ holds  then Sarnak's conjecture holds for all (zero entropy) systems whose {\bf ergodic} measures yield discrete spectrum. So far it is only known that Sarnak's conjecture holds for systems whose {\bf all} invariant measure yield discrete spectrum  \cite{Hu-Wa-Zh,Hu-Wa-Ye,Fe-Ku-Le}.

Ruling out~\eqref{unip} (or, more generally, nilpotent type systems) from the list of potential Furstenberg systems of $\lio$ is important in view of Frantzikinakis and Host's results \cite{Fr-Ho1,Fr-Ho2} concerning the structure of Furstenberg systems of multiplicative functions (although, for the moment, this structure is known only for the logarithmic case).

In the light of \cite{Ma-Ra-Ta}, it would be also interesting to know whether $\|\bfu\|_{u^2}=0$ holds for some classical multiplicative functions. Note that this is {\bf not} the case for the class of aperiodic multiplicative functions defined in \cite{Ma-Ra-Ta} since as shown in \cite{Go-Le-Ru1}  they have the unipotent system as a Furstenberg system\footnote{In fact, for such functions $\bfu$ we have already $\|\bfu\|_{u^1((N_k))}>0$ (for some $(N_k)$), see Corollary~6.5 in \cite{Go-Le-Ru1}.} (see also Remark~\ref{r:mrtID}).
\end{Remark}

\subsection{Orthogonality to $\mathscr{C}_{\rm DISP}$. Averaged Chowla property for multiplicative functions} \label{s:avCh} The assertion ``iff''  of Theorem~\ref{tB} cannot be applied  to the class $\mathscr{C}_{\rm DISP}$. In this section we will show however that the assertion of this theorem holds whenever $\bfu\colon\N\to\C$
satisfies the following  additional property:
\begin{equation}\tag{$\ast$}
\mbox{all rotations on the circle satisfy the strong $\bfu$-MOMO property}.
\end{equation}
We will need the following fact (see, e.g.,~\cite{Ed}):
\beq
\label{modeldsc} \mbox{each discrete spectrum automorphism is a factor of $R_\alpha\times {\rm Id}_{[0,1]}$},
\eeq
for some ergodic rotation by $\alpha\in G$ on a compact (Abelian) metric group $G$. Our key tool will be the following lemma.
\begin{Lemma}\label{hasmomo}
Suppose that $(\ast)$ holds. Then $R_\alpha\times {\rm Id}_{[0,1]}$ satisfies the strong $\bfu$-MOMO property.
\end{Lemma}
\begin{proof}
It is enough to check the strong $\bfu$-MOMO for functions $F$ of the form $\chi\ot f$, where $\chi\in\widehat{G}$ and $f\in C([0,1])$. We have
\begin{align*}
\frac1{b_K}\sum_{k<K}&\left| \sum_{b_k\leq n<b_{k+1}}F((R_\alpha\times {\rm Id})^n(h_k,u_k))\bfu(n)\right|\\
&=\frac1{b_K}\sum_{k<K}\left| \sum_{b_k\leq n<b_{k+1}}\chi((R_\alpha^n(h_k))f(u_k)\bfu(n)\right|\\
&={\rm O}\left( \frac1{b_K}\sum_{k<K}\left| \sum_{b_k\leq n<b_{k+1}}\chi(n\alpha)\bfu(n)\right| \right).
\end{align*}
Our claim follows from $(\ast)$.
\end{proof}

\begin{Th}\label{t:am} Assume that $\bfu$ enjoys the property $(\ast)$. Then
 $\bfu\perp  \mathscr{C}_{DISP}$ if and only if $\pi_0\perp L^2(\mathcal{K}(\kappa))$ for each $\kappa\in V_S(\bfu)$  (iff the spectral measure $\sigma_{\pi_0}$ is continuous for each Furstenberg system $\kappa$).
 \end{Th}
\begin{proof}
We only need to show that $\bfu\perp  \mathscr{C}_{DISP}$ implies $\pi_0\perp L^2(\mathcal{K}(\kappa))$ for each $\kappa\in V_S(\bfu)$.\footnote{$\mathcal{K}(\kappa)$ stands for the Kronecker factor of $(X_{\bfu},\kappa,S)$.} Using \eqref{modeldsc}, let
$p$ settle a factor map from $R_\alpha\times {\rm Id}_{[0,1]}$ acting on  $(G\times [0,1],m_G \ot {\rm Leb})$ and $(X_{\bfu}/\mathcal{K}(\kappa),
\mathcal{K}(\kappa),\kappa|_{\mathcal{K}(\kappa)})$. Let $(m_G\ot {\rm Leb})_p$ stand for the corresponding graph joining and $\rho$ for the relatively independent extension of it to a joining of $(G\times [0,1], m_G\ot {\rm Leb},R_\alpha\times {\rm Id})$ with $(X_{\bfu}, \kappa,S)$.

Now, by Proposition~\ref{p:momogen}, the integral $\int F\ot \pi_0\, d\rho$ can be computed using a quasi-generic sequence $((g_n), (S^n\bfu))$. Since, by Lemma~\ref{hasmomo}, our topological system $R_\alpha\times {\rm Id}$ satisfies the strong $\bfu$-MOMO property, this integral vanishes. On the other hand, for each $F\in C(G\times [0,1])$,
$$
\int F\ot \pi_0\,d\rho=\int \EE(F|X_{\bfu})\pi_0\,d\kappa$$
and since $\EE^\rho(C(G\times [0,1])|X_{\bfu})$ is dense in $L^2(\mathcal{K},\kappa|_{\mathcal{K}})$ (in view of the definition of $\rho$), it follows that $\pi_0\perp L^2(\mathcal{K}(\kappa))$.
\end{proof}

\begin{proof}[Proof of Corollary~\ref{c:averagedCh}]
Note that in the proof of Theorem~\ref{t:am},  we have shown that our original assumption $(\ast)$ already implies the Veech condition. In particular,  the Sarnak and the Veech properties are equivalent. Condition~\eqref{avch1} is just rewriting the Wiener condition combinatorially. Finally, the last part~\eqref{avch2} is proved in Appendix~\ref{a:drugi}.
\end{proof}

\begin{proof}[Proof of Corollary~\ref{c:averagedCh2}]
By Corollary~\ref{c:averagedCh}, we only need to show that irrational rotations satisfy the strong $\bfu$-MOMO property. This follows from the fact that irrational rotations satisfy the AOP property \cite{Ab-Le-Ru} and that the AOP property implies the strong $\bfu$-MOMO property \cite{Ab-Ku-Le-Ru2}.
\end{proof}

\section{No strong $\bfu$-MOMO in positive entropy} \label{s:momo}
In this section we discuss  the problem of orthogonality to $\mathscr{C}_{\rm ZE}$ and the reversed problem of the {\bf absence} of orthogonality to an {\bf arbitrary} positive entropy systems, following some ideas from \cite{Ab-Ku-Le-Ru2}. Recall that the following has been proved in \cite{Ab-Ku-Le-Ru2}.
\begin{Prop}[\cite{Ab-Ku-Le-Ru2}]
Let $\bfu\colon\N\to\C$ be a bounded arithmetic function. The following are equivalent:
\begin{enumerate}
\item[(a)] $\bfu\perp \mathscr{C}_{\rm ZE}$.
\item[(b)] For each $(X,T)$ of zero entropy and $f\in C(X)$, \eqref{ort1} holds uniformly in $x\in X$.
\item[(c)] Each zero entropy $(X,T)$ satisfies the strong $\bfu$-MOMO property.
\end{enumerate}
\end{Prop}

On the other hand, it follows from the results of Downarowicz and Serafin~\cite{Do-Se,Do-Se1} that for each $\bfu\perp \mathscr{C}_{\rm ZE}$ {\bf there exists} $(X,T)$ such that
\beq\label{Do-Se}
\bfu\perp (X,T) \text{ and }(X,T)\not\in {\rm ZE}.
\eeq
In fact, one can get a positive entropy system $(X,T)$ in which for every $f\in C(X)$ \eqref{ort1} holds uniformly in $x\in X$.

We prove however that \eqref{Do-Se} fails if orthogonality is replaced by the strong $\bfu$-MOMO property.
To avoid technical details, we restrict ourselves to the case of an arithmetic function $\bfu$ taking finitely many values.

\begin{Thx}\label{t:thmC}
Let $\bfu:\N\to\C$ be an arithmetic function taking  finitely many values.
Assume that $\bfu\perp \mathscr{C}_{\rm ZE}$. Then no positive entropy topological dynamical system satisfies the strong $\bfu$-MOMO property.
\end{Thx}

\subsection{Proof of Theorem~\ref{t:thmC}}
We fix a bounded arithmetic function $\bfu\colon\N\to\C$.
We need a series of results from \cite{Ab-Ku-Le-Ru2} in some modified forms. In~\cite{Ab-Ku-Le-Ru2}, the equivalence of certain three properties (P1), (P2) and (P3) of an ergodic measure-theoretic dynamical system $(Z,\cb(Z),\kappa,R)$ was proved. Condition (P1) was nothing but the strong $\bfu$-MOMO for {\bf some} topological system being a model of the system given by $\kappa$. Instead of recalling (P2), let us formulate red its subsequence version:
\beq\label{czy1}\begin{array}{l}\tag{P2'}
\mbox{Assume that $(X,T)$ is any topological system and let $x\in X$.}\\
\mbox{If $x$ is generic along $(N_k)$ for a measure  which is isomorphic}\\\mbox{(as dynamical systems) to $\kappa$ then}\\
\mbox{$\lim_{k\to\infty}\frac1{N_k}\sum_{n\leq N_k}f(T^nx)\bfu(n)=0$
for each $f\in C(X)$}.\end{array}\eeq
The proof of the implication (P1) $\implies$ (P2') is a repetition of the proof of (P1) implies (P2). In Lemma~17 in \cite{Ab-Ku-Le-Ru2}, we need to consider the sequence $(N_k)$ instead of $\N$ and start with $\limsup$ along this sequence.

As a consequence of the above, we obtain the following version of Corollary~12 from~\cite{Ab-Ku-Le-Ru2}.
\begin{Cor}\label{czy2}
Assume that $\kappa$ is an ergodic shift-invariant measure on $\D_L^{\Z}$, and that there exists $y\in \D_L^{\Z}$, generic along $(N_k)$ for  $\kappa$, correlating with $\bfu$ along $(N_k)$, \textit{i.e.} the sequence $(\frac1{N_k}\sum_{n\leq N_k}y(n)\bfu(n))$ does not go to zero. Then the strong  $\bfu$-MOMO property fails for any uniquely ergodic model of $(\D_L^{\Z},\kappa,S)$.
\end{Cor}

Then, by repeating the proof from~\cite{Ab-Ku-Le-Ru2}, we obtain the following form of Corollary~14 in~\cite{Ab-Ku-Le-Ru2}.
\begin{Cor}\label{czy3}
Assume that $y$ is generic along $(N_k)$ for a Bernoulli measure $\nu$, and that $y$ and $\bfu$ correlate along $(N_k)$. Then the strong $\bfu$-MOMO property fails for any $(X,T)$ with $h(X,T)>h(\nu)$.
\end{Cor}

We also need the following crucial probabilistic lemma whose proof we postpone to the next subsection.

\begin{Lemma}
\label{l:omo2}
Assume that $X=(X_n)_{n\in\Z}$ is a a stationary process of positive entropy, taking finitely many complex values. Then for any non-trivial probability distribution $\beta$ concentrated on a finite subset of $\R$, there exists a stationary coupling of $X$ with a Bernoulli process $Y=(Y_n)_{n\in\Z}$ of distribution $\beta^{\otimes\Z}$ such that $\E[X_0 Y_0]\neq\E[X_0]\E[Y_0]$.
\end{Lemma}

We now assume that $\bfu$ takes finitely many values and satisfies the Veech condition: $\pi_0\perp L^2(\Pi(\kappa))$ for each Furstenberg system $\kappa$ of $\bfu$.

\begin{Lemma}\label{l:omo3} For each $h>0$ there exists a sequence $y$, generic for a Bernoulli measure of entropy $h$ along some increasing sequence $(N_k)$, and correlating with $\bfu$ along $(N_k)$.\end{Lemma}
\begin{proof}
Let $\kappa$ be a Furstenberg system of $\bfu$, and $(M_\ell)$ such that $\bfu$ is generic for $\kappa$ along $(M_\ell)$. By assumption, the entropy of the stationary process defined by $\pi_0$ under $\kappa$ is positive.
Take a real-valued Bernoulli shift of entropy $h$ (Bernoulli measure denoted by $\nu$). Using Lemma~\ref{l:omo2}, find a joining of $\kappa$ and $\nu$ for which $\pi_0$  (in $L^2(X_{\bfu},\kappa)$) is not orthogonal to $\pi_0$ in $L^2(\nu)$:
$\int \pi_0\ot\pi_0\,d\rho\neq0$. Now, use a subsequence version of the lifting lemma (Theorem~5.16 in~\cite{Be-Do-Va}) to find $y$ in the subshift defining the Bernoulli automorphism such that $(\bfu,y)$ is generic, along a subsequence $(N_k)=(M_{\ell_k})$, for $\rho$. Then
$$
0\neq \int \pi_0\ot\pi_0\,d\rho=\lim_{k\to\infty}\frac1{N_{k}}\sum_{n\leq N_{k}}\pi_0(S^n\bfu)\pi_0(S^ny)=\lim_{k\to\infty}\frac1{N_{k}}\sum_{n\leq N_{k}}\bfu(n)y_n$$
which means that $\bfu$ and $y$ correlate along $(N_{k})$.
\end{proof}

Now the proof of Theorem~\ref{t:thmC} is a straightforward consequence of Lemma~\ref{l:omo3} and Corollary~\ref{czy3}.

\subsection{Proof of Lemma~\ref{l:omo2}}

\newcommand{\law}{\mathcal{L}}
\newcommand{\U}{\mathcal{U}}

Let $X=(X_n)_{n\in\Z}$ be a positive entropy stationary process as in the statement of the lemma. Without loss of generality (considering its real or imaginary part), we can assume that this process takes its values in a finite subset $\{x_1<x_2<\cdots <x_r\}$ of $\R$. We also consider a given probability measure $\beta$ supported on a possibly different finite subset of $\R$ $\{y_1<y_2<\cdots <y_s\}$, which is supposed to be non trivial (i.e.\ not reduced to a Dirac measure). Thus we can assume that $s\geq2$, and $\beta(y_j)>0$ for each $1\leq j\leq s$. The purpose of this section is to show how we can construct a stationary coupling of $X$ with a Bernoulli process $Y$ whose distribution is $\beta^{\otimes \Z}$, in such a way that for each $n\in\Z$,
\begin{equation}
 \label{eq:positive_correlation}
 \E[X_n Y_n] > \E[X_n]\, \E[Y_n].
\end{equation}
We observe that the validity of the preceding inequality is unchanged if we replace $Y_n$ by $Y_n+C$ for a fixed $C$. Thus we can and we do assume without loss of generality that the probability $\beta$ is such that $\E[Y_n]=0$.

To construct the announced coupling, we just assume that, on the probability space where the process $X$ is defined, we also have an i.i.d.\ process $V=(V_n)_{n\in\Z}$ such that
\begin{itemize}
 \item each $V_n$ is uniformly distributed on $[0,1]$,
 \item $V$ is independent of $X$.
\end{itemize}
The construction will be divided into two steps: first we construct an auxiliary (uniform i.i.d.) process $U$ and then we use it to construct $Y$ which satisfies the assertion of Lemma~\ref{l:omo2}.

\subsubsection*{Step 1: uniform i.i.d.\ process $U$}

For $n\in\Z$ and $j\in\{1,\ldots,r\}$, we consider the random variable $P_{j,n}$ defined by
\[
 P_{j,n} := \PP\bigl( X_n=x_j \,|\, (X_m)_{m\leq n-1}\bigr).
\]
When $j$ is fixed, $(P_{j,n})_{n\in\Z}$ is a stationary process. On the other hand, if we fix $n$, then $(P_{1,n},\ldots,P_{r,n})$ is the conditional distribution of $X_n$ given $(X_m)_{m\leq n-1}$, in particular we have almost surely $0\leq P_{j,n}\leq 1$, and
\[
 \sum_{j=1}^r P_{j,n} = 1.
\]
This allows us to define a random partition of $[0,1[$ into disjoint subintervals $I_{1,n},\ldots, I_{r,n}$ where for each $j$, $I_{j,n}$ is the interval of length $P_{j,n}$ defined by
\[
 I_{j,n}:=\left[\sum_{1\leq i\leq j-1} P_{i,n}\ ; \sum_{1\leq i\leq j} P_{i,n}\right[.
 \]

 Then we can define the random variable $U_n$ by
 \[
  U_n := \sum_{j=1}^r \ind{X_n=x_j}\left(\sum_{1\leq i\leq j-1} P_{i,n} + V_n P_{j,n}\right).
 \]
\begin{figure}
 \begin{center}
  \includegraphics[width=0.6\linewidth]{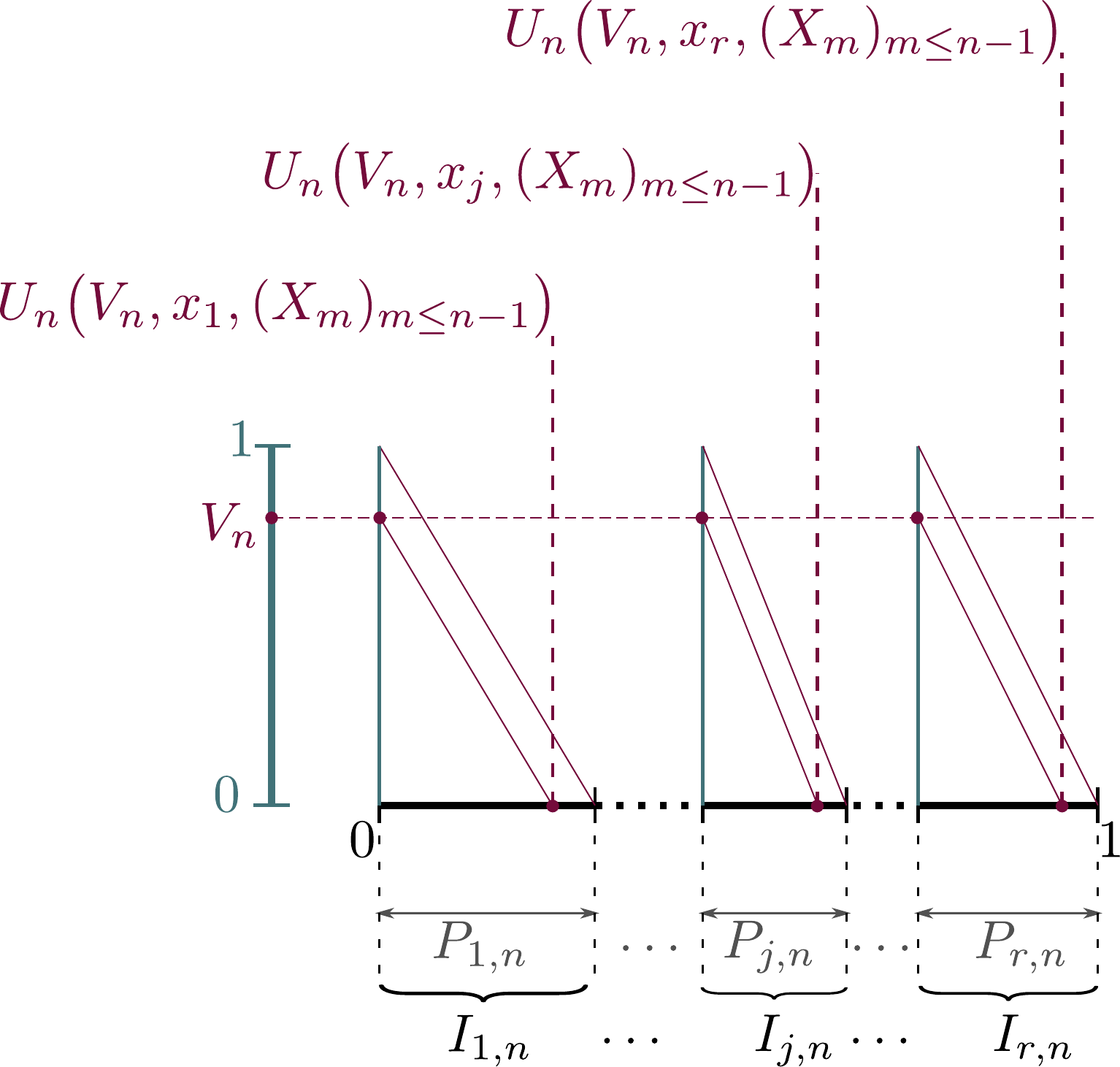}
  \caption{Definition of $U_n$}
  \label{fig:intervals}
\end{center}
\end{figure}
 Informally, if $X_n=x_j$, we pick $U_n$ uniformly at random (using $V_n$) inside $I_{j,n}$ (see Figure~\ref{fig:intervals}). Therefore,
\[
 \law\left(U_n\,|\,(X_m)_{m\leq n-1},(V_m)_{m\leq n-1}\right)=\U_{[0,1]},
\]
i.e., it is uniform on $[0,1]$. But all $U_m$, $m\leq n-1$, are measurable with respect to $(X_m)_{m\leq n-1}$ and $(V_m)_{m\leq n-1}$, thus we also have
 \begin{equation}\label{ti}
 \law\left(U_n\,|\,(U_m)_{m\leq n-1}\right)=\U_{[0,1]} \text{ and } \law\left(U_n\right)=\U_{[0,1]}.
\end{equation}
 Indeed, this is just the application of the tower property of conditional expectations: to obtain the left equality, notice that for any measurable $A\subset [0,1]$, we have
\begin{align*}
\mathbb{P}\bigl(U_n\in A &\,|\, (U_m)_{m\leq n-1}\bigl) \\
&=\mathbb{E}\Bigl[
\underbrace{\mathbb{P}\bigl(U_n\in A\,|\, (X_m)_{m\leq n-1},(V_m)_{m\leq n-1}\bigr)}_{{\rm Leb}(A)}
                            \,|\,  (U_m)_{m\leq n-1}\Bigr]\\
&={\rm Leb}(A).
\end{align*}
 Moreover, it also follows from~\eqref{ti} that $U$ is i.i.d.

Note that by construction, $U_n$ is a measurable function of $V_n$, $X_n$ and $(X_m)_{m\le n-1}$, which we abusively write as
\[
U_n = U_n \left(V_n,X_n,(X_m)_{m\le n-1}\right).
\]
Moreover, whenever we fix realizations $\xi$ of $(X_m)_{m\leq n}$ and $v$ of $V_n$ then $U_n$ as a function of its second argument is \emph{increasing}:
\begin{equation}\label{rem:increasing}
U_n(v,x_{j_1},\xi) < U_n(v,x_{j_2},\xi), \text{ whenever }x_{j_1}<x_{j_2}.
\end{equation}

\subsubsection*{Step 2: process $Y$ as a function of $U$}

We want to define $Y_n$ for a given $n\in\Z$. We use another partition of $[0,1[$ into subintervals, according to the probability distribution $\beta$ intended for $Y_n$: for $1\le k\le s$, set $\beta_k:=\beta(y_k)$ and define the interval $J_k:=\bigl[\beta_1+\cdots+\beta_{k-1};\beta_1+\cdots+\beta_k\bigr[$.
Then we simply define $Y_n$ as a function of $U_n$ by setting
\[ Y_n := \sum_{k=1}^s  y_k\,\ind{J_k}(U_n).
\]
It follows by the choice of the intervals $J_k$ and by $\law(U_n)=\U_{[0,1]}$ that $Y_n$ is distributed according to $\beta$. Moreover, by the independence of $U$, we have the independence of $Y$. Thus, $Y$ is a Bernoulli process with distribution $\beta^{\otimes \mathbb{Z}}$.

It remains to prove the announced inequality~\eqref{eq:positive_correlation}.
Observe that $Y_n$ is, like $U_n$, constructed as a measurable function of $V_n$, $X_n$ and $(X_m)_{m\leq n-1}$, which we also abusively write as
\[ Y_n = Y_n \left(V_n,X_n,(X_m)_{m\le n-1}\right). \]
Since $Y_n$ is a non-decreasing function of $U_n$, we get from~\eqref{rem:increasing} that for a fixed realization $\xi$ of $(X_m)_{m\leq n-1}$ and $v$ of $V_n$, we have for $1\le j_1 < j_2 \le r$
\begin{equation}
 \label{eq:Yn}
 Y_n \left(v,x_{j_1},\xi\right) < Y_n \left(v,x_{j_2},\xi\right)
\end{equation}
and it follows that the map
\[ x\in A \mapsto \E \bigl[ Y_n \,|\, (X_m)_{m\le n-1}=\xi , X_n=x\bigr] \]
is non-decreasing. Moreover, by the construction of $Y$, we have
\[ \law\bigl(Y_n \,|\, (X_m)_{m\le n-1}\bigr) = \law(Y_n)=\beta, \]
whence
\begin{equation}\label{gwi}
\E \bigl[ Y_n \,|\, (X_m)_{m\le n-1}\bigr] = \E[Y_n] =0.
\end{equation}
Thus there exists $j_0\in\{1,\ldots,r\}$ (depending on $\xi$) such that
\begin{equation}
\begin{split}\label{gwi2}
 &\E \bigl[ Y_n \,|\, (X_m)_{m\le n-1}=\xi , X_n=x_j]\leq 0 \text{ for }1\leq j\leq j_0,\\
 &\E \bigl[ Y_n \,|\, (X_m)_{m\le n-1}=\xi , X_n=x_j]> 0\text{ for }j_0+1\leq j\leq r.
 \end{split}
\end{equation}
We then have, using~\eqref{gwi} and~\eqref{gwi2},
\begin{align}
 \E \bigl[ X_n Y_n \,|\, (X_m)_{m\le n-1}=\xi\bigr] & =  \E \bigl[ (X_n-x_{j_0}) Y_n \,|\, (X_m)_{m\le n-1}=\xi\bigr]  \nonumber \\
   &= \sum_{j=1}^{j_0} (x_j-x_{j_0})\, \E \bigl[ Y_n \,|\, (X_m)_{m\le n-1}=\xi , X_n=x_j] \nonumber \\
   &\quad + \sum_{j=j_0+1}^{r} (x_j-x_{j_0})\, \E \bigl[ Y_n \,|\, (X_m)_{m\le n-1}=\xi , X_n=x_j] \label{eq:sum}\\
   & \geq0. \nonumber
\end{align}

Now, we claim that the announced result is a consequence of the following lemma.
\begin{Lemma}
 \label{lemma:non-trivial-distribution}
 If the realization $\xi$ of $(X_m)_{m\leq n-1}$ is such that the conditional distribution $\law(X_n \,|\, (X_m)_{m\le n-1}=\xi)$ is non-trivial, then
 \[ \E \bigl[ X_n Y_n \,|\, (X_m)_{m\le n-1}=\xi\bigr] > 0. \]
\end{Lemma}

Indeed, since $X$ has positive entropy, $\law(X_n \,|\, (X_m)_{m\le n-1})$ is non-trivial with positive probability, and thus we can conclude that
\[ \E\bigl[ X_n Y_n\bigr] = \E\Bigl[\E \bigl[ X_n Y_n \,|\, (X_m)_{m\le n-1}\bigr]\Bigr] >0. \]

\begin{proof}[Proof of Lemma~\ref{lemma:non-trivial-distribution}]
We fix a realization $\xi$ of $(X_m)_{m\leq n-1}$ such that the conditional distribution $\law\bigl(X_n \,|\, (X_m)_{m\le n-1}=\xi\bigr)$ is non-trivial. Then the random variables $P_{j,n}$ and the intervals $I_{j,n}$ are fixed, because their values only depend on $\xi$. Setting
\begin{align*}
 j_1 &:= \min\bigl\{j\in\{1,\ldots,r\}:\ P_{j,n}>0\bigr\},\\
 \text{and}\quad j_2 &:= \max\bigl\{j\in\{1,\ldots,r\}:\ P_{j,n}>0\bigr\},
\end{align*}
we have $j_1<j_2$.
 Moreover the intervals $I_{j_1,n}$ and $I_{j_2,n}$ are respectively of the form $[0,P_{j_1,n}[$ and $[1-P_{j_2,n},1[$, with $0<P_{j_1,n}\leq 1-P_{j_2,n}<1$.
We now discuss according to the relative position of the interval $I_{j_2,n}$ with respect to the interval $J_1$ (used to define $Y_n$).

\begin{figure}
 \begin{center}
  \includegraphics[width=0.6\linewidth]{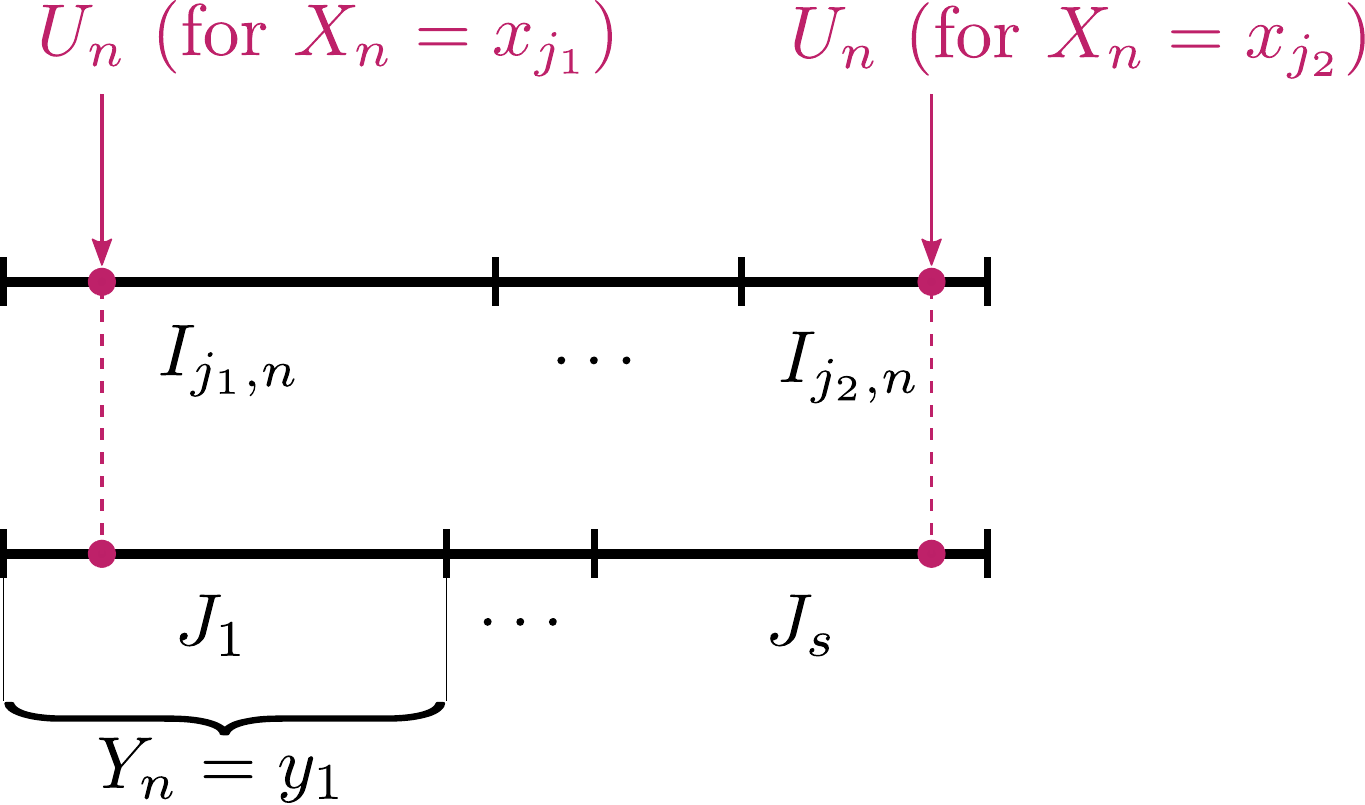}
  \caption{Case 1 ($J_1\cap I_{j_2,n}=\emptyset$)}
  \label{fig:case1}
\end{center}
\end{figure}

\paragraph{Case 1:} $J_1\cap I_{j_2,n}=\emptyset$ (see Figure~\ref{fig:case1}). Then we have
 \begin{equation}\label{a:raz}
  \PP\bigl(Y_n=y_1 \,|\, (X_m)_{m\le n-1}=\xi,X_n=x_{j_2}\bigr)=0,
 \end{equation}
 whereas
 \begin{equation}\label{a:b}
  \PP\bigl(Y_n=y_1 \,|\, (X_m)_{m\le n-1}=\xi,X_n=x_{j_1}\bigr)>0.
 \end{equation}
 Moreover, notice that~\eqref{a:raz} is equivalent to
 \begin{equation}\label{a:bb}
  \PP\bigl(Y_n>y_1 \,|\, (X_m)_{m\le n-1}=\xi,X_n=x_{j_2}\bigr)=1,
 \end{equation}
 It follows from~\eqref{a:b} and~\eqref{a:bb} that there exists a $V_n$-measurable event $A$ of positive probability such that, on $A$,
 \[ Y_n\bigl(V_n,x_{j_2},\xi\bigr) > y_1 = Y_n\bigl(V_n,x_{j_1},\xi\bigr). \]
 Remembering~\eqref{eq:Yn}, we get
 \begin{equation}
  \label{eq:strict}
  \E \bigl[ Y_n \,|\, (X_m)_{m\le n-1}=\xi , X_n=x_{j_2}\bigr] >  \E \bigl[ Y_n \,|\, (X_m)_{m\le n-1}=\xi , X_n=x_{j_1}\bigr].
 \end{equation}

\begin{figure}
 \begin{center}
  \includegraphics[width=0.6\linewidth]{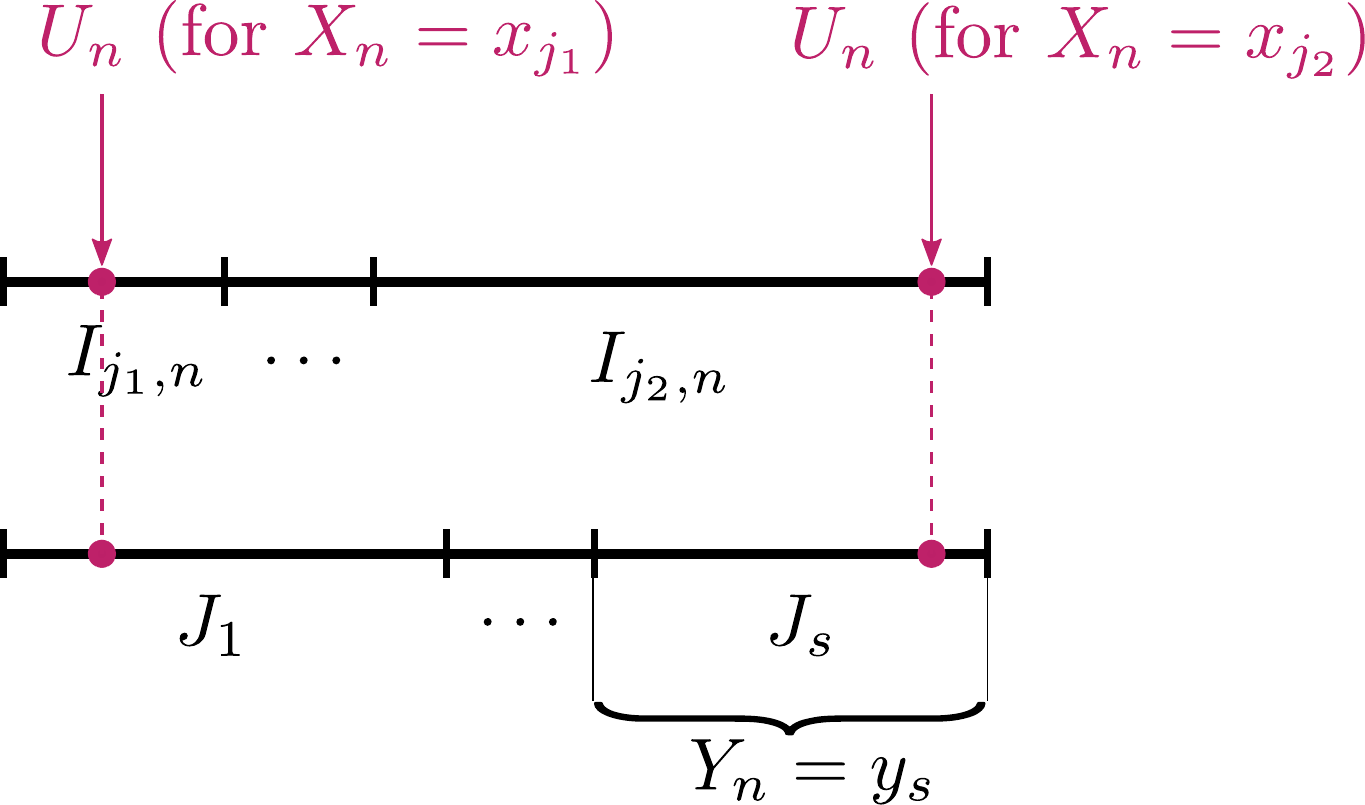}
  \caption{Case 2 ($J_1\cap I_{j_2,n}\neq\emptyset$)}
  \label{fig:case2}
\end{center}
\end{figure}

\paragraph{Case 2:} $J_1\cap I_{j_2,n}\neq\emptyset$ (see Figure~\ref{fig:case2}). Then $I_{j_1,n}\subset J_1$ and $I_{j_1,n}\cap J_s=\emptyset$. It follows that
 \[
  \PP\bigl(Y_n=y_s \,|\, (X_m)_{m\le n-1}=\xi,X_n=x_{j_1}\bigr)=0,
 \]
 whereas
 \[
  \PP\bigl(Y_n=y_s \,|\, (X_m)_{m\le n-1}=\xi,X_n=x_{j_2}\bigr)>0.
 \]
 In this case, we get a $V_n$-measurable event $A$ of positive probability such that, on $A$,
 \[ Y_n\bigl(V_n,x_{j_2},\xi\bigr) = y_s > Y_n\bigl(V_n,x_{j_1},\xi\bigr), \]
 and as before we conclude that~\eqref{eq:strict} holds.

 Now, since~\eqref{eq:strict} always holds, and since
 \[ \E \bigl[ Y_n \,|\, (X_m)_{m\le n-1}=\xi\bigr] = 0 =\sum_{j=1}^rP_{j,n}\,  \E \bigl[ Y_n \,|\, (X_m)_{m\le n-1}=\xi , X_n=x_{j}\bigr], \]
 we deduce that
 \[ \E \bigl[ Y_n \,|\, (X_m)_{m\le n-1}=\xi , X_n=x_{j_2}\bigr]>0. \]
 It follows that in the sum~\eqref{eq:sum}, at least the term corresponding to $j=j_2$ is positive, and this yields
 \[ \E \bigl[ X_n Y_n \,|\, (X_m)_{m\le n-1}=\xi\bigr] > 0. \]
\end{proof}

\pagebreak[4]

\begin{center}
\huge {\bf Appendix}
\end{center}

\appendix

\section{From averaged double to averaged multiple correlations}\label{a:drugi}
This section follows some arguments from \cite{Ma-Ra-Ta}.

\begin{Remark}\label{r:tosamo}
In the proof below we will use the following standard fact: let $(x(n))$ be a sequence of complex number bounded by 1. Then
$$
\sum_{m\leq M}|x(m)|=o(M)
$$
is equivalent to
$$
\sum_{m\leq M}|x(m)|^2=o(M).
$$
The little  ``o'' is uniform with respect to $M$.~\footnote{
If $\vep:=\frac1M\sum_{m\leq M}|x(m)|^2$ then by Markov's inequality
$$
\frac1M|\{m\leq M:\: |c_m|^2\geq \vep^{1/2}\}|\leq\frac1{\vep^{1/2}}\cdot\vep=\vep^{1/2}$$
and then
$$
\frac1M\sum_{m\leq M}|x(m)|=
\frac1M\sum_{m\leq M, |x(m)|\geq e^{1/4}}|x(m)|+\frac1M\sum_{m\leq M, |x(m)|<\vep^{1/4}}|x(m)|\leq \vep^{1/2}+\vep^{1/4}.$$}
\end{Remark}

We have the following general lemma:
\begin{Lemma} Let $(N_\ell)_{\ell \in \N}$ be a sequence of natural numbers. For $k\in\mathbb{N}$  let $a,b_1,\ldots b_k\colon\mathbb{N} \to \mathbb{C}$ be sequences bounded by $1$. Assume that $a$ satisfies
\begin{equation}\label{eq:H1}
\lim_{H\to \infty}\frac{1}{H} \sum_{h\leq H}\lim_{\ell\to \infty}\frac{1}{N_\ell}\Big|\sum_{n\leq N_\ell}a(n)a(n+h)\Big|=0.
\end{equation}
Then
\begin{equation}\label{teza}
\lim_{H\to \infty}\frac{1}{H^k} \sum_{h_1,\ldots, h_k\leq H}\lim_{\ell\to \infty}\frac{1}{N_\ell}\Big|\sum_{n\leq N_\ell}a(n)\prod_{i=1}^kb_i(n+h_i)\Big|=0.
\end{equation}
\end{Lemma}
\begin{proof}
Notice first that~\eqref{teza} can be rewritten as the following: for every $\vep>0$, there exists $H_\vep$ such that for $H>H_\vep$ and all $\ell$ sufficiently large (depending on $H$), we have
\[
A:=\frac{1}{H^k}\sum_{h_1,\ldots, h_k\leq H}\frac{1}{N_\ell}\Big|\sum_{n\leq N_\ell}a(n)\prod_{i=1}^kb_i(n+h_i) \Big|<\vep.
\]

Now, notice that for any $H,N_\ell,H'$ and any $h'\leq H'$, by shifting the summation over $n\leq N_\ell$ by $h'$ (for every fixed choice of $h_1,\ldots h_k$), we have
$$
\sum_{h_1,\ldots, h_k\leq H}\Big| \sum_{n\leq N_\ell}a(n)\prod_{i=1}^kb_i(n+h_i)\Big|=
$$
$$
\sum_{h_1,\ldots, h_k\leq H}\Big|\sum_{h'\leq n\leq N_\ell+h'}a(n)\prod_{i=1}^kb_i(n+h_i)\Big|+{\rm O}(H^k\cdot h')=
$$
$$
\sum_{h_1,\ldots, h_k\leq H}\Big|\sum_{n\leq N_\ell}a(n+h')\prod_{i=1}^kb_i(n+h_i+h')\Big|+{\rm O}(H^k\cdot h')
$$
and ${\rm O}(H^k\cdot h')={\rm O}(H^k\cdot H')$.
Notice that as $h_i$ is taken from $[0,H]$, then $h_i+h'$ is taken from $[h',H+h']$ (which is a small shift of $[0,H]$ if $h'$ is much smaller than $H$). So putting $h'$ to the summation over $h_i$, we get
$$
\sum_{h_1,\ldots, h_k\leq H}\Big|\sum_{n\leq N_\ell}a(n+h')\prod_{i=1}^kb_i(n+h_i+h')\Big|=
$$
$$
\sum_{h'\leq h_1,\ldots, h_k\leq H+h'}\Big|\sum_{n\leq N_\ell}a(n+h')\prod_{i=1}^kb_i(n+h_i)\Big|=
$$
$$
\sum_{h_1,\ldots, h_k\leq H}\Big|\sum_{n\leq N_\ell}a(n+h')\prod_{i=1}^kb_i(n+h_i)\Big|+ {\rm O}\left((h' H^{k-1}+(h')^2H^{k-2}+\ldots+(h')^k)N_\ell\right)
$$
and ${\rm O}\left((h' H^{k-1}+(h')^2H^{k-2}+\ldots+(h')^k)N\right)={\rm O}((H')^kH^{k-1}N_\ell)$. Putting the two displayed equations together we get that for every $h'\leq H'$,
\[
A=\frac{1}{H^k}\sum_{h_1,\ldots, h_k\leq H}\frac{1}{N_\ell}\Big|\sum_{n\leq N_\ell}a(n+h')\prod_{i=1}^kb_i(n+h_i) \Big|+{\rm O}\Big(\frac{H'}{N_\ell}\Big)+{\rm O}\Big(\frac{(H')^k}{H} \Big).
\]
Averaging the above equation over all $h'\leq H'$, we get that
$$
A=\frac{1}{H^k}\sum_{h_1,\ldots, h_k\leq H}\frac{1}{H'}\sum_{h'\leq H'}\frac{1}{N_\ell}\Big|\sum_{n\leq N_\ell}a(n+h')G(n)\Big|+{\rm O}\Big(\frac{H'}{N_\ell}\Big)+{\rm O}\Big(\frac{(H')^k}{H} \Big),
$$
where $G(n)=\prod_{i=1}^kb_i(n+h_i)$.

We will now estimate
\begin{multline}\label{B}
\frac{1}{H'}\sum_{h'\leq H'}\frac{1}{N_\ell^2}\Big|\sum_{n\leq N_\ell}a(n+h')G(n) \Big|^2=\\
\frac{1}{N_\ell^2}\sum_{n,n'\leq N_\ell}\Big(\frac{1}{H'}\sum_{h'\leq H'}a(n+h')\overline{a(n'+h')}\Big)G(n)\overline{G(n')},
\end{multline}
which will be easier to handle than the above expression for $A$ (and then use Remark~\ref{r:tosamo} to get rid of the squares). Clearly, to obtain an upper bound for~\eqref{B}, it suffices to obtain an upper bound for
\begin{equation}\label{E}
\frac{1}{N_\ell^2}\sum_{n,n'\leq N_\ell}\frac{1}{H'}\Big|\sum_{h'\leq H'}a(n+h')\overline{a(n'+h')} \Big|.
\end{equation}
Again, it will be easier to deal with
\begin{equation}\label{C}
\frac{1}{N_\ell^2}\sum_{n,n'\leq N_\ell}\frac{1}{H'^2}\Big|\sum_{h'\leq H'}a(n+h')\overline{a(n'+h')} \Big|^2
\end{equation}
(and use Remark~\ref{r:tosamo} to get rid of the squares). Expanding the square again we get that~\eqref{C} is equal to
$$
\frac{1}{N_\ell^2}
\sum_{n,n'\leq N_\ell}\frac{1}{H'^2}\sum_{h',h''\leq H'}a(n+h')\overline{a(n'+h')}\overline{a(n+h'')\cdot \overline{a(n'+h'')}}=
$$
$$
\frac{1}{N_\ell^2}\sum_{n,n'\leq N_\ell}\frac{1}{H'^2}\sum_{h',h''\leq H'}a(n+h')\overline{a(n+h'')}a(n'+h'')\overline{a(n'+h')}.
$$
The sum in the last term by exchanging the order of summation is equal to
$$
\frac{1}{H'^2}\sum_{h',h''\leq H'}\Big|\frac{1}{N_\ell}\sum_{n\leq N_\ell}a(n+h')\overline{a(n+h'')}\Big|^2=
$$$$
\frac{1}{H'^2}\sum_{h',h''\leq H'}\Big|\frac{1}{N_\ell}\sum_{n\leq N_\ell}a(n)\overline{a(n+h''-h')}\Big|^2+ {\rm O}(\frac{H'^2}{N_\ell^2}).
$$
Finally, grouping according to $h=h''-h'$, we get that that the above is equal to
$$
\frac{1}{H'^2}\sum_{|h|\leq H'}|H'-h|\cdot \Big|\frac{1}{N_\ell}\sum_{n\leq N_\ell}a(n)\overline{a(n+h)}\Big|^2+{\rm O}(\frac{H'^2}{N_\ell^2})\leq $$$$
\frac{1}{H'}\cdot \sum_{|h|\leq H'}\Big|\frac{1}{N_\ell}\sum_{n\leq N_\ell}a(n)\overline{a(n+h)}\Big|^2+{\rm O}(\frac{H'^2}{N_\ell^2}).
$$
That is, the expression from~\eqref{C} equals
\begin{equation}\label{D}
\frac{1}{H'}\sum_{|h|\leq H'}\left|\frac{\sum_{n\leq N_\ell}a(n)\overline{a(n+h)}}{N_\ell} \right|^2 +{\rm O}\left(\frac{(H')^2}{N_\ell^2} \right).
\end{equation}

Now, by the assumption of our lemma, it follows that
\[
\frac{1}{H'}\sum_{|h|\leq H'}\left|\frac{\sum_{n\leq N_\ell}a(n)\overline{a(n+h)}}{N_\ell} \right|={\rm o}(1),
\]
which, by Remark~\ref{r:tosamo}, is equivalent to
\[
\frac{1}{H'}\sum_{|h|\leq H'}\left|\frac{\sum_{n\leq N_\ell}a(n)\overline{a(n+h)}}{N_\ell} \right|^2={\rm o}(1).
\]
Therefore, ~\eqref{D} (and, thus, also~\eqref{C}) is of the order of ${\rm o}(1)+{\rm O}\left(\frac{(H')^2}{N_\ell^2} \right)$. Using again Remark~\ref{r:tosamo}, we conclude that also~\eqref{E} is of the same order. It follows immediately that also the order of~\eqref{B} is the same.

Thus, we have proved that
\[
A={\rm o}\left(1\right)+{\rm O}\left(\frac{(H')^2}{N_\ell^2} \right)+{\rm O}\left(\frac{H'}{N_\ell} \right)+{\rm O}\left(\frac{(H')^k}{H}\right).
\]
\end{proof}

\vspace{2ex}

\noindent
{\bf Acknowledgments:}
We would like to thank  Tomasz Downarowicz, Nikos Frantzikinakis and Krzysztof Fr\c{a}czek for useful discussions on the paper.
Research of the second and third authors supported by Narodowe Centrum Nauki grant UMO-2019/33/B/ST1/00364.

\vspace{1ex}

\bigskip

\noindent
Department of Mathematics, The Maryland University\\
adkanigowski@gmail.com\\
\smallskip

\noindent
Faculty of Mathematics and Computer Science\\
Nicolaus Copernicus University, Toruń, Poland\\
joasiak@mat.umk.pl, mlem@mat.umk.pl\\
\smallskip

\noindent
Laboratoire de Math\'ematiques Rapha\"el Salem, CNRS -- Université de Rouen Normandie\\
Avenue de l’Universit\'e – 76801 Saint \'Etienne du Rouvray, France\\
Thierry.de-la-Rue@univ-rouen.fr

\end{document}